\documentclass{amsart}
\usepackage{amsmath,amssymb,amsthm,amsfonts,mathdots,amscd,bbm}
\usepackage[numeric]{amsrefs}
\input colordvi

\swapnumbers

\newtheorem{thm}[equation]{Theorem}
\newtheorem{cor}[equation]{Corollary}
\newtheorem{lem}[equation]{Lemma}

\newtheorem{prop}[equation]{Proposition}
\newtheorem{definition}[equation]{Definition}

\newcommand{\thmref}[1]{theorem~\ref{#1}}

\newcommand{\lemref}[1]{lemma~\ref{#1}}
\newcommand{\secref}[1]{section~\ref{#1}}

\DeclareMathOperator{\sg}{sgn}

\DeclareMathOperator{\tr}{tr}

\numberwithin{equation}{section}

\renewcommand\a{\alpha}

\newcommand\g{\gamma}
\renewcommand\d{\delta}
\newcommand\e{\varepsilon}
\renewcommand\l{\lambda}
\renewcommand\L{\mathcal L}

\newcommand\G{\Gamma}
\newcommand\f{\frac}
\newcommand\smallf[2]{{\textstyle{\frac{#1}{#2}}}}
\newcommand\srel[2]{\begin{smallmatrix} {#1} \\ {#2} \end{smallmatrix}}

\newcommand{\Z}{{\mathbb{Z}}}
\newcommand{\R}{{\mathbb{R}}}
\newcommand{\RP}{{\mathbb{RP}}}
\newcommand{\C}{{\mathbb{C}}}
\newcommand{\A}{{\mathbb{A}}}
\newcommand{\Q}{{\mathbb{Q}}}

\newcommand{\Sch}{{\mathcal{S}}}
\newcommand\re{\text{Re~}}

\renewcommand\Re{\text{Re~}}

\renewcommand\i{^{-1}}
\renewcommand\({\left(}
\renewcommand\){\right)}
\newcommand{\ttwo}[4]{
\(\begin{smallmatrix}{#1} & {#2}
\\ {#3} & {#4} \end{smallmatrix}\)}

\newcommand{\sgn}{\operatorname{sgn}}

\newcommand{\gobble}[1]{}
  \newcommand{\rangeref}[2]{%
    \ref{#1}--\afterassignment\gobble\fam 0\ref{#2}%
  }
\makeatletter
\def\imod#1{\allowbreak\mkern5mu({\operator@font mod}\,#1)}
\makeatother

\begin{document}

\title{The archimedean theory of the Exterior Square $L$-functions over $\Q$}

\author{Stephen D. Miller}
\address{Department of Mathematics, Rutgers University, Piscataway, NJ 08854-8019}
\email{miller@math.rutgers.edu}
\thanks{Partially supported
by NSF grant DMS-0901594 and an Alfred P. Sloan Foundation Fellowship.}

\author{Wilfried Schmid}
\address{Department of Mathematics, Harvard University, Cambridge, MA 02138}
\email{schmid@math.harvard.edu}
\thanks{Partially supported by DARPA grant HR0011-04-1-0031 and
NSF grant DMS-0500922}

\date{September 12, 2011}

\dedicatory{In memory of  Joseph Shalika}


\maketitle

\begin{abstract}
The analytic properties of automorphic $L$-functions have historically been obtained either through integral
representations (the ``Rankin-Selberg method''), or  properties of the Fourier expansions of Eisenstein series (the
``Langlands-Shahidi method''). We introduce a  method based on pairings of automorphic distributions, that appears
to be applicable to a wide variety of $L$-functions, including all which have integral representations.  In some
sense our method could be considered a completion of the Rankin-Selberg method because of its common features.
 We consider a particular but representative example, the exterior square $L$-functions on $GL(n)$, by constructing a pairing which we compute  as  a product of this $L$-function times an explicit ratio of Gamma functions.  We use this to deduce that  exterior square $L$-functions, when multiplied by the Gamma factors predicted by Langlands, are holomorphic on $\C-\{0,1\}$ with at most simple poles at 0 and 1, proving a conjecture of Langlands which has not been obtained by the existing two methods.
\end{abstract}

\section{Introduction}
\label{introsec}

Let $\pi=\otimes_{p\le \infty} \pi_p$ be a cuspidal automorphic representation of $GL(n)$ over $\Q$.  The adelic
representation $\pi$ is composed of $\pi_\infty$,  an archimedean representation of $GL(n,\R)$, and nonarchimedean
representations $\pi_p$ of $GL(n,\Q_p)$ for each prime $p$.  For all places $p< \infty$ outside of a finite set
$S$, $\pi_p$ is unramified and parameterized by
 principal series parameters $\{\a_{p,j}\}$.  Letting $A_p \in
GL(n,\C)$ denote the diagonal matrix $\operatorname{diag}(\a_{p,1},\ldots,\a_{p,n})$ and $\rho$ a finite
dimensional representation of $GL(n,\C)$,  Langlands \cite{eulerproducts} predicts that his $L$-functions
\begin{equation}\label{lfns}
   L^S(s,\pi,\rho) \ \ = \ \
     \prod_{p \,\notin\, S}\det\(1-\rho(A_p)\, p^{-s}\)\i \, , \ \ \  \ \Re{s} \ \gg \ 0 \,,
\end{equation}
possess certain analytic properties similar to those held by the Riemann $\zeta$-function.  In particular, he
posits the existence of factors $L_p(s,\pi,\rho)$ for each place $p$ ---   agreeing with the factor in (\ref{lfns})
when $p\notin S$ ---  such that  completed $L$-function
\begin{equation}\label{lambdafns}
    \Lambda(s,\pi, \rho) \ \ = \ \ \prod_{p
    \, \le \,
    \infty}L_p(s,\pi,\rho)\, ,  \ \ \  \ \Re{s} \ \gg \ 0 \,,
\end{equation}
has an analytic continuation to $\C-\{0,1\}$.    The completed $L$-function may have poles at $s=0$ or $1$, as it
does in the case of the completion of the Riemann $\zeta$-function $\xi(s)=\pi^{-s/2}\G(\f s2)\zeta(s)$.  Since
$L$-functions sometimes factor as products of $\zeta(s)$ times other factors, the best statement one can hope for
is the following property we call ``full holomorphy'':
\begin{definition}\label{def:fullholom}
A partial product of $L_p(s,\pi,\rho)$ from (\ref{lambdafns}) over a subset of places $\{p\le \infty\}$ is said to
be ``fully holomorphic'' if it has meromorphic continuation to $s\in \C$, with no poles outside of $\{s=0,1\}$.
\end{definition}

In addition to being fully holomorphic,  Langlands also conjectures that $\Lambda(s,\pi,\rho)$ satisfies a
functional equation of the form $\Lambda(1-s,\pi,\rho) =  \omega_{\hbox{root}} \, q^s\,
\Lambda(s,\tilde{\pi},\rho)$   for some $\omega_{\hbox{root}}\in \C^*$ and $q\in \Z_{> 0}$.  Here as
elsewhere in the paper $\tilde{\pi}$ refers to the automorphic representation contragredient to $\pi$; its
$L$-function is $\Lambda(s,\tilde\pi,\rho)=\overline{\Lambda(\bar{s},\pi,\rho)}$.  Langlands himself
\cite{langlandsrealgroups} provided a formula for $L_\infty(s,\pi,\rho)$,  while Harris and Taylor
\cite{harristaylor} and Henniart \cite{henniart}  did likewise for $L_p(s,\pi,\rho)$ when $p<\infty$.

Though in general almost nothing is known about this conjecture, several striking cases have been established.  For
example, the conjecture was quickly established by Godement and Jacquet \cite{goja} for the standard representation
$Stan$ of $GL(n,\C)$. In general, Langlands further conjectures that his $L$-functions are always in fact equal to
the $L$-functions $L(s,\Pi,Stan)$ of automorphic forms on $GL(d,\Q)\backslash GL(d,\A)$, where $d$ is the dimension
of $\rho$; the conjectured analytic properties would of course then
 be consequences of \cite{goja}.

Moreover, Langlands' general conjectures \cite{langlandsdc} about the holomorphy of his $L$-functions may be stated
not just for $GL(n)$ over $\Q$, but for general reductive groups and global fields. Two well known techniques exist
for approaching certain cases of this conjecture: the method of integral representations (the Rankin-Selberg
method), and the method of Fourier coefficients  of Eisenstein series (the Langlands-Shahidi method).  Each has had
stunning success in a number of examples, yet neither seems capable of treating the full picture. In particular,
typically even in cases where both methods are applicable, only partial results are known. The Rankin-Selberg
method usually gives a fully holomorphic expression which is related, but not equal to $\Lambda(s,\pi,\rho)$; the
Langlands-Shahidi method treats exactly $\Lambda(s,\pi,\rho)$ for many examples of $\pi$ and $\rho$, but often
cannot prove the full holomorphy because the Eisenstein series it captures $\Lambda(s,\pi,\rho)$ from might have
unwanted poles. The Langlands-Shahidi method also produces the  functional equation for $\Lambda(s,\pi,\rho)$.

In this paper we introduce a   new   technique to obtain the analytic properties of $L$-functions using automorphic
distributions, which can bridge this gap and prove new instances of Langlands' conjectures for
$\Lambda(s,\pi,\rho)$ --- instances that had not been obtained using the combined strength of known results.
 We shall consider the
$\f{n(n-1)}{2}$-dimensional exterior square representation of $GL(n,\C)$, tensored with a Dirichlet character
$\chi$.  With the notation as earlier,
\begin{equation}\label{extsqlocdef}
    L_p(s,\pi,Ext^2\otimes\chi) \  \ = \ \ \prod_{1\le j < k \le n}
    (1\,-\,\a_{p,j}\,\a_{p,k}\,\chi(p)\,p^{-s})\i
\end{equation}
at primes $p$ for which both $\pi$ and $\chi$ are unramified. The factor $L_\infty(s,\pi,Ext^2\otimes\chi)$ is a
certain product of Gamma functions which we will describe  in \secref{sec:functionalequation}.  We continue with
the notation   $L^S(s,\pi,\rho)=\prod_{p\notin S} L_p(s,\pi,\rho)$,   whether $S$ includes the place at infinity or
not. Our main result is the following:
\begin{thm}\label{mainthmwiths}
    Let $\pi$ be a cuspidal automorphic representation of $GL(n)$ over $\Q$,  $\chi$ a Dirichlet character,  and $S$ any finite subset of places including the ramified nonarchimedean primes for $\pi$ and $\chi$
    (in particular $S$  need not include
    the archimedean place).  Then
   $L^S(s,\pi,Ext^2\otimes\chi)$   is fully holomorphic  in the sense of definition~\ref{def:fullholom}, with  at most simple poles at $s=0$ and 1.
\end{thm}
The main contribution of our method is that  we do not insist that $\infty\in S$, as one frequently does using the
Rankin-Selberg method; we hope it will be applicable to other $L$-functions as well. The special case when
$S=\emptyset$ and $\chi$ is equal to the trivial character corresponds classically to automorphic forms on
$GL(n,\Z)\backslash GL(n,\R)$, i.e., ``full level'' ones.  In this case   $L^S(s,\pi,Ext^2)$   equals the completed
$L$-function $\Lambda(s,\pi,Ext^2)$ itself and the theorem reads
\begin{cor}\label{mainthm}
If $\pi=\otimes_{p\le \infty}\pi_p$ is a cuspidal automorphic representation of $GL(n)$ over $\Q$ which is
unramified at each $p<\infty$ (so that $\pi$ corresponds to a cusp   form on the full level quotient
$GL(n,\Z)\backslash GL(n,\R)$), then $\Lambda(s,\pi,Ext^2)$ is fully holomorphic,    with at most simple poles at
$s=0$ and 1.
\end{cor}

Our method also gives a functional equation, which we give a complete account  of in this paper for full level
forms. We did not state this as part of the theorems here because it was previously known
 as part of a general result of Shahidi \cite{shahidi} on
the Langlands-Shahidi method. Shahidi furthermore proved that $\Lambda(s,\pi,Ext^2\otimes\chi)$ has a meromorphic
continuation to $\C$ with at most a finite number of   simple   poles, for any automorphic cusp form $\pi$ on
$GL(n)$. Furthermore, several cases of our theorem were known by earlier results, some of which are stronger
because they treat number fields and ramified nonarchimedean cases as well (where our technique works in what is so
far only a limited number of cases). Kim \cites{kimgl4,kimextsq} proved that $\Lambda(s,\pi,Ext^2)$ is fully
holomorphic   with at most simple poles   if either $n$ is odd, or if $n$ is even and $\pi$ is not twist-equivalent
to $\tilde\pi$. The Rankin-Selberg method also has provided earlier partial results.  Jacquet-Shalika
\cite{jsextsq} and Bump-Friedberg \cite{bumpfriedberg} found different integral representations for
$\Lambda(s,\pi,Ext^2\otimes\chi)$, which for example have been used to characterize when it  has poles at $s=0$ and
1.  It is now known that $\Lambda(s,\pi,Ext^2)$ has these poles if and only if $\pi$ is a functorial lift of a
generic   cuspidal automorphic representation of $SO(2n+1)$ \cite{jiang}*{Theorem 2.2(3)}.
 Stade \cite{stade} completed the archimedean theory of the Bump-Friedberg integral
 representation in the case that $\pi_\infty$ is a
 spherical principal series representation.  This
 provides the full holomorphy   with at most simple poles, and also the   functional equation of
$\Lambda(s,\pi,Ext^2)$ when $\pi_p$ is spherical for all $p\le \infty$.

Because Kim's result \cite{kimextsq} covers our result when $n$ is odd, we have chosen to restrict the content of
this paper to the case where $n$ is  even; however, it is not difficult to extend our method to cover all $n$.  Our
results extend from $\Q$ to a general number field; we hope to return to this as well as to adapt our archimedean
methods to the nonarchimedean setting in a future paper.

Our technique involves pairings of automorphic distributions of cusp forms and Eisenstein series, which are the
topic of section~\ref{sec:automorphicdistributions}. It could be thought of as a completion of the Rankin-Selberg
method, because it is heavily influenced by the Jacquet-Shalika integral representation \cite{jsextsq}.  One of the
key differences at this stage is that the analytic continuation employs a  mechanism different from the one  used
for integral representations.  The unfolding computation for our pairing is presented in section~\ref{unfolding}.
This identifies the pairing with an Euler product for $L(s,\pi,Ext^2\otimes\chi)$ times our archimedean integral,
which is a pairing of ``Whittaker distributions''.   This latter integral is a key difference between our method
and the Rankin-Selberg method.  While the archimedean Rankin-Selberg integrals are notoriously difficult to
compute,   this integral is explicitly computed using a matrix decomposition in sections~\ref{slicingsec} and
\ref{sec:localintegrals}.  The matrix decomposition is thus in some sense the crux of the paper.   This gives the
full holomorphy  with at most simple poles   of $L(s,\pi,Ext^2\otimes\chi)$ times an explicit ratio of Gamma
factors, which is not equal to $L_\infty(s,\pi,Ext^2\otimes\chi)$.  The relation between that factor and
$L_\infty(s,\pi,Ext^2\otimes\chi)$ is calculated in section~\ref{sec:functionalequation}, and the full holomorphy
proved in section~\ref{sec:fullholomorphy}.

 The full details of our methodology were illustrated for $GL(4)$ in our earlier
 paper \cite{korea}.  That case avoids many  difficulties and computations that are present in this paper.

It is a pleasure to acknowledge Bill Casselman, James Cogdell, David Ginzburg, Roe Goodman, Herve Jacquet, Henry
Kim, Erez Lapid, Ilya Piatetski-Shapiro, Ze'ev Rudnick,  Peter Sarnak, Freydoon Shahidi, David Soudry, Akshay
Venkatesh, and Nolan Wallach for their helpful advice and discussions.

\section{Automorphic Distributions}\label{sec:automorphicdistributions}

 In this section we summarize the
notion of automorphic distribution, but only to the extent needed for our purposes. Further details can be found in
\cite{mirabolic}*{\S 2-\S 5}.

\subsection{Cuspidal Automorphic Distributions for $GL(n)$}
\label{autdistsec}

Let $G$ equal the algebraic group $GL(n)$, $B_{-}=\{$nonsingular lower triangular matrices$\}$ a maximal solvable
subgroup, $N=\{$unit upper triangular matrices$\}$ a maximal unipotent subgroup, and  $Z=\{$non-zero scalar
multiples of the identity$\}$  its center. The flag variety of $G(\R)$,
\begin{equation}
\label{ps3} X\ = \ G(\R)/B_{-}(\R)\,,
\end{equation}
is compact, and its open dense $N(\R)$-orbit through the base point $eB_{-}(\R)\subset X$ is its open Schubert
cell, which can be identified with $N(\R)$:
\begin{equation}
\label{ps4} N(\R) \ \ \simeq \ \ N(\R)\cdot eB_{-}(\R)\ \ \hookrightarrow \ \ X\,.
\end{equation}

 For any  $\l \in \C^n$ and $\d \in (\Z/2\Z)^n$,  define the character $\chi_{\l,\d} : B_{-}(\R)  \longrightarrow  \C^*$ by the formula
\begin{equation}
\label{ps5}
\chi_{\l,\d} \left(\begin{smallmatrix}
a_1 & 0 & {\textstyle\dots} & 0 \\
{ }_{\scriptstyle *} & a_2 & {\textstyle\dots}   & 0\\
\vdots & \vdots &  \ddots    & \vdots\\
{ }_{\scriptstyle *} & { }_{\scriptstyle *} &{\textstyle\dots} &
a_n
\end{smallmatrix}\right)\ \ = \ \ {\prod}_{j=1}^n  \sgn (a_j)^{\d_j} |a_j|^{\l_j}\,.
\end{equation}
The principal series $V_{\l,\d}$ is the representation induced from   $\chi_{\l-\rho,\d}$ from $B_{-}(\R)$ to
$G(\R)$, where
\begin{equation}
\label{ps6} \rho\ \ = \ \ \left(\textstyle
\frac{n-1}2,\,\textstyle \frac{n-3}2,\, \dots,\, \textstyle
\frac{1-n}2 \right)\,.
\end{equation}
In particular, its space of smooth vectors consists of smooth sections of a line bundle ${\mathcal
L}_{\l-\rho,\d}\to X$,
\begin{equation}
\label{ps7}
\gathered
 V_{\l,\d}^\infty \ = \ C^\infty(X, \L_{\l-\rho,\d}) \ \ \simeq\qquad
\qquad\qquad\qquad\qquad\qquad\qquad
\qquad\qquad \qquad   \\ \{f \in  C^\infty(G(\R))
\mid f(gb)=\chi_{\l-\rho,\d}(b^{-1})f(g)\ \, \text{for}\ \, g\in G(\R),\,
b\in B_{-}(\R)\}.
\endgathered
\end{equation}
The principal series $V_{\l,\d}$ is not necessarily unitary, but is naturally dual to the principal series
$V_{-\l,\d}$ via integration over the flag variety $X$.
 Analogously, the space of distribution vectors for $V_{\l,\d}$ consists of distribution sections of the same line bundle:
\begin{equation}
\label{ps9}
\gathered
V_{\l,\d}^{-\infty} \  = \ C^{-\infty}(X, \L_{\l-\rho,\d}) \ \simeq \qquad
\qquad\qquad\qquad\qquad\qquad\qquad \qquad\qquad \ \ \ \
\\
 \{\sigma \in  C^{-\infty}(G(\R)) \mid
\sigma(gb)=\chi_{\l-\rho,\d}(b^{-1})\sigma(g)\ \text{for}\ g \in  G(\R),
b \in  B_{-}(\R)\}.
\endgathered
\end{equation}
The restriction of the equivariant line bundle $\L_{\l-\rho,\d}\to X$ to the open Schubert cell (\ref{ps4}) is
canonically trivial, because $N(\R)\,\cap\, B_{-}(\R)=\{e\}$.  Its distribution\footnote{``Distribution'' for us
refers to ``generalized functions'', so that scalar valued distributions are dual to smooth measures, and include,
for example, continuous functions.} sections therefore become scalar, resulting in the identification
\begin{equation}
\label{ps10} C^{-\infty}(N(\R), \L_{\l-\rho,\d}) \ \simeq \ C^{-\infty}(N(\R))\,,
\end{equation}
which is $N(\R)$-invariant, of course.

Let $\A=\A_\Q=\R\times \A_f$ denote the adeles of $\Q$, where $\A_f$ denotes the finite adeles (i.e., the
restricted direct product of all $\Q_p$, $p<\infty$).  We will use the notation $a=\prod_{p\le \infty}a_p$ to
represent the respective components of an adele, and analogously by extension, use similar subscript notation for
adelic points in algebraic groups defined over $\Z$. By the usual convention,   $\Q$ is identified with its
diagonally embedded image in $\A$, as are the rational points of any algebraic group defined over $\Z$ considered
diagonally embedded in its adelic points. Suppose that $\pi=\otimes_{p\le \infty} \pi_p$ is a cuspidal automorphic
representation of $G(\A)$ with central character $\omega:\A^*\rightarrow \C^*$.  By tensoring with an appropriate
power of the determinant, we may assume that $\omega$ has finite order. By definition the  archimedean
representation $\pi_\infty$ occurs automorphically in  $L^2_{\omega_\infty}(\Gamma\backslash G(\R))$, where
$\omega_\infty$ is the archimedean component of $\omega$ and $\G$ is some discrete subgroup of $G(\R)$.  This
embedding of $\pi_\infty$ into $L^2_{\omega_\infty}(\Gamma\backslash G(\R))$ maps smooth vectors to smooth
functions on $G$, continuously with respect to the intrinsic topology on the space of smooth vectors for
$\pi_\infty$ and the $C^\infty$ topology on $C^\infty(\G\backslash G)$. Evaluation of smooth $\G$-invariant
functions at the identity is continuous with respect to the $C^\infty$ topology, of course. In this way, the
embedding of $\pi_\infty$ into $L^2_{\omega_\infty}(\Gamma\backslash G(\R))$ determines an automorphic distribution
-- i.e., a $\G$-invariant continuous linear functional $\tau$ on the space of smooth vectors for $\pi_\infty$\,;\,\
to any smooth vector $v$, $\tau$ associates the value at the identity of the smooth $\G$-invariant function that
corresponds to $v$. The automorphic distribution $\tau$ completely determines the embedding of $\pi_\infty$ into
$L^2_{\omega_\infty}(\Gamma\backslash G(\R))$ \cites{voronoi,korea}. Also, as we shall explain shortly, the finite
adeles act on the automorphic distribution $\tau$. In effect, $\tau$ completely encodes the original automorphic
representation $\pi$.

We should alert the reader to the use of the symbol $\infty$ in two very different senses: as subscript, $\infty$
refers to the archimedean place, and as superscript to the degree of differentiability of functions and of vectors
in representation spaces. Both conventions are completely standard; deviating from the customary notation might be
more confusing than the dual meaning of $\infty$ as subscript and superscript.

As a continuous linear functional on the smooth vectors for $\pi_\infty$, the automorphic distribution $\tau$
should be regarded as a distribution vector for the dual representation $\pi_\infty'$.  Theorems of
Casselman \cite{Casselman:1980} and Casselman-Wallach \cites{Casselman:1989,Wallach:1983} imply that any
representation of $G(\R)$, in particular the Hilbert dual $\pi_\infty'$ of $\pi_\infty$, can be embedded into some
principal series representation.  Analogously, the space of distribution vectors\begin{footnote}{The space
of distribution vectors carries a natural topology; see \cite{rapiddecay}, for example}\end{footnote} for
$\pi_\infty'$ can then be realized as a closed subspace of the space of distribution vectors for the dual principal
series representation:
\begin{equation}
\label{autodist6}
(\pi_\infty')^{-\infty} \ \hookrightarrow\
V_{\l,\d}^{-\infty} \ \ \ \ \text{for some} \ \ \l\,\in\,\C^n\ \text{and} \ \d\,\in\,(\Z/2\Z)^n\,;
\end{equation}
  for details, see \cite{rapiddecay}. In general this embedding is not unique. When it is not, the
particular choice of embedding parameters $(\lambda,\delta)$ will not matter for most of the discussion. However,
in order to eliminate unwanted poles of the $L$-function associated to $\pi$ -- or equivalently, to $\tau$ -- we
shall eventually play off the various possible choices against each other.

The realization (\ref{autodist6}) allows us to view $\tau$ as an element of $C^{-\infty}(X,
\L_{\l-\rho,\d})^\G$. The automorphic distribution $\tau$ can be further realized as distributions on
$N(\R)$ by (\ref{ps10}):~even though distributions are not normally determined by restrictions to dense open sets,
this is the case for automorphic distributions because the $\G$-translates of $N(\R)$ cover all of $X$. Details of
this procedure can be found, for example, in \cite{mirabolic}*{\S 2}.

The adelic automorphic representation $\pi$ accounts for many simultaneous realizations of $\pi_\infty$,
corresponding to the restrictions to $G(\R)$ of functions on $G(\Q)\backslash G(\A)$ which correspond to pure
tensors for $\pi=\otimes_{p\le \infty}\pi_p$.
  We now indicate how to adelize the automorphic distribution of the previous paragraph, summarizing from \cite{mirabolic}*{\S5}.
 Each right translate of these functions by a fixed element of $G(\A_f)$ gives rise to an automorphic distribution, too; it is automorphic under a conjugate of $\G$.  This gives a map from $G(\A_f)$ to automorphic distributions on $G(\R)$, which we write as $\tau(g_\infty\times g_f)$.  We stress that this notation signifies that  for each fixed parameter $g_f$,  $\tau(g_\infty\times g_f)$ is an automorphic distribution in the $g_\infty\in G(\R)$ variable.  It satisfies the property that   for any $\g=\g_\infty\times \g_f\in G(\Q)$ (regarded as diagonally embedded in $G(\A)$,
the distributions $\tau(\g_\infty g_\infty \times \g_f g_f)$ and $\tau(g_\infty\times g_f)$ have equal integrals
against any smooth function of compact support in $g_\infty\in G(\R)$.  Since it is  invariant under multiplication
on the left by any element of $G(\Q)$,  we shall call the resulting object $\tau$ an {\em adelic
automorphic distribution} on $G(\A)$.

The remainder of this subsection concerns the  Fourier expansions of automorphic distributions.  Let $\psi_{+}$ be
the unique  additive character of $\A$ which is trivial on $\Q$ and  whose restriction to $\R$ is
$\psi_{+}(a_\infty)= e^{2\pi i a_\infty}$.  The composition $\psi_{+}\circ c$, where
\begin{equation}\label{ccharacter}
    c \ : \ (n_{ij}) \ \ \mapsto
    \ \ n_{1,2} \,+\, n_{2,3} \,+ \, \cdots
\end{equation}
is the sum of the entries just above the diagonal of a matrix, gives a nondegenerate character of $N(\Q)\backslash
N(\A)$.
 Recall the global Whittaker integrals for an automorphic representation $\pi$:
\begin{equation}\label{classwhit}
    W_{\phi}(g) \ \ = \ \ \int_{N(\Q)\backslash N(\A)} \phi(ng) \, \psi_{+}(c(n))\i\,dn\ , \ \ \ \, \phi \,\in\,\pi\,,
\end{equation}
where the  Haar measure $dn$ on $N(\A)$  is  normalized to give the (compact) quotient $N(\Q)\backslash N(\A)$
volume 1. The adelic Whittaker distribution is the analogous integral for $\tau$:
\begin{equation}\label{whittonta1}
    w(g) \ \ = \ \    w_\tau(g)   \ \ =  \ \ \int_{N(\Q)\backslash N(\A)} \tau(ng) \, \psi_{+}(c(n))\i\,dn\, ,
\end{equation}
or in terms of the left translation operator $\, \ell(n):\tau(g)\mapsto \tau(n\i g)$,
\begin{equation}\label{whittonta2}
    w \ \ = \ \   w_\tau  \ \ = \ \  \int_{ N(\A)/N(\Q)} \ell(n)  \tau \  \psi_{+}(c(n))\,dn\, .
\end{equation}
This integration is shown in \cite{mirabolic}*{\S5}  to define a   function of $g_f\in G(\A_f)$ with values in
$C^{-\infty}(G(\R))$.

The adelic automorphic distribution can be reconstructed as a sum of left translates of $w$ by the formula
\begin{equation}\label{tauareconstruct}
    \tau(g) \ \ = \ \ \sum_{\g\,\in\,N_{n-1}(\Q)\backslash GL(n-1,\Q)}w\(\ttwo{\g}{}{}{1} g\),
\end{equation}%
where $N_{n-1}$ is the subgroup of unit upper triangular matrices in $GL(n-1)$ (\cite{mirabolic}*{(5.14)}).
It is shown in \cite{mirabolic}*{Proposition~5.13} that there exists a realization of $\pi_\infty$ within the
automorphic representation $\pi$, and a corresponding adelic automorphic distribution $\tau$, such that the
integral (\ref{whittonta1})  factorizes as
\begin{equation}\label{factorizepure2}
      w(g) \ \ = \ \ w_\infty(g_\infty)\prod_{p< \infty}W_p(g_p) \, ,
\end{equation}
where the $W_p$ are taken to be arbitrary elements of the Whittaker model ${\mathcal W}_p$ of $\pi_p$ for finitely
many primes $p$, and the standard spherical Whittaker function (i.e., whose restriction to $G(\Z_p)$ is identically
1) at all other primes. The distribution $w_\infty\in V_{\l,\d}^{-\infty}$  satisfies the transformation property
$w_{\infty}(ng)=\psi_{+}(c(n))w_{\infty}(g)$ for all $n\in N(\R)$. It is a scalar multiple of a distribution
$w_{\l,\d}$ whose  restriction    to the open cell (\ref{ps4}) is the continuous function defined by the formula
\begin{equation}\label{btransform}
\gathered
    w_{\l,\d}\(\(\begin{smallmatrix}
      1 & x_{1} & \star & \star  & \star          \\
      & 1     & x_{2}& \star & \star \\
       & & \ddots  & \star & \star \\
       & & & 1 & x_{n-1} \\
       & & & & 1
    \end{smallmatrix}\)\(\begin{smallmatrix}
      a_1 &  & & &          \\
      \star & a_2      &  & &  \\
      \star  & \star & \ddots  & & \\
        \star &  \star &  \star & a_{n-1} & \\
       \star &  \star&  \star&  \star& a_n
    \end{smallmatrix}\)\) \qquad\qquad\qquad\qquad\qquad\qquad
    \\ \qquad\qquad\qquad\qquad = \ \ e(x_1+\cdots +
    x_{n-1})\,\prod_{j=1}^n
    |a_j|^{(n+1)/2-j-\l_j}\sgn(a_j)^{\d_j}\,,
\endgathered
\end{equation}
where
\begin{equation}
\label{ab7} e(z)\ =_{\text{def}}\ e^{2\pi i z}\,;
\end{equation}
that scalar can be normalized out, or shifted to one of the $W_p$, allowing us to assume $w_\infty=w_{\l,\d}$ for
the rest of the paper. Formula (\ref{factorizepure2}) is thus analogous to the classical Whittaker integral
(\ref{classwhit}) of a pure tensor $\phi$; the only difference is that the archimedean Whittaker function is
replaced by $w_{\l,\d}$.

The dual of the automorphic representation $\pi$ has an automorphic distribution that can be  constructed  directly
from $\tau$ using the contragredient map
\begin{equation}
\label{ab9} g\ \mapsto \ \widetilde g\,,\ \ \widetilde g\, = \,
w_{\text{long}}(g^t)^{-1}w_{\text{long}}^{-1}\,, \ \ \text{with}\
\ w_{\text{long}}\, = \, \left(
\begin{smallmatrix}
 & & & & 1 \\
 & & & \cdot & \\
 & & \cdot & & \\
 & \cdot &  & & \\
 1 & & & &
\end{smallmatrix}
\right),
\end{equation}
which defines an outer automorphism of $G= GL(n)$ that preserves the subgroups $GL(n,\Z_p)$, $B_{-}$, and $N$. The
contragredient adelic automorphic distribution is defined by
\begin{equation}
\label{ab10} \widetilde\tau (g)\ =_{\text{def}} \tau\(\widetilde
g\) \,,
\end{equation} and defines a map from $G(\A_f)$ to automorphic distributions in $V_{\widetilde\l,\widetilde\d}^{-\infty}$ with principal series parameters
\begin{equation}
\label{ab12} \widetilde \l\ = \ (-\l_{n}, -\l_{n-1}, \dots ,
-\l_1)\,, \ \ \ \widetilde \d\ = \ (\d_{n}, \d_{n-1}, \dots ,
\d_1)\,.
\end{equation}

\subsection{Mirabolic Eisenstein distributions and pairings}
\label{eisenstein}

The pairing we use to obtain the exterior square $L$-function on $GL(2n)$ involves automorphic distributions for
Eisenstein series induced from a one dimensional representation of   the so-called {\it mirabolic subgroup} $P'$,
the standard upper triangular $(n-1,1)$ parabolic subgroup of $G=GL(n)$. The mirabolic Eisenstein series originally
appear in papers of Jacquet-Shalika \cite{jseuler,jsextsq}, and in particular are prominent ingredients in the
integral representations \cite{bumpfriedberg,jsextsq} of the exterior square $L$-functions, the latter of which
served as an inspiration for the distributional pairing used in this paper.  In this subsection we summarize the
pertinent  properties of the mirabolic Eisenstein distributions from \cite[\S3\,and\,\S5]{mirabolic}.

Let $V$ be $n$-dimensional vector space  with standard basis $\{e_1,\ldots,e_n\}$, viewed as an algebraic group of
row vectors defined over $\Z$.   Let $\Phi_\infty$ denote the $\d$-function at any nonzero point in $V(\R)$, and
let $\Phi_p$ be a Schwartz-Bruhat function on $V(\Q_p)$ for $p<\infty$ -- that is, a locally constant function of
compact support.  The latter is  {\em unramified} when it is  the characteristic function of $V(\Z_p)$.

Suppose now that $\Phi(g)=\prod_{p\le \infty}\Phi_p(g_p)$ is a product of such $\Phi_p$ which are unramified for
all but finitely many $p$, and that $\omega$ and $\chi$ are finite order characters of $\Q^*\backslash \A^*$
(specifically, we shall let $\omega$ be the central character of $\pi$, and $\chi$ the adelization of a Dirichlet
character of parity $\eta\in \Z/2\Z$). The integral
\begin{equation}\label{IfromPhi}
        I(g,s) \ \ = \ \ \chi(\det g)\i\, |\det g|^s \,\int_{\A^*}\Phi(e_n tg)\,|t|^{ns}\,\chi^{-n}(t)\,\omega(t)\,d^*t
\end{equation}
is shown in \cite[\S5]{mirabolic} to define  a distribution in $g_\infty\in GL(n,\R)$ for any fixed value of
$g_f\in GL(n,\A_f)$, and the periodization
\begin{equation}\label{adeliceisnew}
    E(g,s) \ \ = \  \ \sum_{\g\in P'(\Q)\backslash G(\Q)} I(\g g,s)\,,
\end{equation}  convergent in the strong distributional topology for $\Re{s}>1$,  defines an  adelic automorphic Eisenstein distribution:~a map from $G(\A_f)$ to automorphic distributions on $G(\R)$.   It has a meromorphic continuation to $s\in \C-\{1\}$ with at most a simple pole at $s=1$.  These statements remain valid when $\Phi$ is  replaced by a finite linear combination of such products.
 In this paper, we always keep $\Phi_\infty$ equal to $\d_{e_1}$, the $\d$-function at $e_1\in V(\R)$.  The central character of $E$ is $\omega\i$,  inverse to the central character $\omega$ of the cuspidal automorphic distribution $\tau$ from the previous subsection.    Since $\Phi$ factors as $\Phi_\infty \times \Phi_f$, with $\Phi_f$ a Schwartz-Bruhat function on $V(\A_f)$, the integral (\ref{IfromPhi}) also factorizes as a product $I(g_\infty\times g_f,s)=I_\infty(g_\infty,s)I_f(g_f,s)$, where $I_\infty(g_\infty,s)$ represents the integral restricted to $\R^*$.

The restriction  of $E(g,s)$ to $G(\R)$ is a distribution vector for a certain degenerate principal series which we
now define. Let $P_{-}=w_{\text{long}}P'w_{\text{long}}$. The quotient
\begin{equation}
\label{mira2} Y\ \ = \ \ G(\R)/P_{-}(\R)
\end{equation}
can be naturally identified with the projective space of hyperplanes
 in $\R^n$, and is called a {\it generalized flag variety}, in analogy to (\ref{ps3}).  Let
\begin{equation}\label{Uradical}
   U \ \ = \ \  \left\{  \left(
\begin{smallmatrix}1& {\textstyle{*}} & \textstyle{\dots}  & {\textstyle{*}} \\
  & { }_{\scriptstyle{1}} & \ \  { }_{{ }_{\textstyle{0}}} &  \\
{ }^{\textstyle{0}} & & \ \ \ddots &  \\
 &  &  & 1
 \end{smallmatrix}
\right)  \right\}
\end{equation}
be the ``opposite'' of the unipotent radical of $P_{-}$. In analogy to (\ref{ps4}), we can identify its real points
 with the open Schubert cell in $Y$,
\begin{equation}
\label{mira3} U(\R)\ \simeq \ U(\R)\cdot eP_{-}(\R)\ \hookrightarrow \ Y\,,
\end{equation}
because $U$ and $P_{-}$ intersect trivially.

Any character of $P_{-}(\R)$ that is trivial on the identity component of   the center $Z(\R)$ of $G(\R)$ has the
form
\begin{equation}
\label{mira4}
\chi_{s, \e,\eta} \left(
\begin{smallmatrix}
\textstyle{a} & 0 &\textstyle{\dots} & 0 \\
{ }_{\textstyle{*}} &    &  & \\
\vdots & & { }^{\textstyle{B}} &   \\
\textstyle{*} &  & &
\end{smallmatrix}\right)\ = \ \,  |a|^{(n-1)(s-1)}  \, \sgn (a )^{\e+\eta} \, |\det B|^{1-s}\, \sgn(\det B)^\eta\,,
\end{equation}
for some $s\in \C$  and $\e, \eta \in \Z/2\Z$. It uniquely defines a $G$-equivariant $C^\infty$ line bundle
$\L_{s,\e,\eta } \to Y$  on whose fiber  at $eP_{-}(\R)$ the isotropy group $P_{-}(\R)$ acts by $\chi_{s,\e,\eta
}$.  Left translation by the group $G(\R)$ exhibits
\begin{equation}
\label{mira6}
\gathered
W_{s, \e,\eta  }^\infty \  = \ C^\infty(Y,\L_{s,  \e,\eta  })  \ \simeq
\qquad\qquad\qquad\qquad\qquad\qquad\qquad \qquad\qquad \ \ \ \  \
\\ \ \
 \{f \in  C^\infty(G(\R)) \mid
f(gp)=\chi_{s, \e,\eta  }(p^{-1})f(g)\ \text{for}\ g \in  G(\R),
p \in P_{-}(\R)\}
\endgathered
\end{equation}  as
the space of smooth vectors for a degenerate principal series representation $W_{s,\e,\eta}$.  Analogously to
(\ref{ps7}-\ref{ps9}), its space of distribution vectors  $W_{s, \e,\eta  }^{-\infty}$ is the space of distribution
sections of the same line bundle; that is, (\ref{mira6}) remains valid if  all three superscripts are changed to
$-\infty$. In particular, distributions $f\in W_{s,  \e,\eta  }^{-\infty}$   transform on the right as follows:
\begin{equation}
\label{mira6explicated}
    f\(g\ttwo{a}{}{\star}{B}\) \ = \
    |a|^{n(1-s)}\,  \sgn(a)^{\e+\eta}\,\sgn(\det B)^\eta  f(g)  \ \ \ \text{if} \ \ |a||\det B|=1\,.
\end{equation}
When $\Phi_\infty=\d_{e_1}$, a change of variables in (\ref{IfromPhi}) shows that $I_\infty(g,s)$ obeys the same
transformation law, making the adelic Eisenstein distribution $E(g,s)$ a map from $G(\A_f)$ to automorphic
distributions in $W_{s,\e,\eta}^{-\infty}$.

\subsection{Pairing of Automorphic Distributions}
\label{pairingsec}

In this subsection we describe our main analytic tool:~the pairing of a cuspidal automorphic distribution against a
mirabolic Eisenstein distribution. The result is a meromorphic function of the variable $s$ that parameterizes the
Eisenstein series. In the following sections we shall identify this function with the exterior square $L$-function,
multiplied by a product of functions $G_\d$ that are defined below in (\ref{eisen3}).  We shall describe the
pairing only  to the extent needed for our present purposes.  For the proof, and a generalization to arbitrary Lie
groups, we refer the reader to \cite{pairingpaper}.

We use subscripts to distinguish the different groups and flag varieties involved in the pairing:~for example,
$G_k$  denotes $GL(k)$, and $X_k$, $Y_k$  its flag varieties from  (\ref{ps3}) and (\ref{mira2}), respectively.
The partition  $2n=n+n$  induces embeddings
\begin{equation}
\label{pair1} G_n \times G_n \ \hookrightarrow \ G_{2n}\,,\qquad
X_n \times X_n \ \hookrightarrow \ X_{2n}\,.
\end{equation}
The translates of the latter  under the real points of the abelian subgroup
\begin{equation}
\label{pair2} U_{n,n} \  \ = \ \ \left\{  \begin{pmatrix}
\textstyle{I_{n}}  & \star \\ \textstyle{0_{n}}  &
\textstyle{I_{n}}  \end{pmatrix}\right\} \ \ \subset\ \   \ G_{2n}
\end{equation}
are disjoint, giving the embedding
\begin{equation}
\label{pair3} U_{n,n}(\R) \times X_n \times X_n \ \hookrightarrow \
X_{2n}
 \end{equation}
a dense open image.

We define the character
  \begin{equation}\label{pair4a}
   \theta : U_{n,n}(\A)\ \longrightarrow \ \C^*\,,\qquad
\theta\(\begin{smallmatrix} I_n & A \\ 0_n & I_n \end{smallmatrix}
\) = \psi_{+}(\tr A)\,,
\end{equation}
where again $\psi_{+}$ is the standard additive character on $\Q\backslash \A$ from section~\ref{autdistsec}. If
$f_1$ and $f_2 \in X_n$ are in general position, their isotropy subgroups are Borel subgroups whose intersection is
a Cartan  subgroup of $GL_n(\R)$.  As the latter acts with an open orbit on $Y_n\simeq \RP^{n-1}$, $G_n(\R)$ acts
with an open orbit on the triple product $X_n\times X_n\times Y_n$; in fact, the action provides a local
diffeomorphism between a neighborhood of the identity in $G_n(\R)/Z_n(\R)$ and a neighborhood of an arbitrary point
in the open orbit.

 Let $\phi\in C_c^\infty(G_n(\R))$ have total integral 1, and let $(f_1,f_2,f_3)\in X_n\times X_n \times Y_n$ be a point in the open orbit.
 With $E(s)$ as in (\ref{adeliceisnew}), the automorphic pairing is defined as
\begin{equation}\label{adelicpairing}
\gathered
    P(\tau,E(s)) \ \ =  \ \  \int_{Z_n(\A)G_n(\Q)\backslash G_n(\A)}F(g)\,dg\,, \ \ \ \text{where}
\\
  F(g) \ \ = \ \
    \int_{G_n(\R)}
    \left[ \int_{U_{n,n}(\Q)\backslash U_{n,n}(\A)}
    \tau\(u\!\ttwo{ghf_1}{}{}{ghf_2}\)\overline{\theta(u)}du   \right]
    E(ghf_3,s)\,\phi(h)\,dh\,.
\endgathered
\end{equation}
 Using the properties of the open orbit as well as decay estimates for $F(g)$,
this integral is shown to be well defined in \cite[\S4-5]{mirabolic}, though it is important that the integration
be done in the order specified in order to make sense.  The integrations defining $F(g)$ smooth the distribution to
give a function on $G_n(\A)$ which is  invariant under the center $Z_n(\A)$, because $\tau$ and $E$ have opposite
central characters $\omega$ and $\omega\i$, respectively.   This property of the finite order character $\omega$
means that
\begin{equation}
\label{pair8} \l_1 + \l_2 + \dots + \l_{2n}\ = \ 0\  \  \text{and} \, \ \ \ \d_1 +
\d_2 + \dots + \d_{2n}\ \equiv \ \e + n\,\eta \imod 2
\end{equation}
(cf.~(\ref{IfromPhi})). At full level $\omega$ is trivial,  and invariance under $-e\in G_{2n}(\Z)$ forces the
second sum to be zero. The overall invariance under the center means that the smoothing over $h\in GL(n,\R)$ really
drops to one over $SL^{\pm}_n(\R)$, the group of $n\times n$ real matrices with determinant $\pm 1$. The
$u$-integration appears in \cite{jsextsq} and is commonly known as a ``Shalika period''.  Indeed, the integral
(\ref{adelicpairing}) is very closely modeled on the integral representation in \cite{jsextsq}.  We require the
distributional analog because it is possible to  compute it explicitly as the exterior square $L$-function times a
ratio of Gamma factors, which is apparently not possible for the Jacquet-Shalika integral.  We should also
emphasize that even though there is a test function $\phi$ in the definition of the pairing, the value of the
pairing is independent of it:
 \smallskip

\begin{thm}[\cite{pairingpaper,mirabolic}]
\label{thmpairing} For every test function $\,\phi\in C_c^{\infty}(G_n(\R))$ the result of the inner two
integrations in (\ref{adelicpairing}) is a left $G_n(\Q)$- and $Z_n(\A)$-invariant function on $G_n(\A)$, whose
restriction to $G_n(\R)$ is smooth and left invariant under some congruence subgroup $\G'\subset GL(2n,\Z)$.  This
function is moreover integrable over $\G'\backslash G_n(\R)/Z_n(\R)$, and its integral over this quotient (which
equals $P(\tau,E(s))$  times the index of $\G'$) is holomorphic for $s\in \C-\{1\}$, with at most a simple pole at
$s=1$.  If  $\int_{G_n(\R)} \phi(g)\,dg=1$, as we have assumed, the integral does not depend on the choice of the
function $\phi$.
\end{thm}

Let $\phi_\theta(u)$  equal the product of $\overline{\theta(u)}$ with an arbitrary function $\phi_U \in
C_c^\infty(U_{n,n}(\R))$ of total integral 1, and consider the integral
\begin{equation}\label{bigPhidef}
\aligned
   & \!\!\!\!\!\!\!  \!\!\!\!  \Phi(g_{2n},g_n) \ \ = \\ & \int_{G_n(\R)}\int_{U_{n,n}(\R)}    \tau\(g_{2n}u\!\ttwo{hf_1}{}{}{hf_2}\)\,E(g_nhf_3,s)\,\phi_\theta(u)\,\phi(h)\,du\,dh\,.
   \endaligned
\end{equation}
Both $\tau$ and $E$ are now together smoothed over a dense open subset of the full flag variety for
$G_{2n}(\R)\times G_n(\R)$, making $\Phi$ a smooth automorphic function on this group.  Because the smoothing on
the right commutes with left translation,
\begin{equation}\label{bigPhiuse}
    F(g) \ \ = \ \ \int_{U_{n,n}(\Q)\backslash U_{n,n}(\A)}\Phi\(u\ttwo{g}{}{}{g},g\)\,\overline{\theta(u)}\,du
\end{equation}
(\cite[Lemma~3.9]{pairingpaper}).
Thus the pairing $P(\tau,E(s))$ is the integral of a Shalika period of a smooth, automorphic function, over
$Z_n(\A)G_n(\Q)\backslash G_n(\A)$. Jacquet-Shalika's integral representation \cite{jsextsq} for the exterior
square involves integration over the same domains, but of a different type of function.

The value of the pairing depends on the choice of flag representatives $f_1$, $f_2$, and $f_3$. To be concrete, we
shall choose
\begin{equation}
\label{fixflag} f_1  \ =  \ I_n \,, \ \ \ f_2  \ = \ \(\begin{smallmatrix}  0 & \cdots & 0 & 1  \\  \vdots & &
\iddots &  \\ 0 & 1 & & \\ 1 & 0 & \cdots & 0 \end{smallmatrix}\) \,, \ \ \ \text{and}\ \ \ f_3 \ =\
\(\begin{smallmatrix}  1 &1 & \cdots & 1  \\  0 &  & &  \\ \vdots &  & \textstyle{I_{n-1}}  &  \\ 0  &  & &
\end{smallmatrix}\) \,,\end{equation} which in fact determine an open $G_n(\R)$ orbit as required.  The value of
the pairing is unchanged if the base points are simultaneously multiplied by an element of $G_n(\R)$ on the left,
and changes by a factor of automorphy coming from (\ref{ps9}) and (\ref{mira6}) if individually multiplied on the
right.  We use this in section~\ref{unfolding} to switch to base points which are more convenient for a computation
there.

The Eisenstein distributions, like Eisenstein series, have functional equations relating $s$ and $1-s$. The pairing
$P(\tau,E(s))$, too, inherits such a functional equation from them, as of course will the exterior square
$L$-functions that they will be shown to represent. We use this in particular to derive the functional equation of
the exterior square $L$-functions for full level forms; the weaker, general functional equation enters into the
proof of theorem~\ref{mainthmwiths}  to extend full holomorphy from  $\Re{s}\ge 1/2$ to all of $\C$. The statement
involves the functions
\begin{equation}
\label{eisen3} G_\d(s) \ \  = \ \     \int_\R \! e(x)
\left(\sg(x)\right)^\d |x|^{s-1}\,dx  \ \  =  \ \  \begin{cases}
2(2\pi)^{-s}\, \G(s)\, \cos\textstyle\frac{\pi s}{2} &\text{if}\ \d=0\\
2(2\pi)^{-s}\, \G(s)\, \sin\textstyle\frac{\pi s}{2} &\text{if}\
\d=1
\end{cases}
\end{equation}
which show up in  all known  functional equations of $L$-functions.  The integral is conditionally convergent for
$0 < \re s <1$, and meromorphically continues to $\C$ via the formula on the right, a formula that can be more
succinctly summarized as \begin{equation}\label{eisen3b}
    G_\d(s) \ \ = \ \ i^\d \,\f{\G_\R(s+\d)}{\G_\R(1-s+\d)} \ \ , \ \ \ \ \text{with} \ \,
    \G_\R(s) \,=\,\pi^{-s/2}\G(\smallf s2)\ \ \text{and} \ \ \d\,\in\,\{0,1\}
\end{equation}
using Gamma function identities.

The functional equation of the pairing is calculated in \cite[(5.28)]{mirabolic} to be
\begin{equation}\label{adelicpenufe}
\gathered
    P(\tau,E(1-s)) \ \ =  \qquad\qquad\qquad\qquad\qquad\qquad\qquad\qquad\qquad\qquad\qquad\qquad \\
\qquad N^{2ns-s-n}\,\prod_{j\,=\,1}^n G_{\d_{n+j}+\d_{n+1-j}+\eta}(s+\l_{n+j}+\l_{n+1-j})\,P(\tau',E'(s))\,,
\endgathered
\end{equation}
where $\tau'$ is a translate of  the contragredient automorphic distribution $\widetilde \tau$ defined in
(\rangeref{ab10}{ab12}), and $E'$ is a mirabolic Eisenstein distribution induced from a possibly different linear
combination of $\Phi$'s (each having $\Phi_\infty=\d_{e_1}$).  The translate is by an element of $GL(2n,\Q)$ which
is ramified only at  places that $\tau$ and $E(s)$ are. In the special case that $\tau$ is invariant under
$GL(2n,\Z)$, $N=1$, $\omega$ is trivial, and $\e\equiv \eta\equiv 0\imod 2$, the relation simplifies to
\begin{equation}\label{penufefulllevel}
\aligned
   &   \!\!\!\!\!  \!\!\!\!\!\!\!\!  P(\tau,E(1-s))  \  \  = \\  \ \  \ \ \ \  &(-1)^{\d_1+\cdots+
\d_{n}}\,{\prod}_{j\,=\,1}^n \,
G_{\d_{n+j}+\d_{n+1-j}}(s+\l_{n+j}+\l_{n+1-j})\,P(\widetilde{\tau},E(s))\,,
\endaligned
\end{equation}
in which the Eisenstein data on both sides corresponds to the unramified choice at all $p<\infty$, and  $\d_{e_1}$
for $p=\infty$ (see \cite[(4.26)]{mirabolic}).

\section{Exterior square unfolding on
$GL(2n,\R)$}\label{unfolding}

In this section we explain how the pairing of automorphic distributions ``unfolds'' into a product of local
integrals, one for each place $p\le \infty$ of $\Q$. There are two possible approaches to unfolding distributional
pairings  such as ours that are patterned from integral
 representations of $L$-functions.  The first, carried out for the exterior square $L$-functions
 on $GL(4)$ in \cite{korea}, works directly  with the Fourier expansion of the automorphic
distribution $\tau$, and proceeds through a chain of intermediate pairings analogous to (\ref{adelicpairing}).  The
second approach reduces the distributional unfolding statement to the corresponding unfolding of the classical
integral representation.
 For space  reasons we shall execute the latter, as it allows us to quote from \cite{jsextsq}.  It is additionally possible to execute a hybrid argument that uses the mechanics of the classical unfolding, but applied to the smoothed function (\ref{bigPhidef}).

Throughout this section we assume that $\Re{s}$ is arbitrarily large, an assumption that is always made when
identifying Dirichlet series, and which entails no loss of generality. The unfolding involves the ``card shuffle''
permutation matrix
    \begin{equation}\label{sigmapermdef}
    \sigma \ \ = \ \
    \(\begin{smallmatrix}
 1 &   &   &   &   &   &   &   &   &   &   &  \\
  &   &   &   &   &   &  1 &   &   &   &   &  \\
  &  1 &   &   &   &   &   &   &   &   &   &  \\
  &   &   &   &   &   &   &  1 &   &   &   &  \\
  &   &  1 &   &   &   &   &   &   &   &   &  \\
  &   &   &   &   &   &   &   &  1 &   &   &  \\
  &   &   &   &  \ddots &   &   &   &   &   &   &  \\
  &   &   &   &   &   &   &   &   &   &  \ddots &  \\
  &   &   &   &   &  1 &   &   &   &   &   &  \\
  &   &   &   &   &   &   &   &   &   &   & 1 \\
\end{smallmatrix}\)\  , \ \ \det \sigma = \left\{%
\begin{array}{ll}
    +1, & n \equiv 0,1 \imod 4 \\
    -1, &  n \equiv 2,3 \imod 4 \,,  \\
\end{array}%
\right.
\end{equation}
which sends the elementary basis vectors
\begin{equation}\label{sigmamap}
\begin{array}{llllllllll}
~~          & e_1, & e_2, & e_3, & \ldots, & e_n, & e_{n+1}, &
e_{n+2}, & \ldots, & e_{2n}  \\
\hbox{to} & e_{1}, & e_3, & e_5, & \ldots,& e_{2n-1}, & e_2, &
e_4, & \ldots, & e_{2n}\, ,
\end{array}
\end{equation}
respectively.  In other words $e_k$ is mapped to $e_{2k-1\!\pmod{2n-1}}$ and
\begin{equation}\label{conjbysigma}
\gathered
    \text{the $(i,j)$-th entry of $g$ equals the}\\ (2i-1,2j-1)\text{-th entry (mod $2n-1$) of $\sigma g \sigma\i$,}
     \endgathered
\end{equation}
for $1\le i,j,k \le 2n-1$. This permutation maps the positive Weyl chamber for the diagonal image of $GL(n)
\hookrightarrow GL(2n)$ into a positive Weyl chamber for the ambient group.  We  regard $\sigma\in GL(2n,\Q)$, the
diagonally-embedded rational subgroup of $GL(2n,\A)$ that the adelic automorphic distribution $\tau$ is  invariant
under.

Let us recall the adelic Whittaker distribution (\ref{factorizepure2}),  $\Phi$ from (\ref{IfromPhi}),   and the
subgroups $N$ and $Z$ of $G=GL(n)$ from section~\ref{autdistsec}. The unfolding of the distributional pairing
results in products of the nonarchimedean integrals  considered by Jacquet-Shalika in \cite{jsextsq},
\begin{equation}\label{unfoldingprop3padic}
\aligned
   &\!\!\!\!\!\!\! \!\!\!\!  \Psi_p(s,W_p,\Phi_p) \ \ =  \\ & \int_{N(\Q_p)\backslash G(\Q_p)} \int_{L_0(\Q_p)} W_p\(\sigma \ell \ttwo{g }{}{}{g }\) \,\Phi_p(e_ng)\,|\det g|^s\,\chi(\det g)\i\,d\ell \,dg
\endaligned
\end{equation}
for $p<\infty$, where $L_0$ is the subgroup of matrices in $U_{n,n}$ for which the top right $n\times n$ block in
(\ref{pair2}) is strictly lower triangular.  Jacquet-Shalika proved this integral is absolutely convergent for
$\Re{s}$ sufficiently large; moreover, so is the product of these integrals over all primes $p$.

Jacquet-Shalika's unfolding of their integral representation involves these nonarchimedean factors, as well as an
archimedean integral of the form (\ref{unfoldingprop3padic}), in which $\Phi_\infty$ is a Schwartz function on
$\R^n$.  Our strategy will be to show that our distributional pairing (\ref{adelicpairing}) is an instance of
Jacquet-Shalika's global integral from \cite[\S5]{jsextsq}, and to compute its unfolded factors. However, we first
will make some independent remarks about how the archimedean integral that arises can also be viewed as a {\em
local pairing of distributions}. Namely, consider the integral
\begin{equation}\label{unfoldingprop3realrevised} \int_{Z(\R) N(\R)\backslash G(\R)} \int_{G(\R)} \int_{L_0(\R)} w_{\l,\d}\(\sigma \ell \ttwo{ghf_1}{}{}{ghf_2}\)\,
I_\infty(ghf_3,  s   )\,
\phi(h)\,d\ell\,dh\,dg\,.
\end{equation}
This differs from the Jacquet-Shalika archimedean integral in that it involves the Whittaker distribution
$w_{\l,\d}$ from (\ref{btransform})  and a distribution $I_\infty$,  as opposed to smooth
functions.  It is similar to the global pairing (\ref{adelicpairing}) in that it first involves the
smoothing of a distribution by right convolution with a smooth function of compact support.  As such it can be
thought of as a local pairing between the Whittaker distribution $w_{\l,\d} \in V_{\l,\d}^{-\infty}$ and
$I_\infty$;  the latter is a distribution vector, specifically a $\d$-function, for the degenerate
principal series $W_{s,\e,\eta}$, as in (\ref{mira6}) and \cite[(5.22)]{mirabolic}.  In
(\ref{bigPhidef}-\ref{bigPhiuse}) we saw that the global pairing could alternatively be computed using an
additional smoothing   performed on $\tau$ by the function $\phi_\theta\in C_c^\infty(U_{n,n}(\R))$ because of
\cite[lemma~3.9]{pairingpaper}.  A similar argument shows that the inner integration in
(\ref{unfoldingprop3realrevised}) can be rewritten as
\begin{equation}\label{wldsmoothing}
    \gathered
   \int_{L_0(\R)} w_{\l,\d}\(\sigma\ell\ttwo{ghf_1}{}{}{ghf_2}\)\,d\ell\qquad \qquad\qquad\qquad\qquad\qquad\qquad\qquad\\
   \qquad\qquad = \ \   \int_{L_0(\R)}\int_{U_{n,n}(\R)} w_{\l,\d}\(\sigma\ell\ttwo{g}{}{}{g} u \ttwo{hf_1}{}{}{hf_2}\)\,\phi_\theta(u)\,du\,d\ell
   \,.
    \endgathered
\end{equation}
Indeed, after performing the change of variables $u\mapsto \ttwo{g}{}{}{g}\i u \ttwo{g}{}{}{g}$, $\phi_\theta(u)$
remains the product of the conjugation-invariant character $\overline{\theta(u)}$ with a smooth function of compact
support and total integral 1, while the measure $du$ is unchanged; this change of variables puts $u$
 immediately after $\ell$ in the argument of $w_{\l,\d}$. We can uniquely factor $u=\ell'n'=n'\ell'$,
where $\ell'\in L_0(\R)$ and $n'=\ttwo{I}{X}{}{I}$, with $X$ upper triangular.  The factor $\ell'$ can be removed
by reversing the order of integration and then changing variables in $\ell$, while the factor $n'$ can be removed
by using the fact that the Whittaker transformation character $\psi_{+}(c(\sigma n'\sigma\i))=e(\tr X)=\theta(u)$.
Hence the $U$-integration drops out, and the right hand side equals the left hand side.

We now insert (\ref{wldsmoothing}) into  (\ref{unfoldingprop3realrevised}). The extra smoothing over $U_{n,n}(\R)$
resmooths the $L_0$-integration, meaning that the latter can be passed outside of the $h\in G(\R)$ integration.  We
conclude that (\ref{unfoldingprop3realrevised}) is the integral over $Z(\R)N(\R)\backslash G(\R)$ and $L_0(\R)$ of
a function $W_{\text{prod}}(\sigma \ell\ttwo{g}{}{}{g},g)$, where $W_{\text{prod}}(g_{2n},g_n)\in
C^{\infty}(GL(2n,\R)\times GL(n,\R))$ is the right convolution of $w_{\l,\d}(g_{2n})I_\infty(g_n,  s  )$ against a
function in $C_c^\infty(GL(2n,\R)\times GL(n,\R))$.  Using the Dixmier-Malliavin factorization theorem
\cite{DixMal} and a Fourier transform argument analogous to the one in the proof of \cite[Lemma (8.3.3)]{jpss}, it
is possible to show the rapid decay and  absolute convergence of   this integration, and hence of the local pairing
(\ref{unfoldingprop3realrevised}).  The local pairing is independent of the choice of smoothing function $\phi$ (up
to normalization by its total integral), by modifying the de Rham argument that shows this independence for the
global pairing in theorem~\ref{thmpairing} (see \cite[\S4]{pairingpaper}). Thus one can independently make sense of
(\ref{unfoldingprop3realrevised}), and think of the global pairing as factoring as an $L$-function times it.  We
emphasize, however, that this is not the approach we take in the calculation below.

Our reduction to the unfolding of \cite{jsextsq} requires only a particular type of smoothing, for which the
integrability can be seen more directly.  (In particular we do not need to use the above remarks about the local
pairing.) We remarked in section~\ref{pairingsec} that the  diagonal action of $G(\R)$  on $X_n\times X_n\times
Y_n$ produces a local diffeomorphism between a neighborhood of the identity in $G(\R)/Z(\R)$, and a neighborhood of
the base point $(f_1,f_2,f_3)\in X_n\times X_n\times Y_n$.  Thus, when the representation parameters $\l$, $\d$,
$s$, $\e$, and $\eta$ are fixed, any product of three functions of small support around this base point, one for
each factor, necessarily lifts to a smooth function $\phi\in C_c^\infty(G(\R)/Z(\R))$. However, in view of its
construction in terms of $(f_1,f_2,f_3)$, smoothing by $\phi$ amounts to a separate smoothing of $\tau$~-- by the
product of  the smoothing functions on $X_n\times X_n$,   viewed as function on $GL(2n,\R)$~-- and of $E$ by the function on $Y_n$, viewed as function on
$G(\R)$.   The
  additional $U_{n,n}$-smoothing in
(\ref{bigPhidef}) also smooths $\tau$, but not $E$.  Taking into account (\ref{pair3}), this $U_{n,n}$-smoothing combines with the smoothing over $X_n\times X_n$ to together smooth $\tau$ over an open subset of $X_{2n}$.  If the supports of $\phi_\theta$ in (\ref{bigPhidef}) and of the smoothing functions on the two $X_n$ factors are sufficiently small, the smoothing over $X_{2n}$ takes place over $\ttwo{f_1}{}{}{f_2}N'(\R)B_{-}(\R)$, where $N'=N_{2n}$ is the subgroup of unit upper triangular matrices in $GL(2n)$. For simpler reasons, the smoothing of $E$ over $Y_n$ takes place over $f_3U(\R)P_{-}(\R)$, provided of course the smoothing function on the $Y_n$ factor has sufficiently small support (see (\ref{mira3})).  We have thus shown that $\Phi(g_{2n},g_n)$ splits as a product of separate smoothings of $\tau$ and $E$ over open subsets of their respective flag varieties, $X_{2n}$ and $Y_n$:
\begin{equation}\label{splitint}
   \Phi(g_{2n},g_n) \, = \,    \int_{N'(\R)}\int_{U(\R)} \tau \(g_{2n}\!\ttwo{f_1}{}{}{f_2}\!n'\)
E(g_n  f_3 u ,s)\,\phi'(n')\,\phi''(u)\,du\,dn',
\end{equation}
where both $\phi'\in
C_c^\infty(N'(\R))$  and  $\phi''\in C_c^\infty(U(\R))$ have  support concentrated near the identity.

In effect,   the right hand side of (\ref{splitint}) is a product of smooth automorphic forms, one of which is a cusp form in the representation space of $\pi$,  and the other  a mirabolic Eisenstein series of the type considered by Jacquet-Shalika (see the comments following
\cite[(5.24)]{mirabolic}).  In particular, for such a smoothing function $\phi$, the global pairing is an instance
of the Jacquet-Shalika integral.  Likewise, $W_{\text{prod}}(g_{2n},g_n)$ splits as a product of functions in each
variable:~the archimedean Whittaker function
\begin{equation}\label{rightconv3}
W_\infty(g_{2n}) \ \ = \ \ \int_{N'(\R)}w_{\l,\d}\(g_{2n}\ttwo{f_1}{}{}{f_2}n'\)\,\phi'(n')\,dn'
\end{equation}
 of that  cusp form, and
    \begin{equation}\label{rightconv2}
    I_{\text{JS},\infty}(g_n,s) \ \ = \ \ \int_{U(\R)}I_{\infty}(g_n  f_3 u ,s)\,\phi''(u)\,du\,,
\end{equation}
the archimedean component of the function that is periodized to form that Eisenstein series.  The global integral
hence splits as a product over all $p\le \infty$ of Jacquet-Shalika local integrals.  In this setting the
archimedean integral $\Psi_\infty(s,W_\infty,\Phi_{\text{JS},\infty})$ (i.e., the analog of
(\ref{unfoldingprop3padic}) for $p=\infty$) is equal to
\begin{equation}\label{rightconv4}
\aligned
& \!\!\!\! \!\!\!\! \Psi_\infty(s,w_{\l,\d}) \  \  := \ \ \int_{Z(\R)N(\R)\backslash G(\R)} \int_{L_0(\R)} W_{\infty}\(\sigma\ell\ttwo g{}{}g\)\,I_{\text{JS},\infty}(g,  s  )\,d\ell\,dg     \\
&   = \ \ \int_{N(\R)\backslash G(\R)} \int_{L_0(\R)} W_{\infty}\(\sigma\ell\ttwo g{}{}g\) \Phi_{\text{JS},\infty}(e_n g)\,|\det g|^s\,\sgn(\det g)^\eta\,d\ell\,dg\,,
\endaligned
\end{equation}
where
\begin{equation}\label{rightconv1}
 \Phi_{\text{JS},\infty}(v) \ \ = \ \    \int_{U(\R)}\d_{e_1}(vf_3u)\,\phi''(u)\,du \ , \ \ \ \ v \, \in \, \R^n\,,
\end{equation} is the archimedean component of the global Schwartz function that Jacquet-Shalika periodize their mirabolic Eisenstein series from
(cf.~(\ref{IfromPhi}) and \cite[(5.20-5.21)]{mirabolic}).  Because of the earlier discussion, the local pairing
(\ref{unfoldingprop3realrevised}) is equal to $\Psi_\infty(s,w_{\l,\d})$ times a normalizing factor related to the
total integral of the smoothing function $\phi$. The absolute convergence of the second integral in
(\ref{rightconv4}) is covered by Jacquet-Shalika \cite{jsextsq}.  We have now shown:

\begin{prop}\label{unfoldingprop}(Unfolding of the automorphic pairing)  Assume  that $\tau$ comes from a pure tensor satisfying (\ref{factorizepure2}), that  $\Re{s}$ is sufficiently large, and that the test function $\phi\in C_c^\infty(G(\R))$ splits the global pairing into a product of smoothings as in (\ref{splitint}).  Then the pairing $P(\tau,E(s))$ from (\ref{adelicpairing}) factorizes as the product
\begin{equation*}
    P(\tau,E(s)) \ \ = \ \ \Psi_\infty(s,w_{\l,\d})\,\times\,\prod_{p\,<\,\infty} \Psi_p(s,W_p,\Phi_p)\,.
\end{equation*}
\end{prop}

\noindent As we mentioned above, the restriction on $\phi$ could be removed using the fact that both the  local
pairing (\ref{unfoldingprop3realrevised}) and $P(\tau,E(s))$ depend only on the total integral of $\phi$.
Technically speaking the argument of \cite{jsextsq} requires the decay estimate \cite[theorem~2.19]{rapiddecay}
for  smooth but not necessarily $K$-finite vectors, which was  then known to experts.

Jacquet and Shalika calculated their nonarchimedean local integrals in the unramified case, i.e., when $W_p$ is the
standard spherical Whittaker function, $\Phi_p$ is the characteristic function of $V(\Z_p)$,  and  both $\omega_p$
and $\chi_p$ are  trivial on $\Z_p^*$ \cite[Prop.~7.2]{jsextsq}:\footnote{The appearance of  $\chi\i$ in
(\ref{unfoldingprop3padic}) instead of their $\chi$  is due to a  different convention relating Dirichlet
characters to idele class characters.}
\begin{equation}\label{localunramcalc}
    \Psi_p(s,W_p,\Phi_p) \ \ = \ \  L_p(s,\pi,Ext^2\otimes\chi)\,.
\end{equation}
For example, when $\pi$ corresponds to a full-level cusp form, then $\pi_p$ is unramified for each $p<\infty$. If
$\omega$ and $\chi$ are furthermore both trivial, then the  pairing $P(\tau,E(s))$ simplifies to
\begin{equation}\label{unramifpnu}
    P(\tau,E(s)) \ \ = \ \ \Psi_\infty(s,w_{\l,\d}) \cdot     \ L(s,\pi,Ext^2)\,,
\end{equation}
  where $L(s,\pi,Ext^2)$ is the product over all the local factors for $p<\infty$.
The functional equation (\ref{penufefulllevel}) applies to this case, with  (\ref{unramifpnu}) adapted to the
contragredient automorphic distribution $\widetilde\tau$,
\begin{equation}\label{missingsign}
P(\widetilde{\tau},E(s)) \ \ = \ \ (-1)^{\d_2+\d_4+\d_6+\cdots+\d_{2n}}
\,\cdot \Psi_\infty(s,w_{\tilde{\l},\tilde{\d}})  \cdot \,L(s,\tilde{\pi},Ext^2)\,.
\end{equation}
This is because the Whittaker integral (\ref{whittonta1}) applied to (\ref{ab10}) involves the inverse character
$\psi_{+}\i$.  It can be converted back to $\psi_{+}$ using automorphy under a diagonal matrix with alternating
$\pm 1$ entries; the sign comes from the value of the inducing character $\chi_{\l,\d}$ on this matrix,
$(-1)^{\d_1+\d_3+\d_5+\cdots +\d_{2n-1}}=(-1)^{\d_2+\d_4+\d_6+\cdots+\d_{2n}}$.

We now focus on the computation of   $\Psi_\infty(s,w_{\l,\d})$ for $\Re{s}$ large. At this point it is convenient
for us  to   switch the base points $f_1$, $f_2$, and $f_3$ from (\ref{fixflag}) to another triple in the same
orbit; this will be accomplished by right-translation in $h$ by the matrix
\begin{equation}\label{h0def}
    h_0 \ \ = \ \ \(\begin{smallmatrix}
 \ 0 &  1 &  1 &  \cdots & 1 \\
 \ 0 &  0 &  1 &  1 & \vdots \\
 \vdots & \ddots  &  0 &  1 & 1 \\
 \ 0  & 0  & \cdots  &  0 & 1 \\
 -1 &  1 &  1 &  \cdots & 1 \\
\end{smallmatrix}\) \ \ = \ \ \(\begin{smallmatrix}
 1 &   &   &  &  & 1 \\
  &  1 &   &  &  & 1 \\
  &   &  \ddots & &  & \vdots \\
  &   &   & \ddots &  & \vdots \\
  &   &   &   & 1 &1 \\
  &   &   &   & & 1 \\
\end{smallmatrix}\)\(\begin{smallmatrix}
 \ 1 &   &   &   &   &  \\
 \ 1 &  -1 &   &   &   &  \\
 \ \vdots &  -1 &  -1 &   &   &  \\
 \ \vdots &  \vdots &  \ddots &  -1 &   &  \\
 \ 1 &  -1 &  \cdots &  -1 &  -1 &  \\
 -1 &  \ 1 &  \cdots &  \cdots &  \ 1 & 1 \\
\end{smallmatrix}\)\!.
\end{equation}
Every entry above the diagonal in the first matrix is 1,
 while  every entry on and below the diagonal in the last matrix is
 -1, except those indicated
  in the first column and last row.   Let
$w_k$ denote the matrix $w_{\text{long}}\in GL(k)$ from (\ref{ab9}),
\begin{equation}\label{wndef}
    w_k \ \ = \ \ \(\begin{smallmatrix}
 & & & 1 \\
 &  & 1 &  \\
  & \iddots & &  \\
 1 & & &
    \end{smallmatrix}\)
 \ \ , \ \ \ \ \
\det(w_k) \ \ = \ \ (-1)^{k(k-1)/2}\,.
\end{equation}
 Note that
  \begin{equation}\label{h0claim1}
\ttwo{w_{n-1}}{}{}1  \, h_0  \, w_n  \ \ = \ \ \ttwo{w_{n-1}}{}{}1  \, h_0  \, f_2  \ \ =  \ \ \(
\begin{smallmatrix}
 1 &   &   &  \\
 1 &  \ddots &   &  \\
 \vdots &  \ddots &  \ 1 &  \\
 1 &  \cdots &  \ 1 & -1 \\
\end{smallmatrix}\)
\end{equation}
and that the first row of  $w_n h_0 \ttwo{1}{{\bf 1}}{}{I_{n-1}}=w_n h_0 f_3$ is $[-1\,0\,0\,\cdots\,0]$, where
$\bf 1$ denotes the $(n-1)$-dimensional row vector of all 1's.  The last matrices in  lines
(\ref{h0def}-\ref{h0claim1}) are both lower triangular.  It follows that we may reassign
\begin{equation}\label{secondfixflags}
    f_1 \ \ = \ \ \ttwo{I_{n-1}}{{\bf 1}^t}{}{1} \ \ , \ \ \
    f_2 \ \ = \ \ \ttwo{w_{n-1}}{}{}{1} \ \ , \ \ \ \text{and} \ \
    f_3 \ \ = \ \ w_n\,
\end{equation}
in (\ref{rightconv3}-\ref{rightconv2}), at the expense of multiplying by an overall sign:
\begin{equation}\label{unfoldingprop30}
    \Psi_\infty(s,w_{\l,\d}) \ \ = \ \ \kappa_1 \, \int_{Z(\R)N(\R)\backslash G(\R)}  \int_{L_0(\R)} W_{\infty}\(\sigma\ell\ttwo g{}{}g\)\,I_{\text{JS},\infty}(g,  s  )\,d\ell\,dg\,,
\end{equation}
where  $\kappa_1 = (-1)^{\d_2+\cdots+\d_{n-1} + \d_{2n}+\e+\eta n(n+1)/2}$ (cf.~the discussion following
(\ref{fixflag})).  The choice (\ref{secondfixflags}) will be in effect for the duration of the paper, and the sign
$\kappa_1$ will be taken into account when using the functional equation (\ref{adelicpenufe}).

The integrand in (\ref{unfoldingprop30}) is smooth, and so the value of the integral is unchanged if the range of
$g$-integration is restricted to the dense open  subset of lower triangular matrices with bottom right entry 1.
 We may decompose such a matrix $g$ uniquely as a product $g=bq$ of   a matrix $b$ in
\begin{equation}\label{archintnot1}
 B_{-,n-1} \ \ = \ \ \left\{ C =  \(\begin{smallmatrix} c_{1,1} & 0 &  0 & 0  \\
                                    \vdots & \ddots &  0 & 0  \\
                                    c_{n-1,1} &\cdots &
                                    c_{n-1,n-1} & 0 \\
                                     0&0 & 0 & 1 \\
        \end{smallmatrix}\) \right\},
\end{equation}
the group of  lower triangular $(n-1)\times(n-1)$ matrices embedded into the upper left corner, and $q$ in $Q$, the
group of unit lower triangular matrices which differ from the identity matrix only in their bottom row.  In these
coordinates, the Haar measure $dg$ restricted to  $B_{-,n-1}\times Q$  is the product of
\begin{equation}\label{archintnot2}
\(\prod_{j\,=\,1}^{n-1}|c_{j,j}|^{j-(n+1)/2}
\)|\det C|^{-(n+1)/2}\, \prod_{1\le j \le i < n}d c_{i,j}
 \end{equation}
and the standard Haar measure $dq$ on the unipotent subgroup $Q$.  We shall abbreviate this measure more succinctly
as $|C^{-\rho}||\det C|^{-(n+1)/2}dCdq$.

Consider the $q$-integration in (\ref{unfoldingprop30}),
\begin{equation}\label{unfoldingprop32newa}
    \int_{Q(\R)}\int_{U(\R)} W_{\infty}\(\sigma \ell\ttwo{Cq}{}{}{Cq}\)
I_\infty(C q  f_3 u,  s  )\,\phi''(u)\,du\,dq\,.
\end{equation}
 The defining integral (\ref{IfromPhi}) has the property that
 \begin{equation}\label{lastlineadded}
 I_\infty(Cqf_3u, s ) \ \ = \ \ |\det C|^s \, \sgn(\det C)^\eta\,   I_\infty(qf_3u, s )\,.
 \end{equation}

   At the same time, $Q$ and $U$ are conjugate by the long Weyl group element $w_n=f_3$, so
 (\ref{secondfixflags}) and the definition (\ref{IfromPhi}) (with our specific choice of $\Phi_\infty=\d_{e_1}$ that defines the Eisenstein distribution $E$) imply $I_\infty(qf_3u,  s )=\sgn(\det w_n)^\eta\d_e(qw_nuw_n)$, i.e., the two integrations collapse to the submanifold $q=w_nu\i w_n$.   Let us write the matrix $\ell \in L_0$ as $\ttwo{I}{Z}{}{I}$, where $$Z= \(\begin{smallmatrix}
          0 & & & \\
          z_{1,1} & 0 & & \\
          \vdots & \ddots & \ddots & \\
          z_{n-1,1} & \cdots & z_{n-1,n-1} & 0 \\
    \end{smallmatrix}\)$$ is a strictly lower triangular matrix.  The Haar measure on $L_0$ is given by $d\ell=dZ=\prod_{1\le j \le i < n}dz_{i,j}$, and satisfies
    $d(ZC\i) = |C^{-{\rho}}|\,|\det C|^{-(n-1)/2}\, dZ$.

 Thus (\ref{unfoldingprop30}) can be written as
\begin{multline}\label{unfoldingprop31newz}
 \kappa_1' \int_{\R^{n(n-1)/2}}\int_{Q(\R)}\int_{\R^{n(n-1)/2}} W_\infty\(\sigma \ttwo{C}{Z}{}{C}\ttwo{q}{}{}{q}\)
\phi''(w_nq\i w_n) \ \times \\ \times \ |C^{-2\rho}||\det C|^{s-n}\,\sgn(\det C)^\eta\,dZ\,dq\,\,dC\,,
\end{multline}
where $\kappa_1'=\kappa_1\sgn(\det w_n)^\eta=(-1)^{\d_2+\cdots+\d_{n-1} + \d_{2n}+\e+n\eta }$.

Recall from the discussion following (\ref{splitint}) that $\phi'\in C_c^\infty(N'(\R))$ and $\phi''\in
C_c^\infty(U(\R))$  can be arbitrary functions of small support, subject to the condition that the total integral
of $\phi\in C_c^\infty(G(\R))$ is equal to one.  As the support shrinks, the product of the total integrals of
$\phi'$ and $\phi''$ tends to one; this is because  the  line bundle characters (implicitly used to set up the
splitting of the integration on the product of flag varieties) take  values close to one near the identity.  We
shall henceforth drop the total integral constraint on $\phi$, but insist that $\phi'$ and $\phi''$ are approximate
identity sequences, with support concentrating on the identity element.  Then (\ref{unfoldingprop31newz}) converges
to $\Psi_\infty(s,w_{\l,\d})$  independently of which functions are chosen, so long as their supports both shrink
to zero.

At this point, $\phi''$ is the only remnant of the Eisenstein series, and it serves to smooth $W_\infty$ --
resmooth, in fact, in light of (\ref{rightconv3}).  For any fixed $\phi'$, shrinking the support of $\phi''$ to
$\{e\}$ in (\ref{unfoldingprop31newz}) gives a well-defined limit that itself approaches $\Psi_\infty(s,w_{\l,\d})$
as the support of $\phi'$ shrinks to $\{e\}$.  Thus we may replace $\phi''$ by a delta function at the identity
element, and simplify (\ref{unfoldingprop31newz})  to the absolutely convergent subintegration over $q=e$,
\begin{equation}\label{rightconv3rewritten}
\kappa_1' \int_{\R^{n(n-1)}} W_\infty\(\sigma \ttwo{C}{Z}{}{C} \)
\,d\mu\,,
\end{equation}
where
\begin{equation}\label{dmu}
    d\mu \ \ = \ \ |C^{-2\rho}||\det C|^{s-n}\,\sgn(\det C)^\eta\,dC\,dZ\,.
\end{equation}
Define the family of approximate identities on $N'(\R)$ by
\begin{equation}\label{phit}
   \phi'_t(n') \  \ = \ \ a_t^{2\rho}\,\phi'(a_tn'a_t\i)\ \, , \ \ \ \ t \,\rightarrow\,\infty\,,
\end{equation}
where the diagonal matrix $a_t=\operatorname{diag}(t^{(n-1)/2},t^{(n-3)/2},\ldots,t^{-(n-1)/2})$   takes the value
$t$ on each positive simple root. Consider the Whittaker integral (\ref{rightconv3}) with $\phi'_t$ in place of
$\phi'$ and without the flags,
\begin{equation}\label{tshrink1}
    W_{\infty,t}(g_{2n}) \  =  \ \int_{N'(\R)}w_{\l,\d}\(g_{2n}\,a_t\i n'a_t\)\,\phi'(n')\,dn'
       \   = \ a_t^{\rho-\l}\,W_{\infty,1}(g_{2n}a_t\i)\,,
\end{equation}
which converges to $w_{\l,\d}(g_{2n})$ as $t\rightarrow\infty$ for $g_{2n}$ lying in the  open Schubert cell
$N(\R)B_{-}(\R)$.    We conclude that
\begin{equation}\label{tshrink2}
    \Psi_\infty(s,w_{\l,\d}) \ \ = \ \ \lim_{t\rightarrow\infty} \kappa_1'\,a_t^{\rho-\l}\, \int_{\R^{n(n-1)}}
W_{\infty,1}\(\sigma\ttwo{C}{Z}{}{C}\ttwo{f_1}{}{}{f_2}a_t\i\)d\mu\,.
\end{equation}
In particular,
 the limiting terms in the integrand defining $\Psi_\infty(s,w_{\l,\d})$ are asymptotic limits of the Whittaker function in the negative Weyl chamber, a fact which reflects the alternative construction of automorphic distributions as boundary values  of automorphic forms \cite{flato}.    In the next two sections we will derive a coordinate change on $C$ and $Z$ that is more convenient for this integration.

\section{Matrix Decompositions}
\label{slicingsec}

In this section we restrict our attention to  real groups, and shall identify the algebraic groups of the previous
sections with their real points for notational compactness. Our calculations of the Gamma factors involve finding
explicit unit upper triangular representatives for cosets in $G/B_{-}$, of $G=GL(n,\R)$ modulo the group $B_{-}$ of
invertible lower triangular matrices.   In linear algebra this is sometimes called the $UDL$
decomposition\begin{footnote}{More common are references to the so-called LU decomposition of a generic matrix
$g\in GL(n,\R)$ as a product $g=\ell u$ of a lower triangular matrix $\ell$ and an upper triangular unipotent
matrix $u$.}\end{footnote}, and is computed by elementary column operations. Such a decomposition for
$g=(g_{i,j})\in GL(n,\R)$ exists  if and only if    the condition
\begin{equation}\label{djnonzerocond}
    d_k  \ =_{\text{def}}  \ \det\(\begin{smallmatrix}
    g_{k,k} & \cdots & g_{k,n} \\
    \vdots & \ddots & \vdots \\
    g_{n,k} & \cdots & g_{n,n} \\
    \end{smallmatrix}\)     \ \ \neq \ \ 0 \, , \ \ 1 \ \le k \ \le n
\end{equation}
holds, in which case it is unique.

\begin{lem}\label{ludmilla}  Let $g=(g_{i,j})$ be an $n\times n$ matrix with indeterminate entries.  Then there exist $n\times n$ matrices $b_+$, $a$, $b_-$, whose entries are polynomial functions of the $g_{i,j}$, with $b_+$ upper triangular, $a$ diagonal, $b_-$ lower triangular, satisfying the formal identity
\[
g \ \ = \ \ b_+\,a^{-1}\,b_-\,.
\]
The decomposition can be chosen so that all nonzero entries of $b_+$ and $b_-$ are determinants of subblocks of $g$
obtained by removing rows and columns, and the diagonal entries of $a$ are products of two determinants of
subblocks; concretely
\begin{equation*}
\begin{aligned}
&a_{i,i}\ \ = \ \ \left(\det\(g_{k,\ell}\)_{i+1\le k,\ell \le n
}\right)\left( \det\(g_{k,\ell}\)_{i\le k,\ell \le n } \right) \,,
\\
&\qquad\qquad \(b_+\)_{i,j}\ \ = \ \
\det\(g_{k,\ell}\)_{\,\stackrel{\scriptstyle{k=i{\rm~or~}k>j}}{j\le
\ell \le n\ \ \ \ }} \ \ \ \rm{for}\ i \leq j \,,
\\
&\qquad\qquad\qquad\qquad \(b_-\)_{i,j}\ \ = \ \
\det\(g_{k,\ell}\)_{\,\stackrel{\scriptstyle{j\le k \le n\ \ \
}}{\ell=i\rm{~or~}\ell>j}} \ \ \ \rm{for}\ i \geq j \,.
\end{aligned}
\end{equation*}
\end{lem}
Zhelobenko \cite{Zhelobenko} describes a less explicit formula of this type for $GL(n)$. Fomin-Zelevinsky \cite{fz}
give a   similar   formula for any reductive matrix group. In principle, our formula can be deduced from theirs,
but it is just as easy to prove our formula directly.
\begin{proof} The crucial observation is that the entries
of $b_+$ do not change if $g$ is multiplied from the right by a unipotent lower triangular matrix $u\,$: for $i\leq
j$, $(b_+)_{i,j}$ is the determinant of a matrix $m$ obtained from $g$ by omitting certain rows and exactly the
first $j-1$ columns; the passage from $g$ to $gu$ has the effect of multiplying $m$ on the right by the left bottom
$(n+1-j)$ square block of $u$, which does not affect $\det m = (b_+)_{i,j}$. In view of the $UDL$ decomposition,
there exists a unipotent lower triangular matrix $u$ whose entries depend rationally on those of $g$, such that
$gu$ is upper triangular. Hence, and because of what was said just before, the identity of the lemma holds for
entries on and above the diagonal if and only if it holds for any upper triangular matrix $g$. In that special
case, the identity can be verified directly. Switching the roles of left and right, as well as upper and lower, the
identity for entries on and below the diagonal follows the same way. \end{proof}

\begin{cor}\label{ludcor}
If $g =n h n_-$, with $n$, $h$, $n_-$ upper triangular unipotent, diagonal, and lower triangular unipotent,
respectively, then
\[
\begin{aligned}
&h_{i,i}\ \ = \ \ \f{d_i}{d_{i+1}} \ \ = \ \
\f{\det\bigl((g_{k,\ell})_{\,k,\ell\geq
i}\bigr)}{\det\bigl((g_{k,\ell})_{\,k,\ell>i}\bigr)}\,,
\\
&\qquad n_{i,i+1} \ \ = \ \
\f{\det\((g_{k,\ell})_{\stackrel{\scriptstyle{\,k\geq i,\, k \neq
i+1}}{\ell>i\ \ \ \ \ \ }}\)}{d_{i+1}} \ \ = \ \
\f{\det\((g_{k,\ell})_{\stackrel{\scriptstyle{\,k\geq i,\, k\neq
i+1}}{\ell>i\ \ \ \ \ \
}}\)}{\det\bigl((g_{k,\ell})_{k,\ell>i}\bigr)}\ .
\end{aligned}
\]
\end{cor}

Our next topic is choosing a convenient set of coordinates for the entries of a matrix so that the factors
$h_{i,i}$ and $n_{i,i+1}$ have simple expressions.  This will ultimately be useful  in our computation of the Gamma
factors for $L(s,\pi,Ext^2)$ in the next section.

\begin{lem}\label{slicinduct}
Let $A=(a_{i,j})$ be an $n \times n$ matrix with indeterminate entries.  Let $A'=(a'_{i,j})$ be the matrix formed
from $A$ by replacing the entries $a_{i,n}$ in the last column with entries $\tilde{a}_{i,n}$ such that for any $1
\le k \le n$
\begin{equation}\label{slicinductproperty}
    \det\(\begin{smallmatrix}
 a_{k,k} &  \cdots &  a_{k,n-1} & \tilde{a}_{k,n} \\
 \vdots &  \ddots &  \vdots & \vdots \\
 a_{n,k} &  \cdots &  a_{n,n-1} & \tilde{a}_{n,n} \\
\end{smallmatrix}\) \ \ = \ \ (-1)^{n-k}\,a_{k,n}\,
\det\(\begin{smallmatrix}
 a_{k+1,k} &  \cdots &  a_{k+1,n-1}  \\
 \vdots &  \ddots &  \vdots &  \\
 a_{n,k} &  \cdots &  a_{n,n-1} \\
\end{smallmatrix}\)\,.
\end{equation}
Let $A''=(a''_{i,j})$ be the matrix formed from $A$ by removing the last column.  Then for any $1 \le i \le  n-1$
one has that
\begin{equation}\label{slicinductpunch}
    \f{\det\((a'_{k,\ell})_{\stackrel{\scriptstyle{\,k\geq i,\, k\neq
i+1}}{\ell>i\ \ \ \ \ \
}}\)}{\det\bigl((a'_{k,\ell})_{k,\ell>i}\bigr)} \ \ = \ \
\f{a_{i,n}}{a_{i+1,n}} \ + \
  \f{\det\((a''_{k,\ell})_{\stackrel{\scriptstyle{\,k\geq i,\, k\neq
i+1}}{\ell>i-1\ \ \ \
}}\)}{\det\bigl((a''_{k,\ell})_{k>i,\,\ell>i-1}\bigr)} \, ,
\end{equation}
i.e., the difference between the quantities $n_{i,i+1}$ for $A'$ and $A''$ is $\f{a_{i,n}}{a_{i+1,n}}$.
\end{lem}
For example, in the case $n=2$ the matrix $A'$ is given by
\begin{equation}\label{aprimen2}
    A' \ \ = \ \ \(\begin{matrix} a_{1,1} & \tilde{a}_{1,2}    \\
    a_{2,1} & \tilde{a}_{2,2} \end{matrix}\) \ \ = \ \
     \(\begin{matrix} a_{1,1} &
     a_{1,2}+\f{a_{1,1} a_{2,2}}{a_{2,1}}
    \\
    a_{2,1} & a_{2,2} \end{matrix}\)\,,
\end{equation}
and for  $i=1$ indeed $\f{\tilde{a}_{1,2}}{\tilde{a}_{2,2}}=\f{a_{1,2}}{a_{2,2}}+\f{a_{1,1}}{a_{2,1}}$ .

\begin{proof}  The assertion for a given value of $i$ is
 independent of the matrix entries that lie above and to the left of
  the lower right
 $(n+1-i)\times (n+1-i)$ block of $A$.  It therefore suffices to prove the
 result for $i=1$.  We must verify an assertion about determinants
 of $(n-1)\times(n-1)$ subblocks, whose rightmost column is
one of the last two columns.  As these are unchanged
 by adding linear combinations of the left $n-2$ columns to their
 last column, we may instead replace by $A'$ by the matrix $B$,
 formed from $A'$ by replacing the last column by a vector of
 the form
 $(y_1,y_2,0,\ldots,0)$.  (This is possible since the $a_{i,j}$
 are indeterminates.)  Because of property
 (\ref{slicinductproperty}) for  $k=2$, and because these
 determinants are unchanged by this column operation, we have that
 $y_2=a_{2,n}$.

The left hand side of (\ref{slicinductpunch}) is equal to
\begin{equation}\label{leftpunch}
    \f{\det\(\begin{smallmatrix}
 \star &  \cdots &  \star & y_1 \\
 a_{3,2} &  \cdots &  a_{3,n-1} & 0 \\
 \vdots &  \ddots &  \vdots & \vdots \\
 a_{n,2} &  \cdots &  a_{n,n-1} & 0 \\
    \end{smallmatrix}\)}{ \det\(\begin{smallmatrix}
 \star &  \cdots &  \star & y_2 \\
 a_{3,2} &  \cdots &  a_{3,n-1} & 0 \\
 \vdots &  \ddots &  \vdots & \vdots \\
 a_{n,2} &  \cdots &  a_{n,n-1} & 0 \\
    \end{smallmatrix}\)} \ \ = \ \ \f{y_1}{y_2} \ \ = \ \
    \f{y_1}{a_{2,n}}\,.
\end{equation}
At the same time, (\ref{slicinductproperty}) shows that $\det B=\det A'=(-1)^{n-1}a_{1,n}\det C$, where $C$ is the
submatrix of $A$ formed by removing its first row and last column.  Expanding $\det B$ by minors along the last
column gives the formula
\begin{equation}\label{slicinductminor}
    (-1)^{n-1}\,a_{1,n}\,\det C \  \ = \ \ \det{B} \ \ = \  \
    (-1)^{n-1} \,y_1\,\det C  + (-1)^n\,y_2 \det D\, ,
\end{equation}
where $D$ is the submatrix formed from $A$ by removing the second row and last column.  This shows that
$\f{a_{1,n}}{a_{2,n}}=\f{y_1}{a_{2,n}}-\f{\det D}{\det C}$, for $y_2=a_{2,n}$.  The determinants of $D$ and $C$ are
the numerator and denominator, respectively, of the last term in (\ref{slicinductpunch}).  The lemma  now follows
from (\ref{leftpunch}). \end{proof}

In  the previous lemma, we did not say anything about the existence of entries $\tilde{a}_{i,n}$ which satisfy
(\ref{slicinductproperty}).  Actually, this is not difficult to show, and since we will rely on this type of change
of coordinates repeatedly, we record it here as part of a more general lemma.

\begin{lem}\label{agreeingproperty}
Let $A=(a_{i,j})$ be an $n\times n$ matrix with indeterminate entries.  Then there exists a unique $n\times n$
matrix $B=(b_{i,j})$, with   entries of the  form $b_{i,j}=a_{i,j}+r_{i,j}$, satisfying the following properties:
\begin{enumerate}
    \item $r_{i,j}=0$ if $i+j\le n$;
    \item the $r_{i,j}=r_{i,j}(a_{k,\ell})$  are rational functions of the $a_{k,\ell}$ for which $k\ge i$,
        $\ell\le j$, and $(k,\ell)\neq (i,j)$;
    \item $ \det\!\(\begin{smallmatrix}
 b_{n-j+i,i} & \textstyle{\dots} & b_{n-j+i,j} \\
 \vdots &  \ddots & \vdots \\
 b_{n, i} &  \textstyle{\dots} & \ b_{n, j} \\
\end{smallmatrix}\) \,  = \
\det\!\(\begin{smallmatrix}
0 & 0  & 0 & a_{n-j+i,j} \\
0 & 0  & a_{n-j+i+1,j-1} & 0 \\
 0 &  \iddots & 0 &  0 \\
a_{n,i} &  0 &   0 &  0
\end{smallmatrix}\) $
for any $i<j$.
\end{enumerate}
\end{lem}

\begin{proof}
The conclusions of the lemma are inherited by the bottom left $(n-1)\times (n-1)$ corner of $B$ when the first
index is reduced by one. We may therefore argue by induction on $n$ and assume that the $b_{i,j}$, for $i>1$ and $j
< n$, are already known, and that the identities asserted by the lemma are satisfied unless $j=n$. That leaves the
$n$ equations corresponding to $j=n$ and $1\leq i\leq n$ to determine the $n$ coefficients $b_{i,n}$, $1\leq i\leq
n$ -- or equivalently, the corresponding modification terms $r_{i,n}$. We now proceed by downward induction on the
index $i$. At each step, when the determinant on the left of the relevant equation is expanded in terms of the top
row, the induction hypothesis insures that we get the term we want, plus $j-1$ terms not involving $b_{i,n}$ at
all, plus the product $\pm\,r_{i,n}\,a_{i+1,n-1}$.
That forces us to choose $r_{i,n}$ so that this product cancels the sum of the other unwanted terms. In other
words, there is exactly one choice of $r_{i,n}$ resulting in the equality we need to establish. \end{proof}

Lemma~\ref{slicinduct} may be applied repeatedly to the matrix $B$ of the previous lemma to give the following
formula for the components $h_{i,i}$ and $n_{i,i+1}$ in corollary~\ref{ludcor}.  In simplifying to the following
statement, we have used (\ref{wndef}) to compute signs.

\begin{cor}\label{extsqmatcalc}
With the notation of corollary~\ref{ludcor}, let $g$ equal the matrix $B$ from \lemref{agreeingproperty}.  Then one
has that
\begin{equation*}
    h_i \ \ = \ \ (-1)^{n-i}\,\f{a_{n,i}a_{n-1,i+1}\cdots
    a_{i,n}}{a_{n,i+1} a_{n-1,i+2}\cdots a_{i+1,n}}  \ \, , \ \ \ 1
    \,\le\, i \, \le \, n \, ,
\end{equation*}
and
\begin{equation*}
    n_{i,i+1} \ \ = \  \
    \sum_{j\,=\,n+1-i}^{n}\,\f{a_{i,j}}{a_{i+1,j}} \ + \
\f{\det\(\begin{smallmatrix}
 a_{i,1} &  a_{i,2} &  \cdots & a_{i,n-i} \\
  a_{i+2,1} &  a_{i+2,2} &  \cdots & a_{i+2,n-i} \\
 \vdots &  \ddots &  \ddots & \vdots \\
 a_{n,1} &  a_{n,2} &  \cdots & a_{n,n-i} \\
\end{smallmatrix}\)}{\det\(\begin{smallmatrix}
 a_{i+1,1} &  a_{i+1,2} &  \cdots & a_{i+1,n-i} \\
 a_{i+2,1} &  a_{i+2,2} &  \cdots & a_{i+2,n-i} \\
 \vdots &  \ddots &  \ddots & \vdots \\
 a_{n,1} &  a_{n,2} &  \cdots & a_{n,n-i} \\
\end{smallmatrix}\)}
       \ \, , \ \ \ 1
    \,\le\,i\,<\,n\,.
\end{equation*}
The denominator in the first formula is equal to one if $n=i+1$ (i.e., it is an empty product).  If the entries
$a_{i,j}=0$ for $i+j\le n$, then $n_{i,i+1} =
 \sum_{j\,=\,n+1-i}^{n}\f{a_{i,j}}{a_{i+1,j}}$.
\end{cor}

The rest of this section concerns the $UDL$ decomposition of the $2n\times 2n$ matrix $\sigma
\ttwo{C}{Z}{}{C}\ttwo{f_1}{}{}{f_2}$ from the end of section~\ref{unfolding}, which we will then input into formula
(\ref{btransform}) to give an explicit formula for $w_{\l,\d}$ evaluated on it, in terms of exponentials and power
functions.  We will also perform a change of variables, similar to the one in \lemref{agreeingproperty}, to
simplify the form of this upper triangular matrix and the ensuing calculation of the integral
(\ref{rightconv3rewritten}). Setting $ c_j = \sum_{i\le j}c_{i,j}$,
$\ttwo{C}{Z}{}{C}\ttwo{f_1}{}{}{f_2}=\ttwo{Cf_1}{Zf_2}{}{Cf_2}$ has the form
\begin{equation}\label{azaflgform}
   \(\begin{smallmatrix}
 c_{1,1} &   &  & &   &  c_1 &   &   &   &  0 &  \\
  c_{2,1} & c_{2,2} &  & &   &  c_2 &   &   &   &  z_{1,1} &  \\
 c_{3,1} &c_{3,2}  &  c_{3,3} & &   &  c_3 &   &   &  z_{2,2} &  z_{2,1} &  \\
  \vdots   &  \vdots   & \ddots & \ddots & & \vdots&  &    \iddots &  \iddots &  \vdots &  \\
 c_{n-1,1}   &  c_{n-1,2} & c_{n-1,3} & \cdots & c_{n-1,n-1} &  c_{n-1} &   &    \iddots   &  \cdots  &  z_{n-2,1}\\
  &   &   & &  &  1 &  z_{n-1,n-1} &  z_{n-1,n-2} &  \cdots &  z_{n-1,1} &  \\
  &   &   & &  &   &   &   &   &  c_{1,1} &  \\
  &   &   & &  &   &   &   &  c_{2,2} &  c_{2,1} &  \\
  &   &   & &  &   &   &  \iddots&  \iddots &  \vdots &  \\
  &   &   & &  &   &  c_{n-1,n-1} &  c_{n-1,n-2} &  \cdots &  c_{n-1,1} &  \\
  &   &   & &  &   &   &   &   &   & 1 \\
\end{smallmatrix}\).
\end{equation}  The last row and column of this  matrix and of the permutation matrix $\sigma$ are
zero except for the 1 in their last entry.   To consider the decomposition of the $2n\times 2n$ matrix
$\sigma\ttwo{C}{Z}{}{C}\ttwo{f_1}{}{}{f_2}$ in $NB_{-}$, it therefore suffices instead to study its upper $(2n-1)
\times (2n-1)$ block
\begin{equation}\label{upperblock}
   A \ \ = \ \  \(\begin{smallmatrix}
 c_{1,1} &  0 &  &  \cdots  &  c_1 &   &   &   & 0 \\
  0 &  0 &   &    \cdots&  0 &   &   &   & c_{1,1} \\
 c_{2,1} &  c_{2,2} & 0  &  \cdots  &  c_2 &   &   &   & z_{1,1} \\
  0 & 0   & 0  &   \cdots &  0 &   &   &  c_{2,2} & c_{2,1} \\
  &   &  _{\scriptstyle{\ddots}} &   &  c_3 &   &   &  z_{2,2} & z_{2,1} \\
 &   &   &   &  0 &   &  \iddots  &  \iddots & \vdots \\
 \vdots &  \vdots & \ddots  &   &  \vdots &   &  0 &  \vdots & \vdots \\
 \vdots &  \vdots &   & 0  &  0 &  0 &  c_{n-2,n-2} &  \cdots & c_{n-2,1} \\
 c_{n-1,1} & c_{n-2,2}  & \cdots  & c_{n-1,n-1} &  c_{n-1} &  0 &  z_{ n-2,n-2} &  \cdots & z_{n-2,1} \\
  &   &   &   &   &  c_{n-1,n-1} &  c_{n-1,n-2} &  \cdots & c_{n-1,1} \\
  &   &   &   &  1 &  z_{n-1,n-1} &  z_{n-1,n-2} &  \cdots & z_{n-1,1} \\
\end{smallmatrix}\).
\end{equation}%

Having defined this matrix $A$, we will now change its coordinates in a way very similar to that in
lemma~\ref{slicinduct}.  Since some of the variables occur multiple times in $A$, this must be done more
delicately.  If $x$ is one of the variables on the {\em right hand side} of $A$, we will use the notation $A_{x}$
to denote the unique contiguous square submatrix of $A$ whose upper right corner has the entry $x$, and whose
bottom row is the bottom row of $A$.  More generally, we shall also use this subscript notation to denote the block
of the same location in other matrices derived from $A$.

Now we describe the actual change of variables. Starting with $x=z_{1,1}$ and continuing in order to $z_{2,1},
z_{2,2}, z_{3,1}, z_{3,2},\ldots,z_{n-1,n-1}$ (i.e., go left as far as possible in each row, then down to the
rightmost entry two rows below), consider $\det{A_x}$, which is a linear combination of the entries in the top row
of $A_x$. Shift this $x=z_{i,j}$ by the unique amount such that the determinant is now a linear function of
$z_{i,j}$; in other words, replace $z_{i,j}$ by the unique expression $\tilde{z}_{i,j}$ so that $\det{A_x}$ becomes
$(-1)^{k-1} z_{i,j} \det{A_{c_{i+1,j+1}}}$, where $k$ is the size of the block $A_x$. The difference
$\tilde{z}_{i,j}-z_{i,j}$ is a rational function of the other variables in the block $A_{x}$, so updates to later
$z_{i,j}$ may affect earlier $\tilde{z}_{i',j'}$ because of this.

 After completing the change of
variables in the $z_{i,j}$'s, we turn to the $c_{i,j}$'s.  Because of their positioning within the matrix, some
terms do not need to be shifted:
\begin{equation}\label{unshifted}
    \tilde{c}_{j,j} \ \ = \ \ c_{j,j} \ \ \  \text{and} \ \ \ \tilde{z}_{n-1,j} \ \ = \ \
    z_{n-1,j} \ , \ \ \ \text{for~} 1 \le j
    \le n-1\,.
\end{equation}
 Starting with $x=c_{n-1,1}$ and
continuing in order to $c_{n-1,2},\ldots,c_{n-1,n-2},c_{n-2,1}$, $c_{n-2,2},\ldots,c_{3,1},c_{3,2},c_{2,1}$ (i.e.,
left and then up, skipping all $c_{j,j}$), we replace $c_{i,j}$ by the unique expression $\tilde{c}_{i,j}$ such
that the determinant of $A_x$ becomes $(-1)^{k-1} c_{i,j} \det A_{z_{i,j+1}}$, where $k$ is the size of the block.
The difference $\tilde{c}_{i,j}-c_{i,j}$ is, likewise, a rational function of the other variables in the block
$A_{c_{i,j}}$.

 It is important to keep in mind that this second step of changing
variables $c_{i,j}\rightsquigarrow\tilde{c}_{i,j}$ also results in updating some of the $\tilde{z}_{i,j}$, further
changing them. Let $B$ denote the matrix $A$ after all these  changes have been performed; if we set
\begin{equation}\label{tildecj}
    \tilde{c}_j \ \ = \  \ \sum_{i \le j}\,\tilde{c}_{i,j}\,,
\end{equation}
then the entries of $B$ are simply the tilded versions of the respective entries of $A$, and one has the relations
\begin{equation}\label{detblockrelations}
    \det B_x \ \ = \ \ \left\{
                         \begin{array}{ll}
                           (-1)^{k-1}\,z_{i,j}\,\det B_{c_{i+1,j+1}}\,, & x \,=\,z_{i,j} \\
                           (-1)^{k-1}\,c_{i,j}\,\det B_{z_{i,j+1}}\, , & x \,=\,c_{i,j}\,,
                         \end{array}
                       \right.
\end{equation}
where again $k$ is the size of the block $B_x$.

In what follows we will use some auxiliary quantities formed from the entries of $B$.   For $1 \le j \le n- 1$, set
\begin{equation}\label{djdef2}
    e_j \ \ = \ \ \det\(\begin{smallmatrix} 0 & 0 & \cdots & 0
& \tilde{c}_{j+1,j+1} & \tilde{c}_{j+1,j}
\\ \vdots & \vdots & \iddots   & \iddots & \vdots& \vdots
\\ 0 & \tilde{c}_{n-2,n-2} & \tilde{c}_{n-2,n-3} & \cdots &
\tilde{c}_{n-2,j+1} & \tilde{c}_{n-2,j}
\\          \tilde{c}_{n-1,n-1}   & \tilde{c}_{n-1,n-2} &
\tilde{c}_{n-1,n-3} & \cdots & \tilde{c}_{n-1,j+1} &
            \tilde{c}_{n-1,j} \\
\tilde{z}_{n-1,n-1} & \tilde{z}_{n-1,n-2} & \tilde{z}_{n-1,n-3} &
\cdots & \tilde{z}_{n-1,j+1} &
\tilde{z}_{n-1,j} \\
\end{smallmatrix}\) \, ,
\end{equation} and
\begin{equation}\label{sjdef}
s_{j} \ \ = \ \  \ \frac{(-1)^{1+(n-j)(n-j+1)/2} \
\tilde{c}_{j} \,e_j}{c_{j,j}\cdots c_{n-2,n-2}\,   c_{n-1,n-1}
}\,.
\end{equation}
 When $j=n-2$, the determinant is to be
interpreted as $\det \ttwo{\tilde{c}_{n-1,n-1}}{\tilde{c}_{n-1,n-2}}{\tilde{z}_{n-1,n-1}}{\tilde{z}_{n-1,n-2}}
=-c_{n-1,n-2}\,z_{n-1,n-1}$, while for $j=n-1$ it is simply $z_{n-1,n-1}$.

 \begin{lem}\label{altsumdet}
With $s_j$ as given in (\ref{sjdef}), one has that
\begin{equation*}
    \sum_{j\,=\,1}^{n-1} \, s_{j} \ \ = \ \ \sum_{j\,=\,1}^{n-1}
    z_{n-1,j}\,.
\end{equation*}
 \end{lem}
\begin{proof}
We will prove the equivalent  identity
\begin{equation}\label{altsumdet1}
\gathered
   \sum_{j\,=\,1}^{n-1} (-1)^{\f{n(n-1)\,-\,(n-j)(n-j+1)}{2}} \
    \tilde{c}_{j}\
    e_j\ c_{1,1}\cdots c_{j-1,j-1} \qquad
    \qquad\qquad\qquad\ \
     \\ \qquad\qquad\qquad = \ \ -\,(-1)^{\f{n(n-1)}{2}} \,c_{1,1}\cdots
    c_{n-1,n-1}\,\sum_{j\,=\,1}^{n-1} z_{n-1,j}\,.
\endgathered
\end{equation}
Consider the $n\times n$ matrix
\begin{equation}\label{altsumdet2}
    \(\begin{smallmatrix}
 0 &  0 & \cdots   &  0  & \tilde{c}_{1,1} \\
  \vdots &  \vdots  &  \iddots   &  \tilde{c}_{2,2} & \tilde{c}_{2,1} \\
  0 &  0  &  \iddots &  \vdots & \vdots \\
 0 &  \tilde{c}_{n-1,n-1} &  \hdots &  \tilde{c}_{n-1,2} & \tilde{c}_{n-1,1} \\
 -\sum_{j=1}^{n-1}z_{n-1,j} &  z_{n-1,n-1} &  \hdots &  z_{n-1,2}
  & z_{n-1,1} \\
    \end{smallmatrix}\) ,
\end{equation}
whose determinant is  the right hand side of (\ref{altsumdet1}). The value of the determinant is unchanged after
adding each of the last $n-1$ columns to the first, and therefore also equals
\begin{equation}\label{altsumdet3}
    \det \( \begin{smallmatrix}
 \tilde{c}_1 &  0 &  \hdots & \hdots  &  0 &  0 & \tilde{c}_{1,1} \\
 \tilde{c}_2 &  \vdots &  \iddots &  0  &  0 &  \tilde{c}_{2,2} & \tilde{c}_{2,1} \\
 \tilde{c}_3 & \vdots  &  \iddots &  0 &  \tilde{c}_{3,3} &  \tilde{c}_{3,2} & \tilde{c}_{3,1} \\
 \tilde{c}_4 &  0  & 0  &  \tilde{c}_{4,4} &  \tilde{c}_{4,3} &  \tilde{c}_{4,2} & \tilde{c}_{4,1} \\
 \vdots &  \iddots &  \iddots &  \iddots &  \iddots &  \vdots & \vdots \\
 \tilde{c}_{n-1} &  \tilde{c}_{n-1,n-1} &  \hdots &  \tilde{c}_{n-1,4} &  \tilde{c}_{n-1,3} &  \tilde{c}_{n-1,2} & \tilde{c}_{n-1,1} \\
 0 &  z_{n-1,n-1} &  \hdots &  z_{n-1,4} &  z_{n-1,3} &  z_{n-1,2} & z_{n-1,1} \\
 \end{smallmatrix} \).
\end{equation}
Expanding this last determinant by minors along the first column yields
\begin{equation}\label{altsumdet4}
    \sum_{j\,=\,1}^{n-1} \,(-1)^{j-1} \,  \tilde{c}_j \ e_j \
    [(-1)^{n}c_{1,1}]\cdots [(-1)^{n-j+2}c_{j-1,j-1}]\, .
\end{equation}
Both the sign here and the sign on the left hand side of (\ref{altsumdet1}) equal $\prod_{k=n-j+1}^{n-1}(-1)^k$, so
 (\ref{altsumdet4})
equals the expression on the left hand side of (\ref{altsumdet1}).
 \end{proof}

As in \lemref{slicinduct}, the reason for changing $A$ to $B$ is to find a simpler expression for its upper
triangular representative modulo $N_{-}$.   This representative, or at least what we need of it for our
computations, is described in the following two propositions.

\begin{prop}\label{ext2reductionshort}
When all $c_{i,j}$ and $z_{i,j}$ are nonzero, the matrix $B$ has a unit upper triangular representative modulo
$B_{-}$, the sum of whose entries  just above the diagonal is
\begin{equation}\label{sumabove}
    \sum_{1\le j \le i < n } \, \f{c_{i,j}}{z_{i,j}} \ \  +  \ \
\sum_{1\le j \le i < n-1  } \, \f{z_{i,j}}{c_{i+1,j}} \ \ -  \ \
\sum_{j=1}^{n-1}\,  z_{n-1,j}\,.
\end{equation}
\end{prop}
At this point it is also possible to explicitly write down the diagonal entries of the lower triangular factor;
however, we postpone this until after the statement of proposition~\ref{lowertrianpartprop} below, which gives
additional information.

 \begin{proof}
Recall the explicit formula for these quantities denoted $n_{i,i+1}$ that was given in corollary~\ref{ludcor}.
Suppose momentarily that the terms $\tilde{c}_j$, $1 \le j < n$, were not present in the matrix $B$.  Then
\lemref{agreeingproperty} would show that the sum  of the $n_{i,i+1}$, over $1\le i \le 2n-2$, equals the
expression in (\ref{sumabove}) -- but without the last term $-\sum_{j=1}^{n-1}z_{n-1,j}$.  Instead, according to
lemma \ref{slicinduct}, the effect of the $\tilde{c}_j$'s is to shift $n_{2j-1,2j}$ by the ratio of determinants
\begin{equation}\label{ratioofdets}
\f{\det\(\begin{smallmatrix}
  &   &   &   &  \tilde{c}_j &   &   &   &  \\
 c_{j+1,j+1} &   &   &   &  \tilde{c}_{j+1} &   &   &   & \tilde{z}_{j,j} \\
  &   &   &   &   &   &   &  c_{j+1,j+1} & \tilde{c}_{j+1,j} \\
  &  c_{j+2,j+2} &   &   &  \tilde{c}_{j+2} &   &   &  \tilde{z}_{j+1,j+1} & \tilde{z}_{j+1,k} \\
  &   &  \ddots &   &  \vdots &   &   &  \iddots & \vdots \\
  &   &   &  c_{n-1,n-1} &  \tilde{c}_{n-1} &   &  \iddots &  \iddots & \vdots \\
  &   &   &   &   &  c_{n-1,n-1} &  \cdots &  \cdots & \tilde{c}_{n-1,j} \\
  &   &   &   &  1 &  z_{n-1,n-1} &  \cdots &  \cdots & z_{n-1,j} \\
\end{smallmatrix}\)  }{\det\(\begin{smallmatrix}
  &   &   &   &   &   &   &   & c_{j,j} \\
 c_{j+1,j+1} &   &   &   &  \tilde{c}_{j+1} &   &   &   & \tilde{z}_{j,j} \\
  &   &   &   &   &   &   &  c_{j+1,j+1} & \tilde{c}_{j+1,j} \\
  &  c_{j+2,j+2} &   &   &  \tilde{c}_{j+2} &   &   &  \tilde{z}_{j+1,j+1} & \tilde{z}_{j+1,k} \\
  &   &  \ddots &   &  \vdots &   &   &  \iddots & \vdots \\
  &   &   &  c_{n-1,n-1} &  \tilde{c}_{n-1} &   &  \iddots &  \iddots & \vdots \\
  &   &   &   &   &  c_{n-1,n-1} &  \cdots &  \cdots & \tilde{c}_{n-1,j} \\
  &   &   &   &  1 &  z_{n-1,n-1} &  \cdots &  \cdots & z_{n-1,j} \\
\end{smallmatrix}\)  } \, ,
\end{equation}
which simplifies to $\f{\tilde{c}_j e_j}{c_{j,j}\cdots c_{n-1,n-1}\,\det(w_{n-j+1})}=-s_j$. The sum of these,
according to \lemref{altsumdet}, is indeed this missing term $-\sum_{j=1}^{n-1}z_{n-1,j}$.
 \end{proof}

The previous result suggests a coordinate change that involves multiplying the $c_{i,j}$ and $z_{i,j}$ by the
variables that occur beneath them in their columns on the right hand side of the matrix   $A$  from
(\ref{upperblock}).  It is first useful to introduce an alternative coordinate labeling which does not
differentiate between which came from $C$, and which came from $Z$.  Denote the variable in the $(j,2n-i)$-th
position in $A$ by $y_{i,j}$.  These are defined for $2i \le j < 2n$ (because of the configuration of zeroes at the
top of each column); for example, $y_{1,2}=c_{1,1}$, $y_{1,3}=z_{1,1}$, $y_{1,4}=c_{2,1}$, and $y_{2,4}=c_{2,2}$.
The proposition suggests the change of coordinates $y_{i,j}\rightsquigarrow x_{i,j}x_{i,j+1}\cdots x_{i,2n-1}$ so
that the quotients $\f{y_{i,j}}{y_{i,j+1}}=x_{i,j}$, for $2i\le j \le 2n-2$, and (\ref{sumabove}) equals
$\sum_{2i\le j \le 2n-2}x_{i,j}-\sum_{i<n}x_{i,2n-1}$.  We now  regard $B$  as a $(2n-1)\times(2n-1)$ matrix with
entries $\tilde{c}_{i,j}$ and $\tilde{z}_{i,j}$, which are each rational functions of the $x_{i,j}$.  We shall also
extend the notation $A_x$ from before to let $B_{x_{i,j}}$ denote the $(2n-j)\times(2n-j)$ contiguous square
subblock of $B$ whose top right entry corresponds to the position of the variable $y_{i,j}$ (more simply, we refer
to this as  the ``position of $x_{i,j}$'' in $B$, even though this variable occurs in multiple positions). The
determinantal property defining the change of variables can now be redescribed as follows:
\begin{equation}\label{determinantalproperty}
    \det B_{x_{k,\ell}} \ \ = \ \ (-1)^{(2n-\ell)-1}\,(\det B_{x_{k+1,\ell+1}})\,\prod_{\ell'\ge \ell} x_{k,\ell'}\ , \ \ 2k
    <\ell
    <2n-1\,.
\end{equation}
 The product is taken over all positions at or below $x_{k,\ell}$ in its column in (\ref{upperblock}).  In particular, $\det B_{x_{k,\ell}}$ is $(-1)^{(2n-\ell)(2n-\ell-1)/2}$ times the product of all $x_{k',\ell'}$ that corresponding to the positions of $y_{k',\ell'}$ on or below the antidiagonal of $B_{x_{k,\ell}}$.

The following proposition gives a recursive formula for the lower triangular factor in the $UDL$ decomposition of
$B$ in terms of the $x_{i,j}$.

\begin{prop}\label{lowertrianpartprop}
Decompose the matrix $  B =  B(n)$ as above into the product of a unit upper triangular matrix, and a  lower
triangular matrix $b_{-}=b_{-,n}$.  Let  $u_n$ denote the $n\times n$ unit lower triangular matrix whose
subdiagonal entries are all 1, and $u_{n}^{-}$ the $n\times n$ lower triangular matrix that differs from $u_n$ only
in that all entries in its bottom row are $-1$. Then the following recursive relation holds for all $n \ge 1$ when
each $x_{i,j}\neq 0$:
\begin{equation*}
    b_{-,n+1} \ \ = \ \
    \ttwo{b_{-,n}}{}{}{I_2}m_1\,m_2\,m_3\,m_4\,,
\end{equation*}
where $m_1=\operatorname{diag}(u_n,u_n^{-},1)$ and $m_3=\operatorname{diag}(u_n,u_{n+1})$ are block diagonal
matrices, and  $m_2=\operatorname{diag}(x_{1,2n},\ldots,x_{n,2n},x_{n,2n},\ldots,x_{1,2n},1)$ and $m_4=
\operatorname{diag}(x_{1,2n+1},\ldots,\\ x_{n,2n+1}, 1,x_{n,2n+1},\ldots,x_{1,2n+1})$ are diagonal matrices.
\end{prop}

 \begin{proof}
Let $B^{(j)}=B(n+1)(m_j\cdots m_4)\i$ for $j\le 4$. We shall equivalently demonstrate  that $B^{(1)}$, after being
multiplied by  some unit upper triangular matrix on the left, is equal to $\ttwo{  B(n)}{}{}{I_2}$. The passage
from $  B(n+1)$ to
 $B^{(4)}=  B(n+1)m_4\i$  involves dividing columns $i$ and $2n+2-i$ by  $x_{i,2n+1}$, for $i\le n$.  The entries  of $B^{(4)}$ corresponding to the positions of $\tilde{c}_{i,j}$ and $\tilde{z}_{i,j}$ in $B(n+1)$ are $\tilde{c}^{(1)}_{i,j}=\f{\tilde{c}_{i,j}}{x_{j,2n+1}}$ and
$\tilde{z}^{(1)}_{i,j}=\f{\tilde{z}_{i,j}}{x_{j,2n+1}}$, respectively. (Note, however, that the $\tilde{c}_i$ in
the $n+1$-st column are unchanged). In particular, the last $n+1$ entries of the bottom row of $B^{(4)}$ are equal
to one, because of (\ref{unshifted}).

The inverse of $u_n$ is the $n\times n$ matrix with  $1$'s on the diagonal, $-1$'s immediately below the diagonal,
and zeroes everywhere else.  Accordingly, $m_3\i$ has $1$'s on its diagonal, $-1$'s immediately below the diagonal
in all but its $n$-th column, and zeroes everywhere else.  Multiplying a matrix by $m_3\i$ on the right  subtracts
the entry immediately to the right from every entry not in the $n$-th column -- a procedure which can be executed
one column at a time, going from left to right.  The matrix $B^{(3)}$ is formed from $B^{(4)}$ by this operation.
Its bottom row therefore has zeroes in all but its last entry, which is 1.  Let $\tilde{c}^{(3)}_{i,j}$ denote its
entries that correspond to positions of $\tilde{c}_{i,j}$ that occur on the {\em left} hand side of $B(n+1)$. By
the subtraction construction, these are the negatives of the entries that lie immediately to the left of the
positions of $\tilde{c}_{i,j}$ on the right hand side of $B(n+1)$.  The $n+1$-st column remains unchanged from
$B(n+1)$, except that its $2n$-th entry is now $-\f{c_{n,n}}{z_{n,n}}=-x_{n,2n}$.

Recall that for $x=c_{i,j}$ or $z_{i,j}$, the notation $B^{(3)}_x$ refers to the square contiguous subblock of
$B^{(3)}$ whose top right entry is the instance of $x$ on the right hand side of the matrix $A$ from
(\ref{upperblock}), and whose bottom row is the bottom row of $B^{(3)}$.  Let $B^{(3)}_{\dagger x}$ denote the
further subblock of $B^{(3)}_x$ formed by  removing its last row and column.  Had we stopped subtracting the entry
to the right just before reaching the column containing $x$, the entries in $B^{(3)}_{\dagger x}$ would  equal the
corresponding ones in $B^{(3)}_x$, and the determinants of these blocks would agree because the bottom row of
$B^{(3)}_x$ would have all zeroes except for a $1$ in its last entry.  Since this partially-formed $B^{(3)}_x$
would be created from $B^{(4)}_x$ by column subtractions within its block, both determinants  would in turn equal
$\det B^{(4)}_x$, which equals  $\det B(n+1)_x$ divided by the $x_{i,2n+1}$ that fell in its columns. In
particular, for the entries $x=c_{n,i}$, $i<n$, that occur in the second to last row of $B(n+1)$, we have
\begin{equation}\label{bcindpf1}
\aligned
\tilde{c}^{(3)}_{n,i} \ \ & = \ \  -\det   B^{(3)}_{\dagger c_{n,i}}  \ \ = \ \ -\, \f{\det B(n+1)_{c_{n,i}}}{x_{i,2n+1}\,x_{i+1,2n+1}} \ \ = \ \ -\,\f{\det\ttwo{\tilde c_{n,i+1}}{\tilde c_{n,i}}{\tilde z_{n,i+1}}{\tilde z_{n,i}}}{x_{i,2n+1}\,x_{i+1,2n+1}} \\
&    = \ \ \f{c_{n,i}\,z_{n,i+1}}{x_{i,2n+1}\,x_{i+1,2n+1}} \ \  = \ \ \f{(x_{i,2n}\,x_{i,2n+1})(x_{i+1,2n+1})}{x_{i,2n+1}\,x_{i+1,2n+1}}  \ \ = \ \  x_{i,2n}\,.
\endaligned
\end{equation}

The passage from $B^{(3)}$ to $B^{(2)}=B^{(3)}m_2\i$ is created by dividing columns $i$ and $2n+1-i$ by $x_{i,2n}$,
for $i\le n$.  In particular it converts the entries in the second-to-bottom row from $-\tilde{c}^{(3)}_{n,i}$ to
$-1$,  $i\le n$.  It furthermore divides the determinant of any block by the product of the $x_{i,2n}$'s that
divided its columns.  The matrix $m_1\i$ has all $1$'s on its diagonal except for $-1$ in its $2n$-th position,
$-1$'s immediately below the diagonal except for the $n$-th and $2n$-th columns, and zeroes everywhere else.
Multiplying by it on the right  involves elementary column operations, and gives a matrix
$B^{(1)}=B^{(2)}m_1\i=B(n+1)m_4\i m_3\i m_2\i m_1\i$ which has $1$'s on the diagonal in its bottom two rows, and
zeroes below the diagonal in these rows.  If $u$ is the unit upper triangular $(2n+1)\times (2n+1)$ matrix that is
formed by replacing the last two columns of the identity matrix with the last two columns of $B^{(1)}$, then $u\i
B^{(1)}$ has the form $\ttwo{B'}{}{}{I_2}$, where $B'$ is the upper left $(2n-1)\times(2n-1)$ block of $B^{(1)}$.
The matrix $B'$ can also be constructed from the upper left $(2n-1)\times(2n-1)$ block of $B^{(2)}$ by successively
subtracting the entry immediately to the right of any entry, going from left to right, but skipping over the $n$-th
column.

We have hence reduced the proposition to showing that $B'$ equals $B(n)$, modulo a unit upper triangular matrix on
the left; equivalently, that $B(n)$ can be obtained from $B'$  by adding multiplies of lower rows to higher rows.
Let us first compare the structure of $B'$ to $B(n)$, starting with the left hand side. The entries
$\widetilde{c}^{(3)}_{n,i}$ from (\ref{bcindpf1}) are scaled to $1$ in $B^{(2)}$, and then subtracted from each
other in $B^{(1)}$, leaving zeroes in the first $n-1$ entries in the bottom row, and a $1$ in the $n$-th entry.  On
the right hand side, the subtraction operations in the multiplications by $m_2\i$ and $m_4\i$ scaled and moved the
entries $c_{i,i}$ two spots left.  As a result, the positions of the zeroes of $B(n)$ and $B'$ agree in all but the
$n$-th and $n+1$-st columns.  The $n$-th column has all zero entries, aside from a $1$ on the bottom.  By adding
multiplies of this row to the odd numbered rows, this column can be given the form of the middle column of $B(n)$.
However, that operation will move multiples of the bottom row into odd numbered rows on the right hand side, and
thereby spoil  a number of previously zero positions on the right hand side of the matrix, namely those an odd
number of positions above and two to the left of ones originally occupied by one of the  $c_{i,i}$, $i<n$. These
latter quantities migrated over two spots by consecutive subtractions, and divisions by nonzero quantities.
 Hence they are nonzero, and can be used as pivot columns to restore those zeroes spoiled by recreating the middle column.

These row operations do not affect the even numbered rows, the ones corresponding to $c_{i,j}$ entries. We thus get
a matrix of the same pattern as (\ref{upperblock}) in terms of the location of its zero entries, the equality
between the $c_{i,j}$ on the left and right hand sides, and the  relation $ c_j = \sum_{i\le j}c_{i,j}$.  Its
entries therefore match $B(n)$ if and only if they satisfy the analogous determinant relation
(\ref{determinantalproperty}) for $B(n)$.   These determinants are of contiguous subblocks whose bottom row is the
bottom row of the matrix, and are therefore unaffected by multiplication by unit upper triangular matrices on the
left.  Thus it suffices to verify the determinant property for $B'$. Let us then consider a variable $x_{c_{i,j}}$
or $z_{i,j}$ that occurs in $B(n)$, and the  subblock $B'_x$ of $B'$ formed from the same positions of $B(n)_x$.
This block consists of the same positions of the subblock $B(n+1)_x$ of $B(n+1)$, but without its last two rows and
columns -- what could be called $B(n+1)_{\dagger\dagger x}$ in the above notation.  Consider the intermediate block
$B^{(2)}_{\dagger x}$ of $B^{(2)}$, which has one more row and column than $B'_x$.  Had the subtraction procedure
used to form $B'_x$ stopped after altering its rightmost column, and not continued modifying columns to the right,
the block $B^{(2)}_{\dagger x}$ would have all zeroes on its bottom row except for a $-1$ entry in its bottom right
corner, and the matrix $B'_x$ in its upper left corner.
  The determinant of this hypothetical matrix is $\det B^{(2)}_{\dagger x}$ because it is formed by performing elementary column operations to $B^{(2)}_{\dagger x}$, and at the same time equals $-\det B'_x$.
It also equals $\det B^{(3)}_{\dagger x}$, divided by the $x_{i,2n}$ from its bottom row.
 We saw above that $\det B^{(3)}_{\dagger x}$ equals $\det B(n+1)_x$, divided by the $x_{i,2n+1}$ from its bottom row.     As we mentioned in the equivalent statement right after (\ref{determinantalproperty}), $\det B(n+1)_x$ is the product of all $x_{k,\ell}$ lying on the bottom right part of $B(n+1)_x$ (including the antidiagonal), multiplied by $(-1)^{(2n+2-\ell)(2n+1-\ell)/2}$.
 We conclude that $\det B'_x$, which differs by precisely the negative of the product of the $x_{k,\ell}$ that are in $B(n+1)_x$  but not in $B(n)_x$, satisfies the equivalent statement to (\ref{determinantalproperty}).
 \end{proof}

The formula for $b_{-,n}$ when $n=1$ fits the above pattern if one takes $b_{-,0}=1$.  Therefore the previous
proposition gives a product expansion for the lower triangular part of $ B$ as the product of $4n$ matrices.  Each
variable $x_{i,j}$ occurs in a unique such matrix.  In particular, the lower triangular part $b_{-}$ of  $B$ is
linear in each  $x_{i,j}$.  If $C(x_{i,j})$ denotes the  coefficient matrix of $x_{i,j}$ in $ B$, then it follows
from this factorization that
\begin{equation}\label{indpunchforlater}
    b_{-}\i C(x_{i,j}) x_{i,j} \ \ \text{is a lower triangular matrix with entries in~}\Z[{\mathcal S},{\mathcal S}\i]\,,
\end{equation}
where $\mathcal S=\{x_{k,\ell}\mid \ell>j\}$.  This is because $C(x_{i,j})x_{i,j}$ is itself the same product of
$4n$ matrices, but with the appropriate matrix $m_2$ or $m_4$ that contains $x_{i,j}$ instead altered to have
zeroes in all positions other than its two instances of $x_{i,j}$; the factors containing the variables not in
$\mathcal S$ cancel out in the product $b_{-}\i C(x_{i,j})x_{i,j}$.

It is also possible to read off the diagonal entries of $b_{-}$ from this representation of $b_{-}$ as a product of
$4n$ matrices.  Indeed, their diagonal elements are all either $1$, $-1$ (coming from the second to last entries of
the $m_1$ matrices), $x_{i,2j}$ (in entries $i$ and $2j+1-i$ of the $m_2$ matrices), or $x_{i,2j+1}$ (in entries
$i$ and $2j+2-i$ of the $m_4$ matrices).   Combining this with proposition~\ref{ext2reductionshort}, we get the
explicit formula
\begin{equation}\label{explformforwld}
    \gathered
\!\!\!\!\!\!   \!\!\!\!\!\!\!\! w_{\l,\d}\(\sigma \ttwo{C}{Z}{}{C}\ttwo{f_1}{}{}{f_2}\)  \ \ = \ \
e\(\sum_{2i\le j \le 2n-2}x_{i,j}-\sum_{i<n}x_{i,2n-1}\) \, \times
\\
\ \ \ \ \ \ \ \  \ \times \   \kappa_2 \prod_{2i\le j \le 2n-1}|x_{i,j}|^{-\l_i-\l_{j+1-i} +(2n+1)-i-(j+1-i)}\sgn(x_{i,j})^{\d_{i}+\d_{j+1-i}}\,,
    \endgathered
\end{equation}
where $\kappa_2= (-1)^{\d_2+\d_4+\cdots+\d_{2n-2}}$, valid when each $x_{i,j}\neq 0$.

\section{Local Integrals}\label{sec:localintegrals}

We now return to the calculation of  (\ref{tshrink2}) from the end of section~\ref{unfolding}. The measure factor
(\ref{dmu}) is equal to
\begin{equation}\label{Cmeaschange}
\aligned
   d\mu
\ \ & = \ \
\(\prod_{j\,=\,1}^{n-1}|c_{j,j}|^{s+2j-2n-1}\sgn(c_{j,j})^\eta
\) \prod_{1\le j \le i < n}d c_{i,j}\,dz_{i,j} \\
& = \ \
\(\prod_{j\,=\,1}^{n-1}|y_{j,2j}|^{s+2j-2n-1}\sgn(y_{j,2j})^\eta
\) \prod_{2i\le j < 2n }d y_{i,j} \\ & = \ \
 \prod_{i\,=\,1}^{n-1}\prod_{j=2i}^{2n-1} |x_{i,j}|^{s+j-2n-1}\sgn(x_{i,j})^\eta\,
 d x_{i,j}\,.
\endaligned
\end{equation}
Let us  consider the pointwise limit of the  integral (\ref{tshrink2}) as $t\rightarrow\infty$.  In this limit
$W_{\infty,t}\(\sigma\ttwo{C}{Z}{}{C}\ttwo{f_1}{}{}{f_2}\)$ tends to (\ref{explformforwld}), and the integral
formally becomes
\begin{equation}\label{sec5eq1inxs}
\gathered
 \!\!\!\!\!\!\!\!\!\!\!\!  \!\!\!\!\!\!\!\!\!\!\!\!   \!\!\!\!\!\!\!\!\!\!\!\!   \!\!\!\!\!\!\!\!\!\!\!\!   \!\!\!\!\!\!\!\!\!\!\!\!  \kappa_1'\,\kappa_2\,
\int_{\R^{n(n-1)}}
e\(\sum_{2i\le j \le 2n-2}x_{i,j}-\sum_{i<n}x_{i,2n-1}\)\ \times \\ \ \ \ \ \ \ \ \ \ \ \ \ \ \ \ \ \ \ \ \  \times \ \prod_{2i\le j \le 2n-1}|x_{i,j}|^{s-\l_i-\l_{j+1-i}-1}\sgn(x_{i,j})^{\d_i+\d_{j+1-i}+\eta}\,dx_{i,j}\,.
\endgathered
\end{equation}
Again formally, this is a product  of $n(n-1)$ integrals of the form (\ref{eisen3}), which -- were this possible to
legitimize  -- would  express $\Psi_\infty(s,w_{\l,\d})$ as the product
\begin{equation}\label{recall6}
\Psi_\infty(s,w_{\l,\d}) \ \ = \ \
\kappa_1'\,   \kappa_2\,\kappa_3\, \
    \prod_{\stackrel{\scriptstyle{1 \le i < j \le 2n}}{i+j\le 2n}}
    G_{\d_i+\d_j+\eta}(s-\l_i-\l_j)\,,
\end{equation}
where $\kappa_3=(-1)^{(n-1)\eta + \d_n + \d_{2n}+\e+n\eta} = (-1)^{\eta + \d_n + \d_{2n}+\e}$ comes from the signs
in front of the $x_{i,2n-1}$ terms in the exponential factor.  The overall sign is
\begin{equation*}\label{lasterlineadded}
\aligned
\kappa \  &  =  \ \kappa_1'\kappa_2\kappa_3 \  =  \ (-1)^{\d_2+\cdots+\d_{n-1} + \d_{2n}+\e+n\eta }  (-1)^{\d_2+\d_4+\d_6+\cdots + \d_{2n-2}}  (-1)^{\eta + \d_n + \d_{2n}+\e} \\ & =  \ (-1)^{(n+1)\eta +\sum_{j=2}^n\d_j+\sum_{j=1}^{n-1}\d_{2j}}\,,
\endaligned
\end{equation*}
where $\kappa_1'$ and  $\kappa_2$ are defined just after (\ref{unfoldingprop31newz}) and (\ref{explformforwld}),
respectively.

In this paragraph we shall assume the validity of
 (\ref{recall6}), and derive some of its consequences.
  If $S$ denotes the set of primes at which either $\pi$ or $\chi$ are ramified and
 ${\mathcal G}(s)$ denotes the product
\begin{equation}\label{halfcomplete}
   {\mathcal G}(s) \ \ = \ \  \prod_{\stackrel{\scriptstyle{1 \le i < j \le 2n}}{i+j \le
    2n}} G_{\d_i+\d_j+\eta}(s-\l_i-\l_j)\,,
\end{equation} (\ref{localunramcalc}) and proposition~\ref{unfoldingprop}  imply that
\begin{equation}\label{whatsentire}
\gathered
   P(\tau,E(s)) \ \ = \ \  \kappa \ {\mathcal G}(s)\,  L^S(s,\pi,Ext^2\otimes \chi)   \,   {\prod}_{p\,\in\,S}\Psi_p(s,W_p,\Phi_p)\,.
    \endgathered
\end{equation}
In particular \thmref{thmpairing} shows that
\begin{equation}\label{recall75}
\gathered
    \text{expression (\ref{whatsentire}) is holomorphic for
     $s\in \C-\{1\}$} \\ \text{ and has at most
a simple pole at $s=1$.}
\endgathered
\end{equation}
In the situation where $\pi$ corresponds to a full level form, $\e=\eta\equiv 0\imod 2$, $\chi$ is trivial, and
$S=\{\}$, we may combine this with (\ref{missingsign}) to obtain
\begin{equation}\label{recall8}
  P(\widetilde{\tau},E(1-s)) \ \  = \ \ (-1)^{\d_2+\d_4+\cdots+\d_{2n}}
\,\widetilde{\mathcal G}(1-s)
    \,L(1-s,\widetilde{\pi},Ext^2) \,,
\end{equation}
where
\begin{equation}\label{halfcompletetilde}
\aligned
 \widetilde{\mathcal G}(s) \ \ & = \ \  {\prod}_{\stackrel{\scriptstyle{1 \le i < j \le 2n}}{i+j \le
    2n}}\  G_{\d_{2n+1-i}+\d_{2n+1-j}}(s+\l_{2n+1-i}+\l_{2n+1-j}) \\
    & = \ \  {\prod}_{\stackrel{\scriptstyle{1 \le i < j \le 2n}}{i+j >
    2n+1}} \ G_{\d_i+\d_j}(s+\l_{i}+\l_{j})
\endaligned
\end{equation}
is the analog of (\ref{halfcomplete}) for $\tilde \pi$ (cf. (\ref{ab12})). Since $\e=\eta\equiv 0\imod 2$, the sign
$\kappa$ simplifies to $(-1)^{\sum_{j=2}^n\d_j+\sum_{j=1}^{n-1}\d_{2j}}$, while the analogous sign for
$\widetilde\tau$ on the other side of the functional equation,
$(-1)^{\sum_{j=2}^n\d_{2n+1-j}+\sum_{j=1}^{n-1}\d_{2n+1-2j}}$, is also equal to $\kappa$ (cf.~(\ref{pair8})).
Inserting into (\ref{penufefulllevel}), we get the functional equation
\begin{equation}\label{recall9}
\aligned
  (-1)^{\sum_{j\le n}\d_j+\d_{2j}} \,  {\mathcal G}(s)\,L(s,\pi,Ext^2)  \ \ = \qquad\qquad&\qquad  \\
    \prod_{j=1}^n G_{\d_{n+j}+\d_{n+1-j}}(1-s+\l_{n+j}+\l_{n+1-j})  & \,
\widetilde{\mathcal G}(1-s)\,L(1-s,\widetilde{\pi},Ext^2)\,.
\endaligned
\end{equation}
Both sides of this functional equation are entire. Using the identity $G_\d(s) G_\d(1-s) = (-1)^\d$ (a consequence
of  (\ref{eisen3b})), the $G_\d$-functions on the right can be moved to the left.  This  results in the cleaner
statement
\begin{equation}\label{assumedfeagain}
    \(\prod_{1\le i < j \le
     2n} G_{\d_i+\d_j}(s-\l_i-\l_j)\)\,L(s,\pi,Ext^2) \ \ = \ \
     L(1-s,\tilde\pi,Ext^2)\, ;
\end{equation}
in the next section, both sides will be seen to be holomorphic on $\C-\{0,1\}$, with at most possible simple poles
at 0 and 1.

The rest of this section is devoted to proving (\ref{recall6}), thereby making the above calculation rigorous.
 The integral in
 (\ref{tshrink2})
  is that of a smooth, integrable function, and so its value is unchanged if we restrict the range of integration to the dense open subset $\mathcal D$ of $\R^{n(n-1)}$ on which none of the $x_{i,j}$ vanish -- this is legitimate because  $L^1$ integrals are independent of parametrization. The integral is then equal to
\begin{equation}\label{rigamorale}
    \kappa_1' \, \int_{\mathcal D} a_t^{\rho-\l} \,W_{\infty,1}\!\!\(\sigma\ttwo{C(x)}{Z(x)}{}{C(x)}\ttwo{f_1}{}{}{f_2}a_t\i\)d\mu\,.
\end{equation}
According to propositions~\ref{ext2reductionshort} and  \ref{lowertrianpartprop},
$\sigma\ttwo{C(x)}{Z(x)}{}{C(x)}\ttwo{f_1}{}{}{f_2}a_t\i$ can be  uniquely written as a product
$u(x)b_{-}(x)a_t\i$, where $u(x)$ is unit upper triangular and $b_{-}(x)\in B_{-}(\R)$ is a $2n\times 2n$ lower
triangular matrix whose last row has all zeroes except for a 1 in the last entry.  As in the calculation of
(\ref{explformforwld}), the Whittaker function $W_{\infty,1}$ transforms on the left under $u(x)$ by the character
$e(\sum_{2i\le j \le 2n-2}x_{i,j}-\sum_{i<n}x_{i,2n-1})$. The remaining part,
$a_t^{\rho-\l}W_{\infty,1}(b_{-}(x)a_t\i)$ tends to $
   \kappa_2
\prod_{j\,=\,1}^{n-1} \prod_{k=2j}^{2n-1} \\ |x_{j,k}|^{2n-k-\l_j-\l_{k-j+1}}\sgn(x_{j,k})^{\d_j+\d_{k-j+1}} $ as
$t\rightarrow \infty$ (cf.~the comments at the end of sections~\ref{unfolding} and \ref{slicingsec}).

The integral  (\ref{eisen3}) is only conditionally convergent for $0<\Re{s}<1$. However, one can integrate by parts
$N$ times, $N\geq 0$, to give meaning to the formula as an absolutely convergent integral on any vertical strip $0
< \Re s < N$. To see this, choose $\psi \in C^\infty_c(\R)$, with $\psi(x) \equiv 1$ near $x=0$. Then
\begin{equation}\label{makesenseofgds}
\gathered
  \!\!\!\!\!  \!\!\!\!\! \!\!\!\!\! \!\!\!\!\! \!\!\!\!\! \!\!\!\!\! \!\!\!\!\! \!\!\!\!\! \!\!\!\!\! \!\!\!\!\! \!\!\!\!\! \!\!\!\!\! \!\!\!\!\! \!\!\!\!\! \!\!\!\!\! \!\!\!\!\! \!\!\!\!\!
   G_\d(s) \ \ = \ \ \int_{\R} e(x)|x|^{s-1}\sgn(x)^\d\,\psi(x)\,dx \ + \\ \ \ \ \ \ \ \ \ \ \ \ \ \ \ \ \ \ \ \ \ \ + \ \(\f{-1}{2\pi i}\)^N\int_{\R}e(x)\,\smallf{d^N}{dx^N}\!\(|x|^{s-1}\sgn(x)^\d(1-\psi(x))\)\,dx\,,
    \endgathered
\end{equation}
since each differentiation improves the decay rate at infinity of the second integrand by a power of $|x|^{-1}$,
while the first integral remains absolutely convergent. Our strategy is to divide $\mathcal D$ into $2^{n(n-1)}$
subsets using a smooth partition of unity, to integrate each of the resulting integrals by parts using (5.12), and
to argue that each of these integrals converges absolutely to its pointwise limit.  If $\mathcal V$ is any subset
of the variables $x_{i,j}$, let
\begin{equation}\label{psiV}
    \psi_{\mathcal V}(x) \ \ = \ \ \prod_{x_{i,j}\notin {\mathcal V}} \psi(x_{i,j}) \, \times \, \prod_{x_{i,j}\in {\mathcal V}}(1-\psi(x_{i,j})).
\end{equation}
Thus (\ref{rigamorale}) is equal to
\begin{equation}\label{rigatoni}
    \kappa_1' a_t^{\rho-\l} \sum_{{\mathcal V}\,\subset\, \{x_{i,j}\}}
\int_{{\mathcal D}}e(\sum_{2i\le j \le 2n-2}x_{i,j}-\sum_{i<n}x_{i,2n-1})
W_{\infty,1}(b_{-}(x)a_t\i) \psi_{\mathcal V}(x) d\mu.
\end{equation}
For each summand, integrate by parts $N_{i,j}$ times as in (\ref{makesenseofgds}) in the variables $x_{i,j}\in
\mathcal V$, i.e., integrate up the exponential in $x_{i,j}$ to get a constant multiple of itself, and apply
$\frac{\partial}{\partial x_{i,j}}$ to the rest of the expression to its right.    The derivatives of the partition
of unity and the powers of the $|x_{k,\ell}|$ are straightforward, and play the same role as they do in
(\ref{makesenseofgds}).  The following proposition governs the differentiation of the Whittaker function.

\begin{prop}\label{whitdiffprop}
Fix a variable $x_{i,j}$,  let $\mathcal S=\{x_{k,\ell}\mid \ell>j\}$, and let $W$ be a smooth Whittaker function
for the principal series representation $V_{\l,\d}$. Then the derivative $x_{i,j}\frac{\partial}{\partial x_{i,j}}$
of $W(b_{-}(x)a_t\i)$ is a finite sum of polynomials in $\mathcal S$, $\mathcal S\i$, and $t\i$, times other smooth
Whittaker functions for $V_{\l,\d}$, evaluated at $b_{-}(x)a_t\i$.
\end{prop}
\begin{proof}
Let $C=C(x_{i,j})$ denote the coefficient matrix of $x_{i,j}$ in $b_{-}(x)$, which is an affine
 function of $x_{i,j}$.  For $s$ small, we may equate $b_{-}(x)+s C$, the translation in $x_{i,j}$ by $s$, with $b_{-}(x) \exp(sb_{-}(x)\i C)+O(s^2)$.  Thus
\begin{equation}\label{whitdiffprop1}
\f{\partial}{\partial x_{i,j}}W(b_{-}(x)a_t\i) \ \   = \ \ \left.\f{d}{ds}\right|_{s=0}
W\(b_{-}(x)\,a_t\i e^{s\,Y}\)
\end{equation}
is the right Lie algebra derivative of $W$ by $Y= a_tb_{-}(x)\i Ca_t\i$, evaluated at $b_{-}(x)a_t\i$. It follows
from (\ref{indpunchforlater}) that $a_tb_{-}(x)\i Ca_t\i$ can be expanded as a sum
$x_{i,j}\i\sum_{k>\ell}p_{k,\ell}E_{k,\ell}$, where the $E_{k,\ell}$ range over strictly lower triangular matrices
in ${\frak{gl}}(2n,\R)$ which have zeroes except for a $1$ in their $(k,\ell)$-th entry, and the $p_{k,\ell}$ are
polynomials in $\mathcal S$, $\mathcal S\i$, and $t\i$. The result follows because  Lie algebra derivatives of $W$
by fixed matrices in ${\frak{gl}}(2n,\R)$ are themselves smooth Whittaker functions for $V_{\l,\d}$.
\end{proof}

In particular, the differentiation involved in the integration by parts in $x_{i,j}$ decreases the exponent of
$|x_{i,j}|$, at the cost of altering the exponents of the $x_{k,\ell}$ for which $\ell>j$.  For this reason we
implement the integrations by parts starting with lower values of $j$, so that these alterations may be taken into
account.   The complex variable $s$ enters into the calculation through $d\mu$ in (\ref{Cmeaschange}).  Let us
temporarily decouple the $s$'s that occur there, by replacing the one occurring in the exponent of $|x_{i,j}|$ by
$s_{i,j}$.   The decreased exponents for $x_{k,\ell}\notin\mathcal V$ can be compensated for by increasing $\Re
s_{k,\ell}$, whereas the increased exponents for $x_{k,\ell}\in \mathcal V$ can be decreased increasing the value
of $N_{k,\ell}$; in both cases, we can arrange that the exponents are in the range of absolute convergence, with an
arbitrary amount of room to spare.

To finish the argument, we recall Jacquet's holomorphic continuation of Whittaker functions (see
\cite[Theorem~7.2]{chm}), which allows us to establish the result of our calculation by verifying it only for  $\l$
in a small open set.  The following proposition gives bounds for Whittaker functions which then establish dominated
convergence on each piece of (\ref{rigatoni}).
 Indeed, the expression of the differential operator $x_{i,j}\frac{\partial}{\partial x_{i,j}}$ in terms of Lie algebra derivatives in the last proof shows that the limit of $DW_{\infty,t}(b_{-}(x))$ equals $Dw_{\l,\d}(b_{-}(x))$, for any differential operator $D$ which is a polynomial of the $\f{\partial}{\partial x_{i,j}}$.
 The integrals then converge to products of (\ref{makesenseofgds}) evaluated at $s_{i,j}-\l_i-\l_j$, which equals the product in  (\ref{recall6}) when each $s_{i,j}$ is specialized to $s$.

\begin{prop}\label{whitboundprop}
Suppose that $V_{\l,\d}$ is an irreducible principal series representation of $GL(2n,\R)$, and that
$\l=(\l_1,\ldots,\l_{2n})\in \C^{2n}$, $\|\l\|\le 1$, has distinct entries  satisfying {\rm 1)} $\Re \langle
\rho,\l\rangle \ge \Re \langle \rho,w\l\rangle$ for all $w$ in the Weyl group $\Omega$ of $GL(2n,\R)$,  {\rm 2)}
the entries $(w \l)_m$ of any Weyl translate of $\l$ satisfy $\Re (w \l)_{m-1}\le \Re  (w \l)_m+1$ for all $2\le m
\le 2n$, and {\rm 3)} $\langle \l,\a\rangle \notin \Z$ for any root $\a$.  Let $W$ be a smooth Whittaker function
for $V_{\l,\d}$.  Then $|W(b_{-}(x)a_t\i)|$ is bounded by $|a_t^{\l-\rho}|$ times a polynomial in the
$|x_{i,j}|^{\pm 1}$ whose degree depends only on $n$.
\end{prop}
\begin{proof}
That such an open neighborhood of $\l$ exists is clear:~take a sufficiently small open neighborhood of a small
positive multiple of $\rho$. According to \cite[\S4.2-\S4.3]{jsextsq}, a smooth Whittaker function on a diagonal
matrix $d=\operatorname{diag}(d_1,\ldots,d_{2n})$ can be expanded as a sum
\begin{equation}\label{Waexpn}
    W(d) \ \ = \ \ \sum_{w\,\in\,\Omega} d^{\rho-w\l}f_w(\smallf{d_1}{d_2},\smallf{d_2}{d_3},\ldots,\smallf{d_{2n-1}}{d_{2n}})\,,
\end{equation}
 in which each $f_w$ is a Schwartz function on $\R^{2n-1}$.  Moreover, these Schwartz functions are uniformly bounded if $W$ is replaced by any of its right translates ranging over any fixed compact subset of $GL(2n,\R)$.
Actually the  exponents are not given in \cite{jsextsq}, though Casselman has pointed out that the argument there
indicates that they are controlled by the Jacquet module, and hence match (\ref{Waexpn}); in any event, they can be
deduced from the asymptotic expansions of smooth Whittaker functions proved in \cite[\S6]{goodmanwallach}.
(Without condition 3) the Schwartz functions would need to be multiplied by powers of logarithms.)

The matrix $b_{-}(x)$ has zero entries in its last row, aside from the last entry which is $1$.  It can be
decomposed as $ru_{-}$, where $r$ is diagonal, and $u_{-}$ is unit lower triangular and has the same bottom row as
$b_{-}(x)$. Let $n'a'k'$ be the Iwasawa decomposition of $a_tu_{-}a_t\i$, with $n'$ unit upper triangular,
$a'=\operatorname{diag}(a_1',a_2',\ldots,a_{2n}')$ diagonal, and $k'\in O(2n,\R)$.  In general, the product of the
last $m$ diagonal entries of the diagonal factor in the Iwasawa decomposition must equal the norm of the $m$-th
exterior power of the bottom $m$ rows of the matrix (simply because both share the same invariance properties that
cause them to be determined on diagonal matrices, and both agree on diagonal matrices).  In particular we have the
formula $a'_m=\f{p_m}{p_{m+1}}$, where $p_m$ is the norm of the $m$-th exterior power of the bottom $m$ rows of
$a_tu_{-}a_t\i$, and $p_{2n+1}$ is understood to be 1.  Each exterior power is a vector whose components are
determinants of square subblocks of the bottom rows, one of which -- the one coming from the rightmost square
subblock -- always has determinant one.  Thus each $p_m\ge 1$, and in  particular, $p_1=p_{2n}=1$.

We conclude that $|W(b_{-}(x)a_t\i)|=|W(ra_t\i n'a'k')|=|W(ra_t\i a'k')|$ is bounded by $\sum_{\omega\in
\Omega}|(ra_t\i a')^{\rho-\omega\l}|$.  Because of  assumption 1), each term $|a_t^{\omega \l}|=t^{\Re \langle
\rho,\omega \l \rangle} \le t^{\Re \langle \rho,\l\rangle} = |a_t^{\l}|$.
 Because of assumption 2) and the explicit formula for the entries of $a'$, we have that  $|a'|^{\rho-\omega \l}=\prod_{m=2}^{2n-1}|p_m|^{\Re(\rho_{m}-\rho_{m-1}-(\omega \l)_{m}+(\omega\l)_{m-1})}$ is a product of the $p_m$ to nonpositive powers, and hence is bounded by 1.  The result now follows from  proposition~\ref{lowertrianpartprop}, which among other things asserts that the diagonal entries of $b_{-}(x)$ and hence $r$ are products of the $x_{i,j}$ up to sign.
\end{proof}

\section{Functional Equation}\label{sec:functionalequation}

The computation of the pairing in the last section, combined with its functional equation (\ref{adelicpenufe}),
gives a functional equation for   the exterior square $L$-functions relating $s$ and $1-s$, and involving an
explicit ratio of products of Gamma functions.   In this section we shall  show this ratio agrees with Langlands'
formulation \cite{langlandsrealgroups} of the functional equation (later proved in \cite{shahidi}).
 Knowledge of this ratio
 will also be used to establish the full holomorphy in the following section, and also to give a new proof of the functional equation of $L(s,\pi,Ext^2)$ when $\pi_p$ is unramified for all primes $p<\infty$ (proposition~\ref{samefe}).

 Langlands' formula for the Gamma factors involves the description of  $\pi_\infty$ as a parabolically induced representation, whereas ours is in terms of its Casselman embedding.  We thus begin this section by relating the two.  In order to narrow the scope of discussion, we recall that $\pi_\infty$ is necessarily both unitary and generic.  The classification of such representations of $GL(n,\R)$ has long been known to experts; we summarize it here and refer to \cite[Appendix~A.1]{mirabolic} for the derivation from the results of \cite{tadic,voganlarge}.

Consider the self-dual, square integrable (modulo the center) representations of $GL(n,\R)$.  These are precisely
the following representations:~the trivial representation of $GL(1,\R)$; the sign representation $\sgn(\cdot)$ of
$GL(1,\R)$; and the discrete series representations $D_k$ of $GL(2,\R)$ (corresponding to holomorphic forms of
weight $k \ge 2$).  Each such representation has a twist  $\sigma[s]:=\sigma\otimes|\det(\cdot)|^s$ by a central
character.
  If $P$ is the standard parabolic subgroup of $GL(n,\R)$ associated to a partition $n=n_1+\cdots+n_r$ of $n$ with each $n_i\le 2$, and $\sigma_i$ one of these representations of $GL(n_i,\R)$,  then the tensor product of twists of the $\sigma_i$ can be extended to $P$ from its Levi component by letting the unipotent radical act trivially.  The parabolic induction $I(P;\sigma_1[s_1],\ldots,\sigma_r[s_r])$ of this representation of $P$ is a representation of $GL(n,\R)$ which is normalized to be unitary when each $s_i$ is purely imaginary.  By twisting this representation -- or alternatively shifting each $s_i$ -- we may assume it has a unitary central character.  It is unitary, irreducible, and generic precisely when the following two conditions are met:
\begin{equation}\label{genericunitarydual}
    \text{a) the
    multisets $\{\sigma_i[s_i]\}$ and
    $\{\sigma_i[-\overline{s}_i]\}$  are equal, \ \
~~~and~~~b)}\ \
|\Re{s_i}|<1/2\,.
\end{equation}
Conversely, all generic unitary irreducible representations of $GL(n,\R)$ are obtained this way. The induction data
may be freely permuted, i.e., the induced representations
      $I(P;\sigma_1[s_1],\ldots,\sigma_r[s_r])$ and
         $I(P^\tau;\sigma_{\tau(1)}[s_{\tau(1)}],\ldots,\sigma_{\tau(r)}[s_{\tau(r)}])$ are equal, where  $\tau$ is any permutation on $r$ letters, and $P^\tau$ is the standard parabolic corresponding to the partition $n=n_{\tau(1)}+\cdots + n_{\tau(r)}$ of $n$.  Moreover, the multiset  $\{\sigma_i[s_i]\}$  is uniquely determined up to permutation.

The Casselman embeddings of $I(P;\sigma_1[s_1],\ldots,\sigma_r[s_r])$ into are completely described in
\cite[Appendix~A.1]{mirabolic}.  The next proposition describes one among these which will be particularly useful
in arguing the full holomorphy of $\Lambda(s,\pi,Ext^2\otimes\chi)$ in the next section.

\begin{prop}\label{niceembedding}  Let $n$ be even and
 $\pi_\infty=I(P;\sigma_1[s_1],\ldots,\sigma_r[s_r])$ be   a generic
unitary irreducible representation of $GL(n,\R)$. By permuting  if necessary, arrange that $n_1=\cdots =n_{r_1}=1$
and $n_{r_1+1}=\cdots=n_r=2$, where $r=r_1+r_2$   and $r_1$ is even. Write $\sigma_i=\sgn(\cdot)^{\e_i}$ for $1\le
i \le r_1$, and $\sigma_{r_1+i}=D_{k_i}$ for $1\le i \le r_2$. By again permuting the induction data if necessary,
arrange that
$$\Re{s_1} \, \le \,  \Re{s_2} \, \le  \, \cdots
    \, \le \,  \Re{s_{r_1}}\, \text{~and~}\,\Re{s_{r_1+1}} \, \le  \, \Re{s_{r_1+2}} \, \le \,  \cdots \  \le
\,  \Re{s_{r_1+r_2}}\,.$$
  Then
$\pi_\infty'$ embeds into the principal series $V_{\l,\d}$ with parameters
\begin{equation}\label{lamniceembed}
   \gathered
\l  \ \ = \ \   \(-s_1  \, , \, \ldots \, , \,- s_{r_1/2} \, , \,
\textstyle{-s_{r_1+1}-\f{k_1-1}{2}\, , \, -s_{r_1+1}+\f{k_1-1}{2}}
\, ,
  \ldots , \qquad   \qquad  \right. \\
\left.   \qquad  \qquad  \ldots, \,
\textstyle{-s_{r_1+r_2}-\f{k_{r_2}-1}{2},-s_{r_1+r_2}+\f{k_{r_2}-1}{2}} \, ,
\, -s_{r_1/2+1} \, , \, \ldots \, , \, -s_{r_1}\)
   \endgathered
\end{equation}
and
\begin{equation}\label{deltaniceembed}
   \d \ \ =  \ \
\(\e_1   \, , \, \ldots \, , \, \e_{r_1/2}\, , \,
k_1,0, \, k_2,0,\, \ldots, \, k_{r_2},0, \, \e_{r_1/2+1} \, , \,
\ldots \, , \,  \, \e_{r_1}
 \)\, .
\end{equation}
\end{prop}

\begin{proof}
The assumptions along with  (\ref{genericunitarydual}a) imply that
\begin{equation}\label{realnegequal}
\gathered
\Re{s_i} \  =  \ -\Re{s_{r_1+1-i}}\ \text{ for } i \le  r_1 \\ \text{ ~and both} \ \ \  \Re s_{r_1+i}\   =  \  -\Re s_{r_1+r_2+1-i}\,,\  k_i \ =  \ k_{r_2+1-i}\ \text{ for }i\le r_2\,.
\endgathered
\end{equation}
Because of the independence of permutation,  $\pi_\infty$ is equivalent to the induced representation
$$I(Q;\sigma_1[s_1],\!\ldots\!,\sigma_{\f{r_1}{2}}[s_{\f{r_1}{2}}], \sigma_{r_1+1}[s_{r_1+1}],\!\ldots\!,
\sigma_{r_1+r_2}[s_{r_1+r_2}],\sigma_{\f{r_1}{2}+1}[s_{\f{r_1}{2}+1}],\!\ldots\!,\sigma_{r_1}[s_{r_1}]).$$  The
result now follows immediately from the embedding description in \cite[(A.1) and (A.2)]{mirabolic}.
\end{proof}

We now  summarize Langlands' prediction of $L_\infty(s,\pi,Ext^2\otimes\chi)$ in \cite{langlandsrealgroups}. The
Gamma factors which accompany an $L$-function in its functional equation are always products of shifts of
\begin{equation}\label{gammarc}
    \G_\R(s) \, = \, \pi^{-s/2}\,\G(s/2) \ \ \ \ \text{and} \ \ \ \
     \G_\C(s) \, = \, 2\,(2\pi)^{-s}\,\G(s) \, = \, \G_\R(s)\,\G_\R(s+1)\,.
\end{equation}  Write
  $\pi_\infty = I(P;\sigma_1[s_1],\ldots,\sigma_r[s_r])$, where $n=r_1+2r_2$,  $n_i=1$ for $i\le r_1$, and $n_i=2$ for $r_1 < i \le r_1+r_2$.  Furthermore choose $\e_{ik}$ and $\e_j' \in \{0,1\}$ to be congruent to $\e_i+\e_k$
and $k_j$ modulo 2, respectively. Define the four products
\begin{equation}\label{fivelpieces}
\aligned
L(s,\Pi_2) \ \ & = \ \ \  \, \  {\prod}_{j=1}^{r_2} \G_\R(s+2s_{r_1+j}+ \e_j')\, ,  \\
L(s,\Pi_3)\ \ & = \ \  \  \ \, {\prod}_{\stackrel{\scriptstyle{i \le
r_1}}{j \le r_2}}
\G_\C(s+s_i+s_{r_1+j}+\textstyle{\f{k_j-1}{2}}) \, ,  \\
L(s,\Pi_4)\ \  & = \ \  {\prod}_{1 \le i < k \le
r_1}\G_\R(s+s_i+s_k+\e_{ik}) \, ,
\text{~~~~~~and}  \\
L(s,\Pi_5)\ \  & = \ \ {\prod}_{1 \le j < \ell \le r_2}
\(\G_\C(s+s_{r_1+j}+s_{r_1+\ell}+\textstyle{\f{k_j+k_\ell-2}{2}})
\ \times \right. \\ & \qquad\qquad\qquad\qquad  \left. \times \
\G_\C(s+s_{r_1+j}+s_{r_1+\ell}+\textstyle{\f{|k_j-k_\ell|}{2}})
\)  . \\
\endaligned
\end{equation}
The numbering here is chosen to be consistent with  \cite[(A.18)]{mirabolic}, where it is shown that
$L_\infty(s,\pi,Ext^2) =L(s,\Pi_2)L(s,\Pi_3)L(s,\Pi_4)L(s,\Pi_5)$ is the product of these factors. A similar
formula for $L(s,\pi,Ext^2\otimes\chi)$ reads as follows (\cite[(A.20)]{mirabolic}):
 \begin{equation}\label{ext2tenschiisobar2}
 L_\infty(s,\pi,Ext^2\otimes\chi) \ \ = \ \ L(s,\Pi_2\,\otimes\,\sgn^\eta)\,L(s,\Pi_3)\,L(s,\Pi_4\,\otimes\,\sgn^\eta)\, L(s,\Pi_5)\,,
 \end{equation}
 where $\eta$ is the parity of $\chi_\infty=\sgn(\cdot)^\eta$
(see  (\ref{IfromPhi}) and (\ref{pair8})),
 \begin{equation}\label{twotwistedpieces}
    \aligned
    L(s,\Pi_2\otimes\sgn^\eta) \ \ & = \ \ \  \, \  {\prod}_{j=1}^{r_2} \G_\R(s+2s_{r_1+j}+ \e_{j\eta}')\, ,  \\
    L(s,\Pi_4\otimes\sgn^\eta)\ \  & = \ \  {\prod}_{1 \le i < k \le
r_1}\G_\R(s+s_i+s_k+\e_{ik\eta}) \, , \\
    \endaligned
 \end{equation}
 and $\e_{j\eta}'$ and $\e_{ik\eta}\in\{0,1\}$ are congruent to $\e_j'+\eta\equiv k_j+\eta$
and $\e_{ik}+\eta\equiv \e_i+\e_k+\eta \pmod 2$, respectively.

\begin{prop}\label{samefe}  With the notation as above, one has
that
\begin{equation}\label{neededratio}
    \f{L_\infty(s,\pi,Ext^2)}{\omega\,  L_\infty(1-s,\tilde{\pi},Ext^2)} \ \
    = \ \ {\prod}_{1\le i < j \le
     2n} G_{\d_i+\d_j}(s-\l_i-\l_j)\,,
\end{equation}
where
\begin{equation}\label{omega}
  \omega \ \ = \ \ {\prod}_{1\le i <k \le r_1}i^{-\e_{ik}} \, {\prod}_{j\le r_2}
  i^{\,
  k_j(2j-n)\, - \, \e_j' }  \,.
\end{equation}
More generally, $L_\infty(s,\pi,Ext^2\otimes\chi)/L_\infty(1-s,\tilde{\pi},Ext^2\otimes\chi\i)$ is equal to  a
fourth root of unity times  $\prod_{1\le i<j\le 2n}G_{\d_i+\d_j+\eta}(s-\l_i-\l_j)$.
\end{prop}

The dual $L$-factor $L_\infty(s,\tilde\pi,Ext^2\otimes \chi)$ equals $\overline{L(\bar{s},\pi,Ext^2\otimes\chi)}$,
i.e., the $L$-factor produced by the above recipe, but with each $s_i$ replaced by $\overline{s_i}$ (or
equivalently by $-s_i$, in light of (\ref{genericunitarydual}a)).

\begin{proof}
The dual representation $\pi'_\infty$ is shown in  \cite[(A.1)~and~(A.2)]{mirabolic} to embed into the principal
series with parameters
\begin{equation}\label{lamfe1}
\gathered
\l \  \ = \ \    \textstyle{(-s_1,\,-s_2,\,\ldots,-s_{r_1},\,-s_{r_1+1}-\f{k_1-1}{2},\,
   -s_{r_1+1}+
   \f{k_1-1}{2},\,\ldots,\qquad\qquad\qquad } \\   \textstyle{\qquad \qquad\qquad \qquad \ldots,\,-s_{r_1+r_2}-\f{k_{r_2}-1}{2},\,
    -s_{r_1+r_2}+\f{k_{r_2}-1}{2})}
\\ \text{and}\ \ \
    \d \ \ = \ \
    (\e_1, \, \e_2, \, \ldots,\e_{r_1} , \, k_1 \, ,0 \, ,k_2 \, ,0 \,
    , \, \ldots, \, k_{r_2}, \, 0)\,.
\endgathered
\end{equation}
  In order to compute the absolute value appearing in the formula for $L(s,\Pi_5)$ in (\ref{fivelpieces}), we make the assumption that  $k_j\ge k_\ell$ for $j \le
\ell$, which we may without loss of generality. The terms in the product on the right hand side of
(\ref{neededratio}) can be broken up into four groups, as follows: \vspace{.5cm}

\noindent for $1 \le i < k \le r_1$:
\begin{equation}\label{fety1}
\aligned
   G_{\d_i+\d_k}(s-\l_i-\l_k)  \ \ & = \  \  G_{\e_i+\e_k}(s+s_i+s_k) \\ & = \ \
i^{\e_{ik}}\f{\G_\R(s+s_i+s_k+\e_{ik})}{\G_\R(1-s-s_i-s_k+\e_{ik})}\,
;
\endgathered
\end{equation}
for $
  i \le r_1$,  $j \le r_2$:
\begin{equation}\label{fety2}
\gathered G_{\d_i+\d_{r_1+2j-1}}(s-\l_i-\l_{r_1+2j-1})
G_{\d_i+\d_{r_1+2j}}(s-\l_i-\l_{r_1+2j}) \\ = \ \
G_{\e_i+k_j}(s+s_i+s_{r_1+j}+\textstyle{\f{k_j-1}{2}} )
G_{\e_i}(s+s_i+s_{r_1+j}-\textstyle{\f{k_j-1}{2}} ) \\ = \ \
 i^{k_j} \f{\G_\C(s+s_i+s_{r_1+j}+\f{k_j-1}{2})}{\G_\C(1
    -s-s_i-s_{r_1+j}+\f{k_j-1}{2})} \, ;
\endgathered
\end{equation}
for $1 \le j < \ell \le r_2$:
\begin{equation}\label{fety3}
    \gathered
G_{\d_{r_1+2j-1}+\d_{r_1+2\ell-1}}(s-\l_{r_1+2j-1}-\l_{r_1+2\ell-1})
G_{\d_{r_1+2j}+  \d_{r_1+2\ell-1}}(s-\l_{r_1+2j}  -\l_{r_1+2\ell-1})
 \qquad\qquad\qquad \qquad\qquad\qquad\qquad\qquad\\
\times \,G_{\d_{r_1+2j-1}+\d_{r_1+2\ell}}
(s-\l_{r_1+2j-1}-\l_{r_1+2\ell}) G_{\d_{r_1+2j}+  \d_{r_1+2\ell}}
(s-\l_{r_1+2j}  -\l_{r_1+2\ell})
\qquad\qquad\qquad\qquad\qquad\qquad\qquad \\
= \ \
G_{k_j+k_\ell}(s+s_{r_1+j}+s_{r_1+\ell}+\textstyle{\f{k_j+k_\ell-2}{2}})
G_{k_\ell}(s+s_{r_1+j}+s_{r_1+\ell}-\textstyle{\f{k_j-k_\ell}{2}})
\qquad\qquad\qquad\qquad\qquad\qquad\qquad\qquad \\
\times \,
G_{k_j}(s+s_{r_1+j}+s_{r_1+\ell}+\textstyle{\f{k_j-k_\ell}{2}})
G_{0}(s+s_{r_1+j}+s_{r_1+\ell}-\textstyle{\f{k_j+k_\ell-2}{2}})
\qquad\qquad\qquad\qquad\qquad\qquad\qquad
\\ =
\f{i^{k_j}\,\G_\C(s+s_{r_1+j}+s_{r_1+\ell}+\f{k_j+k_\ell-2}{2})}{\G_\C(1-
s-s_{r_1+j}-s_{r_1+\ell}+|\f{k_j-k_\ell}{2}|)} \times
\f{i^{k_j}\,\G_\C(s+s_{r_1+j}+s_{r_1+\ell}+ |\f{k_j-k_\ell}{2}|)}{\G_\C(1-
s-s_{r_1+j}-s_{r_1+\ell}+\f{k_j+k_\ell-2}{2})} \, ; \qquad \qquad
\qquad \qquad \qquad \qquad \qquad \qquad
    \endgathered
\end{equation}
    and  for $ 1 \le j \le r_2 $:
\begin{equation}\label{fety4}
\aligned
   G_{\d_{r_1+2j-1}+\d_{r_1+2j}}(s-
    \l_{r_1+2j-1}-\l_{r_1+2j}) \ \ & = \ \ G_{k_j}(s+2s_{r_1+j})   \\
 & = \ \ i^{\e_j'}
\f{\G_\R(s+2s_{r_1+j}+\e_j')}{\G_\R(1-s-2s_{r_1+j}+\e_j')} \,.
\endaligned
\end{equation}
In simplifying (\ref{fety2}) and (\ref{fety3}) we have used the identity
\begin{equation}\label{gcancel}\gathered
    z_1-z_2 \  \in \  2\Z+\eta_1-\eta_2+1  \ \ \Longrightarrow
    \qquad\qquad\qquad\qquad\qquad\qquad\qquad\qquad\ \ \ \\
   \qquad\qquad\ \ \ G_{\eta_1}(s+z_1)G_{\eta_2}(s+z_2) \ \ =  \ \
    i^{z_1-z_2+1}\f{\G_\C(s+z_1)}{\G_\C(1-s-z_2)}\,,
    \endgathered
\end{equation}
which follows from the functional equation $\G(s)\G(1-s)=\pi\csc(\pi s)$ and the factorial property
$\G(s+1)=s\G(s)$.

The product over $1\le i < k \le r_1$ in (\ref{fety1}) equals $\f{L(s,\Pi_4)}{L(1-s,\widetilde{\Pi}_4)}\prod_{1\le
i < k \le r_1} i^{\e_{ik}}$.  Similarly, the product over $i\le r_1$ and $j \le r_2$ in (\ref{fety2}) is
$\f{L(s,\Pi_3)}{L(1-s,\widetilde{\Pi}_3)}\prod_{j\le r_2}i^{r_1k_j}$; the product over $1\le j < \ell \le r_2$ in
(\ref{fety3}) is $\f{L(s,\Pi_5)}{L(1-s,\widetilde{\Pi}_5)}\prod_{j< r_2}(-1)^{k_j(r_2-j)}$; and the product over $j
\le r_2$ in (\ref{fety4}) is $\f{L(s,\Pi_2)}{L(1-s,\widetilde{\Pi}_2)}\prod_{j\le r_2}i^{\e_j'}$.  Multiplying
these together proves (\ref{neededratio}), i.e., the untwisted case. In the twisted case, the expressions for
$L(s,\Pi_2)$ and $L(s,\Pi_4)$ must be replaced by (\ref{twotwistedpieces}) instead.  The proof remains the same,
except for changes in the overall multiplicative constant (always involving integral powers of $i$).
\end{proof}

  When $\pi$ is unramified at all
nonarchimedean places and  $\chi$ is trivial (corresponding to  untwisted, full level cusp forms), the global
completed $L$-function $\Lambda(s,\pi,Ext^2)=\prod_{p\le \infty}L_p(s,\pi,Ext^2)$ is fully determined by
(\ref{fivelpieces}) (for $p=\infty$), and by formula (\ref{extsqlocdef}) (for $p<\infty$). As a consequence of
(\ref{assumedfeagain}) and this proposition, we have the explicit functional equation
\begin{equation}\label{provedfe}
  \Lambda(s,\pi,Ext^2) \ \ = \ \ L_\infty(s,\pi,Ext^2)\,L(s,\pi,Ext^2) \  \ = \
  \ \omega  \, \Lambda(1-s,\tilde\pi,Ext^2)\,.
  \end{equation}
It is worth noting that our functional equation (\ref{assumedfeagain}) has a uniform description in all cases in
terms of the Casselman embedding, in marked contrast to the formulas for the $\Gamma$-factors  given in
(\ref{fivelpieces}).

\section{The full holomorphy of $L^S(s,\pi,Ext^2\otimes\chi)$  }
\label{sec:fullholomorphy}

 Finally, our last order of
business is to prove that   $L^S(s,\pi,Ext^2\otimes\chi)$  is fully holomorphic, that is holomorphic on
$\C-\{0,1\}$ with the possible exception of  poles at $s=0$ and $1$   which we show are at most simple.   That
behavior at $s=0$ and 1 was previously obtained by
 Jacquet and Shalika
\cite{jsextsq}; we have chosen to include this aspect as part of our argument as well because it requires little
additional overhead.  Recall   that theorem~\ref{mainthmwiths} was proved for $GL(m)$, $m$ odd, by Kim
\cite{kimgl4,kimextsq}.  This justifies the specialization to  $GL(2n)$  throughout the paper.

Let us write $T=S-\{\infty\}$ for the nonarchimedean places of $S$. It was shown in
(\ref{whatsentire}-\ref{recall75}) that the value of the pairing $P(\tau,E(s))$, or equivalently the global product
\begin{equation}\label{thisisentire}
    {\mathcal G}(s)\,\cdot\,\prod_{p\in T}\Psi_p(s,W_p,\Phi_p)\,\cdot\, L^T(s,\pi,Ext^2\otimes\chi)\,,
\end{equation}
is fully holomorphic  with at most simple poles  for any choice of local data $W_p$ and $\Phi_p$ at the finite set
of primes $T$ for which either $\pi$ or $\chi$ are ramified.

Dustin Belt \cite{belt} has recently shown that for any fixed value of $s\in \C$, there exists some choice of local
data $W_p$ and $\Phi_p$ such that $\Psi_p(s,W_p,\Phi_p)$ is a nonzero complex number.  In light of this, the
product over $p\in T$ in (\ref{thisisentire}) is irrelevant to its holomorphy:
\begin{equation}\label{thisisentire2}
     {\mathcal G}(s)\, L^T(s,\pi,Ext^2\otimes\chi)   \ \ \ \text{is fully holomorphic   with at most simple poles.}
\end{equation}  His result  applies equally to the archimedean analog of the integrals $\Psi_p(s,W_p,\Phi_p)$ which appear in Jacquet-Shalika's unfolded integral, which is given by formula (\ref{unfoldingprop3padic}) when $p=\infty$ (in this situation  $\Phi_\infty$ is a Schwartz function on $\R^n$).
Thus he also proves the full holomorphy with at most simple poles of both (\ref{thisisentire}) or
(\ref{thisisentire2}) without the factor ${\mathcal G}(s)$, in particular theorem~\ref{mainthmwiths} for any subset
$S$ that does not include $\infty$.  We shall thus assume from now on that $\infty\in S$. Belt's result reduces
theorem~\ref{mainthmwiths} to proving the full holomorphy with at most simple poles of
 \begin{equation}\label{thisisentire3}
    L_\infty(s,\pi,Ext^2\otimes\chi)\,\cdot\,\prod_{p\in T}\Psi_p(s,W_p,\Phi_p)\,\cdot\, L^T(s,\pi,Ext^2\otimes\chi)\,.
\end{equation}
This is a stronger condition than (\ref{thisisentire}), since it implies it by dividing by $\Gamma$-functions
(which never vanish).
 The full holomorphy  with at most simple poles   of (\ref{thisisentire3}) in the range  $\Re{s} \ge 1/2$ -- and hence of   $L^S(s,\pi,Ext^2\otimes\chi)$   in that range as well --
 is an immediate  consequence of  the following proposition.
  \begin{prop}\label{fullproveprop}Assume $\pi_\infty'$ has the embedding described in proposition~\ref{niceembedding}.  Then
\begin{equation}\label{fullprove1}
  L_\infty(s,\pi,Ext^2\otimes\chi)  \ \ \
     \text{is holomorphic and nonzero in~}\Re s \,\ge \,1
\end{equation}
and  the quotient
\begin{equation}\label{fullprove2}
\gathered \f{
    L_\infty(s,\pi,Ext^2\otimes\chi)}{{\mathcal G}(s)}~~~
    \text{is holomorphic in~}1/2\,\le\,\Re{s}\,\le\,1.
\endgathered
\end{equation}
\end{prop}

Before turning to the proof, let us first see how this implies the rest of  \thmref{mainthmwiths}, namely the full
holomorphy  with at most simple poles in the half plane $\Re{s}<1/2$.   The functional equation
(\ref{adelicpenufe}) relates the pairing for $\tau$ at $1-s$ the left hand side, to a  pairing for a translate of
$\widetilde \tau$ at $s$ on the right hand side.  The left hand side was computed in (\ref{whatsentire}) to be a
fourth root of unity times
\begin{equation}\label{westside1}
    {\mathcal G}(1-s)\,\cdot\,\prod_{p\in T}\Psi_p(1-s,W_p,\Phi_p)\,\cdot\,  L^T(1-s,\pi,Ext^2\otimes \pi)\,,
\end{equation}
whereas the right hand side is equal to a linear combination of expressions of the form
\begin{equation}\label{westside2}
\gathered
    N^{2ns-s-n}\,\cdot\,\prod_{j=1}^n G_{\d_{n+j}+\d_{n+1-j}+\eta}(s+\l_{n+j}+\l_{n+1-j})
   \,\cdot\, \widetilde{\mathcal G}(s)  \ \times \\ \times \ \prod_{p\in T}\Psi_p(s,\widetilde W_p,\widetilde \Phi_p)\,\cdot\,   L^T(s,\tilde\pi,Ext^2\otimes\chi\i)\,,
\endgathered
\end{equation}
where  $\widetilde{\mathcal G}(s)$ was defined in (\ref{halfcompletetilde}).
  We use the notation $\Psi_p(s,\widetilde W_p,\widetilde\Phi_p)$ to refer to the $p$-adic local integrals that arise in the unfolding (\ref{unfoldingprop}) for the left translate of $\widetilde \tau$.   Note that the set $S$ already contains the ramified places for this contragredient local data, and hence does not need to be enlarged to account for the translation  in (\ref{adelicpenufe}).

Observe that
\begin{equation*}\label{westside3}
\gathered
 \!\!\!\!\!\!\!\!\!\!\!\!\!\!\! \!\!\!\!\!\!\!\!\!\!\!\!\!\!\!\!\!\!\!\!
    \!\!\!\!\!\!\!\!\!\!\!\!\!\!\!\!\!\!\!\!\!\!\!\!\!\!\!\!
    \!\!\!\!\!  \prod_{j=1}^n G_{\d_{n+j}+\d_{n+1-j}+\eta}(s+\l_{n+j}+\l_{n+1-j}) \, \f{\widetilde{\mathcal G}(s)}{{\mathcal G}(1-s)} \ \ =  \\ \ \ \ \  \pm
    \prod_{1\le i < j \le
     2n} G_{\d_i+\d_j}(s+\l_i+\l_j) \ \
     = \ \ \pm  \prod_{1\le i < j \le
     2n} G_{\d_i+\d_j}(1-s-\l_i-\l_j)\i
\endgathered
\end{equation*}
(with signs that we do not need to determine) because of (\ref{eisen3b}).  Proposition~\ref{samefe} identifies this
product with a fourth root of unity times
$\f{L_\infty(s,\tilde\pi,Ext^2\otimes\chi\i)}{L_\infty(1-s,\pi,Ext^2\otimes\chi)}$. Returning to
(\rangeref{westside1}{westside2}), we deduce that
\begin{equation}\label{westside4}
    L_\infty(1-s,\pi,Ext^2\otimes \chi)\,\cdot\,\prod_{p\in T}\Psi_p(1-s,W_p,\Phi_p)\,\cdot\, L^T(1-s,\pi,Ext^2\otimes \chi)
\end{equation}
is equal to a linear combination of expressions of the form
\begin{equation}\label{westside5}
    N^{2ns-s-n}\,
    \cdot \, L_\infty(s,\tilde\pi,Ext^2\otimes\chi\i) \,\cdot\,\prod_{p\in T}\Psi_p(s,\widetilde W_p,\widetilde\Phi_p)\,\cdot\,   L^T(s,\tilde\pi,Ext^2\otimes\chi\i)\,,
\end{equation}
all of which are fully holomorphic  with at most simple poles  in $\Re{s} \ge 1/2$.  Hence (\ref{westside4}) is as
well, i.e., (\ref{thisisentire3}) is fully holomorphic   with at most simple poles   in $\Re s \le 1/2$, which we
have already seen is enough to imply \thmref{mainthmwiths}.

\begin{proof}
We write $\pi_\infty =I(P;\sgn^{\e_1}[s_1],\ldots,\sgn^{\e_{r_1}}[s_{r_1}],
    D_{k_1}[s_{r_1+1}],\ldots,D_{k_{r_2}}[s_{r_1+r_2}])$,
where  $\Re{s_1}\le \Re{s_2}\le \cdots \le \Re{s_{r_1}}$, $\Re{s_{r_1+1}}\le \Re{s_{r_1+2}}\le \cdots \le
\Re{s_{r_1+r_2}}$, and (\ref{realnegequal}) holds.  This ordering implies that
\begin{equation}\label{spos}
    \Re{s_i} \ < \ 0 \ \ \ \Longrightarrow \ \ \ 1 \le i \le r_1/2
    \ \
    \text{~~or~~} \ \  r_1+1 \le i \le r_1 + r_2/2\,.
\end{equation}
 Recall that each $k_i\ge 2$, and that the unitary dual estimate
 (\ref{genericunitarydual}b) states that
  each $|\Re{s_j}|<1/2$. As  $\G(s)$ is holomorphic and nonzero in $\Re{s}>0$,
  it follows that  each of the factors comprising $L_\infty(s,\pi,Ext^2\otimes\chi)$
 in (\ref{fivelpieces}-\ref{twotwistedpieces})
 is holomorphic and nonzero for  $\Re s \ge 1$.  This proves assertion
 (\ref{fullprove1}).

Our strategy for proving (\ref{fullprove2}) is as follows:~we will first identify the singularities of
$L_\infty(s,\pi,Ext^2\otimes\chi)$ in $\Omega = \{\Re{s}\ge 1/2\}$, and then show that ${\mathcal G}(s)$ has poles
of equal or greater order at those points. Let us now identify these poles from the factors in
(\ref{fivelpieces}-\ref{twotwistedpieces}). Poles occur for $L(s,\Pi_2\otimes\sgn^\eta)$ in $\Omega$   exactly when
\begin{equation}\label{polar1}
\gathered
s\,=\,-2s_{r_1+j}\ \text{~and~} \ k_j\equiv \eta\!\!\!\pmod 2,
\qquad\qquad\qquad\qquad\qquad\qquad\qquad \ \ \\ \qquad
\text{for some
     ~$j\le r_2/2$~ which has ~$1/4 \le -\Re
s_{r_1+j} <  1/2$.}
\endgathered
\end{equation}
Since $|\Re{s_i}|+|\Re{s_{r_1+j}}| < 1$ and $\f{k_j-1}{2} \ge 1/2$, $L(s,\Pi_3)$ is holomorphic in $\Omega$. Poles
for $L(s,\Pi_4\otimes\sgn^\eta)$ in $\Omega$ occur exactly when
\begin{equation}\label{polar2}
\gathered
s=-s_i-s_k\ \text{~and~} \ \e_i+\e_k\equiv \eta   \imod 2,
\qquad\qquad\qquad\qquad\qquad\qquad\\ \qquad
\text{for some
     ~$i<k\le r_1/2$~   and
~$1/2 \le -\Re{s_i}-\Re{s_k} <   1$.}
\endgathered
\end{equation}
  Finally, poles for $L(s,\Pi_5)$  in
$\Omega$ occur exactly when
\begin{equation}\label{polar3}
\gathered
s=-s_{r_1+j}-s_{r_1+\ell}\ \text{~and~} \ k_j=k_\ell\,,
\qquad\qquad  \ \
\qquad\qquad\qquad\qquad\qquad\qquad\\ \qquad
\text{for some
     ~$j<\ell \le r_2/2$~  and
~$1/2 \le -\Re{s_{r_1+j}}-\Re{s_{r_1+\ell}}< 1 $.}
\endgathered
\end{equation}  The list (\ref{polar1}-\ref{polar3})
describes the poles, with multiplicity, of $L_\infty(s,\pi,Ext^2\otimes\chi)$ in $\Omega$.

We will now show that each of these potential singularities in $\Omega = \{\Re{s}\ge 1/2\}$, including
multiplicity, is also a singularity of the product
\begin{equation}\label{lastestadded}
{\mathcal G}(s) \ \  = \ \ \prod_{\srel{1\le
i < j \le 2n}{i+j\le 2n}}G_{\d_i+\d_j+\eta}(s-\l_i-\l_j)\,.
\end{equation}
At this point we utilize the embedding described in proposition \ref{niceembedding}.
     We break up
the factors in this product into 6 different types, depending on which blocks $i$ and $j$ belong to. Recall that
$\l$ and $\d$, given in (\ref{lamniceembed}) and (\ref{deltaniceembed}), are arranged according to a partition of
$n$ of the form $(1,1,\ldots,1,2,2,\ldots,2,1,1\ldots,1)$, where there are $r_1/2$ ones, followed by $r_2$ twos,
and followed again by $r_1/2$ ones. The following are the six subsets that the indices $\{1\le i < j \le 2n | i+j
\le 2n\}$ are partitioned into.  The first subset, $\Sch_1$, consists of pairs $(i,j)$ corresponding to entries in
the first group of $r_1/2$ 1-blocks.  The second set, $\Sch_2$, corresponds to pairs $(i,j)$ where $i$ is in the
first set of 1-blocks, and $j$ is in the second set of 1-blocks. The third set $\Sch_3$ corresponds to pairs
$(i,j)$ where $i$ is in the first set of 1-blocks and $j$ is in the set of 2-blocks.  The remaining sets $\Sch_4$,
$\Sch_5$, and $\Sch_6$ correspond to pairs $(i,j)$ within the 2-blocks.  The set $\Sch_4$ corresponds to pairs
$(i,j)$ which are in the same 2-block.  The set $\Sch_5$ consists of pairs $(i,j)$ such that $i$ lies in the
$\ell$-th 2-block and $j$ lies in the $r_2+1-\ell$-th 2-block, for $\ell \le r_2/2$ (recall that there are $r_2$
2-blocks). Finally $\Sch_6$ corresponds to pairs $(i,j)$ which are in different 2-blocks, but not the ones in
$\Sch_5$.

Let ${\mathcal G}_k(s) = \prod_{(i,j)\in \Sch_k} G_{\d_i+\d_j+\eta}(s-\l_i-\l_j)$. We will now show that each of
the singularities from (\ref{polar1}-\ref{polar3}) is also a singularity, of the same order, of at least one of the
${\mathcal G}_i(s)$, and is not a zero of any ${\mathcal G}_i(s)$.  This will finish the proof of
(\ref{fullprove2}).  The reason we need to group the factors into ${\mathcal G}_1(s),\ldots,{\mathcal G}_6(s)$ is
that there may be some cancelation within the factors which comprise some of these partial products.  Let us now
identify what they are.  We have
\begin{align}\label{s1p}
{\mathcal G}_1(s) \  \ & = \ \     {\prod}_{1\le i < j \le r_1/2}
G_{\e_i+\e_j+\eta}(s+s_i+s_j)\,,\\
\label{s2p}
   {\mathcal G}_2(s) \ \ & = \ \  {\prod}_{1\le i < j \le r_1/2}
    G_{\e_i+\e_j+\eta}(s+s_i+s_{r_1+1-j}) \,,  \ \ \ \  \text{and}\\
    \label{s3p}
{\mathcal G}_3(s) \ \ & = \ \    {\prod}_{\stackrel{\scriptstyle{i\le
r_1/2}}{j\le r_2}}\
    i^{k_j}\f{\G_\C(s+s_i+s_{r_1+j}+\f{k_j-1}{2})}{\G_\C(1-
    s-s_i-s_{r_1+j}+\f{k_j-1}{2})}
\end{align}
(see (\ref{fety2})).  The pairs in $\Sch_4$ are indices $(i,j)$ of the form $(r_1/2+2\ell-1,r_1/2+2\ell)$, where
$\ell$ ranges from 1 to $r_2/2$.  For such an index $(i,j)$, we have that $\d_i=k_\ell$, $\d_j=0$,
$\l_i=-s_{r_1+\ell}-\f{k_\ell-1}{2}$, and $\l_j =-s_{r_1+\ell}+\f{k_\ell-1}{2}$.  Therefore ${\mathcal G}_4(s)$
equals
\begin{equation}\label{s4p}
{\mathcal G}_4(s) \ \ = \ \
{\prod}_{\ell\,=\,1}^{r_2/2}G_{k_\ell+\eta}(s+2 s_{r_1+\ell})\,.
\end{equation}
The pairs in $\Sch_5$ are similarly parameterized by pairs  $(i,j)$ of the form
$(r_1/2+2\ell-1,r_1/2+2(r_2+1-\ell)-1)$, with $\ell \le r_2/2$. These lie in the $\ell$-th and $(r_2+1-\ell)$-th
blocks, respectively, and satisfy $i+j=n$; therefore they are the only indices in their blocks.  We find that
$\d_i=k_\ell$, $\d_j=k_{r_2+1-\ell}=k_\ell$, $\l_i=-s_{r_1+\ell} -\f{k_\ell-1}{2}$, and
$\l_j=-s_{r_1+r_2+1-\ell}-\f{k_{r_2+1-\ell}\,-\,1}{2}$, and that
\begin{equation}\label{s5p}
\aligned
    {\mathcal G}_5(s) \ \ = \ \ {\prod}_{\ell\,=\,1}^{r_2/2}
    G_{\eta}(s+s_{r_1+\ell}+s_{r_1+r_2+1-\ell}+k_\ell-1)
    \, .
\endaligned
\end{equation}
Finally, $\Sch_6$ consists of the indices in pairs of 2-blocks which have not been accounted for.  These are the
$\ell_1$-th and $\ell_2$-th of the $2$-blocks, where $\ell_1+\ell_2 \le r_2$ and $\ell_1<\ell_2$. Those conditions
are necessary for the indices to not be in $\Sch_4$ or $\Sch_5$, and their sum to  be $\le n$.  The product of
$G_{\d_i+\d_j+\eta}(s-\l_i-\l_j)$ over pairs in these blocks is given in (\ref{fety3}), so ${\mathcal G}_6(s)$
equals
\begin{equation}\label{s6p}
\gathered
    {\prod}_{\stackrel{\scriptstyle{\ell_1<\ell_2}}{\ell_1+\ell_2\le
    r_2}}  \( (-1)^{k_{\ell_1}}\,
     \f{\G_\C(s+s_{r_1+\ell_1}+s_{r_1+\ell_2}
    +\f{k_{\ell_1} + k_{\ell_2}-2}{2})
}{\G_\C(1-s-
    s_{r_1+\ell_1}-s_{r_1+\ell_2}
    +\f{k_{\ell_1} + k_{\ell_2}-2}{2})} \ \times \qquad \right.
    \\ \left. \qquad\qquad\qquad\qquad
    \f{
    \G_\C(s+s_{r_1+\ell_1}+s_{r_1+\ell_2}+|\f{k_{\ell_1}-k_{\ell_2}}{2}|)}
    {
    \G_\C(1-s-s_{r_1+\ell_1}-s_{r_1+\ell_2}+|\f{k_{\ell_1}-k_{\ell_2}}{2}|)}\ \).
    \endgathered
\end{equation}
We now claim that
\begin{equation}\label{zeroclaim}
    \text{each of~}{\mathcal G}_1(s), \, {\mathcal G}_2(s), \,
    \ldots, \, {\mathcal G}_6(s) \  \text{is nonzero in~}1/2
    \,\le \, \Re{s} \,<\,1\,.
\end{equation}
We first recall that $G_0(z)$ is zero only for odd positive integers, and $G_1(z)$ is zero only for even positive
integers. The shifts in (\ref{s1p}), (\ref{s2p}), and (\ref{s4p}) all have real parts between $-1$ and $0$, so the
arguments of the $G$-functions in these products all have real part less than $1$ when $\Re{s}<1$.  This proves
(\ref{zeroclaim}) for these three products.  Since $\Re{s_{r_1+\ell}}=-\Re{s_{r_1+r_2+1-\ell}}$, the shift  in
(\ref{s5p}) has real part $k_\ell-1$. Therefore the arguments of the factors in (\ref{s5p}) are never integral when
$1/2\le \Re{s}<1$, proving (\ref{zeroclaim}) for $\mathcal G_5(s)$. The argument in the denominator of (\ref{s3p})
is never in $\Z_{\le 0}$ when $1/2\le \Re{s}<1$, because $\Re(-s_i-s_{r_1+j}+\f{k_j-1}{2})\ge \Re{-s_i} -\f 12 + \f
12 \ge 0$ when $i\le r_1/2$.  Similarly, $\Re(-s_{r_1+\ell_1}-s_{r_1+\ell_2})\ge 0$ in (\ref{s6p}), and the
arguments in its denominators are never in $\Z_{\le 0}$ for $s$ in this range.  That means the Gamma functions in
denominators of (\ref{s3p}) and (\ref{s6p}) do not have poles for $1/2\le \Re{s}<1$, and since $\G(s)$ is never
zero, ${\mathcal G}_3(s)$ and ${\mathcal G}_6(s)$ satisfy the claim in (\ref{zeroclaim}).

To finish, we will check that each of the poles in (\ref{polar1}-\ref{polar3}), with multiplicity, occurs in one of
(\ref{s1p}-\ref{s6p}); we have just seen that they are zeroes of none of them.  One sees readily that the poles
listed in (\ref{polar1}), (\ref{polar2}), and (\ref{polar3}), respectively, are found in (\ref{s4p}), (\ref{s1p}),
and (\ref{s6p}), respectively.  Since we had just checked in proving (\ref{zeroclaim}) that these poles are not
canceled by  zeros of any factors in those products, we have finished the proof of (\ref{fullprove2}) and hence
proposition~\ref{fullproveprop}.
\end{proof}

%
%
%
%


\begin{bibsection}


\begin{biblist}


\bib{belt}{thesis}{author={Belt, Dustin},title={On Local Exterior-Square L-functions},type={phd},school={Purdue
University},year={2011},note={\url{http://arxiv.org/abs/1108.2200}}}


\bib{bumpfriedberg}{article}{
    author={Bump, Daniel},
    author={Friedberg, Solomon},
     title={The exterior square automorphic $L$-functions on ${\rm GL}(n)$},
 booktitle={Festschrift in honor of I. I. Piatetski-Shapiro, Part II},
    series={Israel Math. Conf. Proc.},
    volume={3},
     pages={47\ndash 65},
 publisher={Weizmann},
     place={Jerusalem},
      date={1990},
}


\bib{Casselman:1980}{article}{
     author={Casselman, W.},
      title={Jacquet modules for real reductive groups},
  booktitle={Proceedings of the International Congress of Mathematicians (Helsinki, 1978)},
      pages={557\ndash 563},
  publisher={Acad. Sci. Fennica},
      place={Helsinki},
       date={1980},
}


\bib{Casselman:1989}{article}{
     author={Casselman, W.},
      title={Canonical extensions of Harish-Chandra modules to representations of $G$},
    journal={Canad. J. Math.},
     volume={41},
       date={1989},
     number={3},
      pages={385\ndash 438},
}


\bib{chm}{article}{
    author={Casselman, William},
    author={Hecht, Henryk},
    author={Mili{\v{c}}i{\'c}, Dragan},
     title={Bruhat filtrations and Whittaker vectors for real groups},
 booktitle={The mathematical legacy of Harish-Chandra 
 },
    series={Proc. Sympos. Pure Math.},
    volume={68},
     pages={151\ndash 190},
 publisher={Amer. Math. Soc.},
     place={Providence, RI},
      date={2000},
}


\bib{DixMal}{article}{
   author={Dixmier, Jacques},
   author={Malliavin, Paul},
   title={Factorisations de fonctions et de vecteurs ind\'efiniment
   diff\'erentiables},
   journal={Bull. Sci. Math. (2)},
   volume={102},
   date={1978},
   number={4},
   pages={307--330},
   issn={0007-4497},
}


\bib{fz}{article}{
    author={Fomin, Sergey},
    author={Zelevinsky, Andrei},
     title={Double Bruhat cells and total positivity},
   journal={J. Amer. Math. Soc.},
    volume={12},
      date={1999},
    number={2},
     pages={335\ndash 380},
      issn={0894-0347},
}



\bib{goja}{book}{
    author={Godement, Roger},
    author={Jacquet, Herv{\'e}},
     title={Zeta functions of simple algebras},
 publisher={Springer-Verlag},
     place={Berlin},
      date={1972},
     pages={ix+188},
}

\bib{goodmanwallach}{article}{
   author={Goodman, Roe},
   author={Wallach, Nolan R.},
   title={Whittaker vectors and conical vectors},
   journal={J. Funct. Anal.},
   volume={39},
   date={1980},
   number={2},
   pages={199--279},
   issn={0022-1236},
} 	


\bib{harristaylor}{book}{
    author={Harris, Michael},
    author={Taylor, Richard},
     title={The geometry and cohomology of some simple Shimura varieties},
    series={Annals of Mathematics Studies},
    volume={151},
 publisher={Princeton University Press},
      date={2001},
     pages={viii+276},
      isbn={0-691-09090-4},
} \bib{henniart}{article}{
    author={Henniart, Guy},
     title={Une preuve simple des conjectures de Langlands pour ${\rm
            GL}(n)$ sur un corps $p$-adique},
  language={French, with English summary},
   journal={Invent. Math.},
    volume={139},
      date={2000},
    number={2},
     pages={439\ndash 455},
      issn={0020-9910},
}



\bib{jseuler}{article}{
   author={Jacquet, H.},
   author={Shalika, J. A.},
   title={On Euler products and the classification of automorphic
   representations. I},
   journal={Amer. J. Math.},
   volume={103},
   date={1981},
   number={3},
   pages={499--558},
   issn={0002-9327},
}

\bib{jsextsq}{article}{
    author={Jacquet, Herv{\'e}},
    author={Shalika, Joseph},
     title={Exterior square $L$-functions},
 booktitle={Automorphic forms, Shimura varieties, and $L$-functions, Vol.\
            II (Ann Arbor, MI, 1988)},
    series={Perspect. Math.},
    volume={11},
     pages={143\ndash 226},
 publisher={Academic Press},
     place={Boston, MA},
      date={1990},
    }


\bib{jpss}{article}{
     author={Jacquet, Herv{\'e}},
     author={Piatetski-Shapiro, Ilja Iosifovitch},
     author={Shalika, Joseph},
      title={Automorphic forms on ${\rm GL}(3)$},
    journal={Ann. of Math. (2)},
     volume={109},
       date={1979},
     number={1~and~ 2},
      pages={169\ndash 258},
}



\bib{jiang}{article}{
   author={Jiang, Dihua},
   title={On the fundamental automorphic $L$-functions of ${\rm SO}(2n+1)$},
   journal={Int. Math. Res. Not.},
   date={2006},
   issn={1073-7928},
}

\bib{kimextsq}{article}{
    author={Kim, Henry H.},
     title={Langlands-Shahidi method and poles of automorphic $L$-functions:
            application to exterior square $L$-functions},
   journal={Canad. J. Math.},
    volume={51},
      date={1999},
    number={4},
     pages={835\ndash 849},
      issn={0008-414X},
}


\bib{kimgl4}{article}{
    author={Kim, Henry H.},
     title={Functoriality for the exterior square of ${\rm GL}\sb 4$ and the
            symmetric fourth of ${\rm GL}\sb 2$},
   journal={J. Amer. Math. Soc.},
    volume={16},
      date={2003},
    number={1},
     pages={139\ndash 183 (electronic)},
      issn={0894-0347},
}



\bib{langlandsdc}{article}{
    author={Langlands, Robert P.},
     title={Problems in the theory of automorphic forms},
 booktitle={Lectures in modern analysis and applications, III},
     pages={18\ndash 61. Lecture Notes in Math., Vol. 170},
 publisher={Springer},
     place={Berlin},
      date={1970},
}


\bib{eulerproducts}{book}{
    author={Langlands, Robert P.},
     title={Euler products},
 publisher={Yale University Press},
     place={New Haven, Conn.},
      date={1971},
     pages={v+53},
}


\bib{langlandsrealgroups}{article}{
    author={Langlands, Robert P.},
     title={On the classification of irreducible representations of real
            algebraic groups},
 booktitle={Representation theory and harmonic analysis on semisimple Lie
            groups},
    series={Math. Surveys Monogr.},
    volume={31},
     pages={101\ndash 170},
 publisher={Amer. Math. Soc.},
     place={Providence, RI},
      date={1989},
}


\bib{inforder}{article}{
        author={Miller, Stephen D.},
        author={Schmid, Wilfried},
        title={Distributions and analytic continuation of Dirichlet series},
    journal={J. Funct. Anal.},
        volume={214},
        date={2004},
        number={1},
        pages={155\ndash 220},
        issn={0022-1236},
 }

\bib{voronoi}{article}{
   author={Miller, Stephen D.},
   author={Schmid, Wilfried},
   title={Automorphic distributions, $L$-functions, and Voronoi summation
   for ${\rm GL}(3)$},
   journal={Ann. of Math. (2)},
   volume={164},
   date={2006},
   number={2},
   pages={423--488},
}


\bib{korea}{article}{
   author={Miller, Stephen D.},
   author={Schmid, Wilfried},
   title={The Rankin-Selberg method for automorphic distributions},
   conference={
      title={Representation theory and automorphic forms},
   },
   book={
      series={Progr. Math.},
      volume={255},
      publisher={Birkh\"auser Boston},
      place={Boston, MA},
   },
   date={2008},
   pages={111--150}
}


\bib{pairingpaper}{article}{author={Miller, Stephen D.}, author={Schmid, Wilfried},
        title={Pairings of automorphic distributions}, note={To appear in {\em Mathematische Annalen}, \url{http://arxiv.org/abs/1106.2364} }}


\bib{mirabolic}{article}{author={Miller, Stephen D.}, author={Schmid, Wilfried},
        title={Adelization of Automorphic Distributions and Mirabolic Eisenstein Series}, note={To appear in {\em Contemporary Mathematics}, volume in honor of Gregg Zuckerman's 60th birthday, Jeff Adams, Bong Lian, and Siddhartha Sahi, editors, \url{http://arxiv.org/abs/1106.2583}}}



\bib{rapiddecay}{article}{author={Miller, Stephen D.}, author={Schmid, Wilfried},
        title={On the rapid decay of cuspidal automorphic forms}, note={\url{http://arxiv.org/abs/1106.2149}}}




\bib{flato}{article}{
    author={Schmid, Wilfried},
     title={Automorphic distributions for ${\rm SL}(2,\Bbb R)$},
 booktitle={Conf\'erence Mosh\'e Flato 1999, Vol. I (Dijon)},
    series={Math. Phys. Stud.},
    volume={21},
     pages={345\ndash 387},
 publisher={Kluwer Acad. Publ.},
     place={Dordrecht},
      date={2000},
}



\bib{shahidi}{article}{
    author={Shahidi, Freydoon},
     title={A proof of Langlands' conjecture on Plancherel measures;
            complementary series for $p$-adic groups},
   journal={Ann. of Math. (2)},
    volume={132},
      date={1990},
    number={2},
     pages={273\ndash 330},
      issn={0003-486X},
}


\bib{stade}{article}{
    author={Stade, Eric},
     title={Mellin transforms of ${\rm GL}(n,\Bbb R)$ Whittaker functions},
   journal={Amer. J. Math.},
    volume={123},
      date={2001},
    number={1},
     pages={121\ndash 161},
      issn={0002-9327},
}

\bib{tadic}{article}{
   author={Tadi{\'c}, Marko},
   title={$\widehat{\rm GL}(n,\C)$ and $\widehat{\rm GL}(n,\R)$},
   conference={
      title={Automorphic forms and $L$-functions II. Local aspects},
   },
   book={
      series={Contemp. Math.},
      volume={489},
      publisher={Amer. Math. Soc.},
      place={Providence, RI},
   },
   date={2009},
   pages={285--313},
}


\bib{voganlarge}{article}{
   author={Vogan, David A., Jr.},
   title={Gel\cprime fand-Kirillov dimension for Harish-Chandra modules},
   journal={Invent. Math.},
   volume={48},
   date={1978},
   number={1},
   pages={75--98},
   issn={0020-9910},
}



\bib{Wallach:1983}{article}{
    author={Wallach, Nolan R.},
     title={Asymptotic expansions of generalized matrix entries of representations of real reductive groups},
 booktitle={Lie group representations, I 
 },
    series={Lecture Notes in Math.},
    volume={1024},
     pages={287\ndash 369},
 publisher={Springer},
     place={Berlin},
      date={1983},
}

\bib{Zhelobenko}{book}{
   author={{\v{Z}}elobenko, D. P.},
   title={Compact Lie groups and their representations},
   note={Translated from the Russian by Israel Program for Scientific
   Translations;
   Translations of Mathematical Monographs, Vol. 40},
   publisher={American Mathematical Society},
   place={Providence, R.I.},
   date={1973},
   pages={viii+448},
   review={\MR{0473098 (57 \#12776b)}},
}
		




\end{biblist}
\end{bibsection}
\end{document}